\PassOptionsToPackage{numbers, compress}{natbib}
\documentclass{article}

\usepackage[preprint]{neurips_2024}

\usepackage[utf8]{inputenc} %
\usepackage[T1]{fontenc}    %
\usepackage[linktocpage,colorlinks,linkcolor=blue,anchorcolor=blue,citecolor=blue,urlcolor=blue,pagebackref]{hyperref}
\hypersetup{}
\renewcommand*{\backref}[1]{\ifx#1\relax \else Page #1 \fi}
\renewcommand*{\backrefalt}[4]{%
  \ifcase #1 \footnotesize{(Not cited.)}%
  \or        \footnotesize{(Cited on page~#2.)}%
  \else      \footnotesize{(Cited on pages~#2.)}%
  \fi
}%
\usepackage{url}            %
\usepackage{booktabs}       %
\usepackage{amsfonts}       %
\usepackage{nicefrac}       %
\usepackage{microtype}      %
\usepackage{xcolor}         %

\usepackage{caption}
\usepackage{subcaption}
\usepackage{algorithm}
\usepackage{algorithmic}

\usepackage{amsmath, amsfonts, bm, amsthm, amssymb}
\allowdisplaybreaks %
\usepackage{afterpage, pdflscape, capt-of, array}
\usepackage{graphicx}
\usepackage{tikz}
\usepackage{ dsfont }
\usepackage{enumitem}

\usepackage{xspace}
\usepackage{multirow}
\usepackage{cleveref}
\usepackage{tikz}

\newcommand{\bluecheck}{}%
\DeclareRobustCommand{\bluecheck}{%
  \tikz\fill[scale=0.4, color=blue]
  (0,.35) -- (.25,0) -- (1,.7) -- (.25,.15) -- cycle;%
}

\usepackage{macros}

\newcommand{\algo}{Hyper Policy Gradient Descent}
\newcommand{\salgo}{HPGD}
\newcommand{\prob}{Contextual Bilevel Reinforcement Learning}
\newcommand{\sprob}{CB-RL}

\DeclareMathOperator{\softmax}{softmax}

\title{Contextual Bilevel Reinforcement Learning for Incentive Alignment}

\author{%
    Vinzenz Thoma \footnotemark[1] \\
    ETH AI Center\\
    \texttt{vinzenz.thoma@ai.ethz.ch} \\
    \And
    Barna Pasztor \thanks{Equal Contribution} \\
    ETH AI Center\\
    \texttt{barna.pasztor@ai.ethz.ch} \\
    \And
    Andreas Krause \\
    ETH Zurich\\
    \texttt{krausea@ethz.ch} \\
    \And
    Giorgia Ramponi \\
    University of Zurich
    \\
    \texttt{giorgia.ramponi@uzh.ch} \\
    \And
    Yifan Hu \\
    EPFL \& ETH Zurich
    \\
    \texttt{yifan.hu@epfl.ch} \\
}

\begin{document}

\maketitle

\begin{abstract}
The optimal policy in various real-world strategic decision-making problems depends both on the environmental configuration and exogenous events. %
For these settings, we introduce {\em \prob\ } (\sprob), a stochastic bilevel decision-making model, where the lower level consists of solving a contextual Markov Decision Process (CMDP).
\sprob\ can be viewed as a Stackelberg Game where the leader and a random context beyond the leader's control together decide the setup of many MDPs that potentially multiple followers best respond to.
This framework extends beyond traditional bilevel optimization and finds relevance in diverse fields such as RLHF, tax design, reward shaping, contract theory and mechanism design.
We propose a stochastic {\em \algo\ (\salgo)} algorithm to solve \sprob, and demonstrate its convergence.
Notably, \salgo\ uses stochastic hypergradient estimates, based on observations of the followers' trajectories. Therefore, it allows followers to use any training procedure and the leader to be agnostic of the specific algorithm, which aligns with various real-world scenarios. 
We further consider the setting when the leader can influence the training of followers and propose an accelerated algorithm. 
We empirically demonstrate the performance of our algorithm for reward shaping and tax design.\looseness=-1
\end{abstract}

\section{Introduction}
\label{sec:introduction}
In Reinforcement Learning (RL), Markov Decision Processes (MDPs) \citep{puterman2014markov} provide a versatile framework for capturing sequential decision-making problems across various domains such as  health care \citep{yu2021reinforcement}, energy systems \citep{perera2021applications}, economics \citep{charpentier2021reinforcement}, and finance \citep{hambly2023recent}. A considerable amount of work has been devoted to solving standard MDPs~\citep{arulkumaran2017deep, sutton2018reinforcement, wang2022deep}. However, in many applications, MDPs can be configured on purpose or affected by exogenous events, both of which can significantly impact the corresponding optimal policies. For example, in a simplified economic framework, the optimal decision of an individual household depends both on public policies and economic uncertainties \citep{Curry2023Learning, hill2021solving}. The policy maker in turn has to make decisions, anticipating the best response of differently-minded individual agents to its policies and  exogenous events outside the policy maker's control.

\looseness=-1
To study such problems, we introduce \prob~(\sprob), a hierarchical decision-making framework where followers solve a contextual Markov decision processes (CMDP)~\cite{hallak2015contextual}, configured by the leader. \sprob\ is formalized as:
\begin{equation}
    \label{problem:original}
     \begin{aligned}
    \min_{x} \quad & F(x):=\EE_{\xi} [f(x,\pi^*_{x,\xi}, \xi)] & \text{(leader, upper-level)} \\
    \mathrm{where} \quad  & \pi^*_{x,\xi} = \argmax_{\pi} J_{\lambda,x,\xi}(\pi) & \text{(follower, lower-level)},
    \end{aligned}
\end{equation}
where $x$ represents the model configuration of the CMDP chosen by the leader, $\xi$ represents the context that followers encounter, and the function $J_{\lambda,x,\xi}$ denotes the (entropy-regularized) reward of the CMDP for a policy $\pi$, a given model parameters $x$, a context $\xi$, and a regularization parameter $\lambda$.  %

In our framework, the leader chooses $x$ to configure various aspects of the CMDP, such as state transitions, the initial state distribution, and the followers' reward functions. Modeling the problem as a CMDP instead of a standard MDP is essential when modeling situations where the environment is influenced by side information or personal perferences. For example, the context $\xi$ can capture a wide range of real-world scenarios such as: (1) there is one follower, who responds optimally not only to the leader's chosen $x$ but also to a side information $\xi$, such as weather or season, (2) there are multiple followers, each aiming to maximize their own utility, in which case $\xi$ represents different possible follower preferences, and (3) there are multiple followers, each facing an uncertain contextual variable, i.e., $\xi = (i,\eta)$ represents the $i$-th follower encountering a specific context $\eta \sim \mathbb{P}_\eta$.

The proposed framework extends the concept of contextual bilevel optimization~\citep{hu2024contextual}, where the follower engages in static contextual stochastic optimization rather than sequential decision-making.
It also expands upon traditional MDP model design~\cite{zhang2018learning, chen2022adaptive} and configurable MDPs~\cite{metelli2018configurable, ramponi2021learning}, where typically only one follower attempts to solve an MDP, as opposed to CMDPs.
Beyond these fields, our framework finds various applications in Principal-Agent problems \citep{ben2023principal}, RLHF \citep{chakraborty2023parl,shen2024principled}, dynamic Stackelberg games~\cite{Gerstgrasser2023oracles, Wang2023Differentiable}, Security Games~\citep{sinha2018stackelberg, letchford2013optimal}, 
dynamic mechanism design~\cite{Curry2024Automated}, contract theory \cite{Ivanov2024Principal,Wu2024Contractual}, and tax design~\citep{Curry2023Learning,zheng2022ai,hill2021solving}. See \Cref{sec:applications} for the concrete \sprob~formulations of these applications.

Despite its wide applicability, to the best of our knowledge, there are no algorithms specifically designed for \sprob. The closest is an algorithm from model design for MDPs~\citep{chen2022adaptive}, which can be adapted to our setting after some modifications. However, \citep{chen2022adaptive} requires the follower to solve the MDP deterministically using soft value iteration. Moreover, the full hypergradient is computed in each iteration, as part of which the leader requires access to the lower-level computations. Both of these aspects significantly restrict how well the method can scale to larger settings.  

Instead, in this work, we propose a stochastic {\em \algo\ (\salgo)} algorithm for the leader that solely relies on trajectory data from followers. The followers can use a variety of possibly stochastic learning algorithms to find an approximately optimal policy for the lower-level CMDPs. The leader in turn is agnostic of the exact algorithm used, as the hypergradient is estimated using only trajectory samples generated from the follower policy. The fact that both the lower-level and the hypergradient computation are stochastic makes \salgo\ salable to large problem settings.

We show non-asymptotic convergence of \salgo\ to a stationary point and validate these findings through experimental evidence.
In scenarios where followers grant the leader full control over their training procedure, as posited in prior work~\citep{chen2022adaptive}, we present an accelerated \salgo\ algorithm,  designed to minimize the number of lower-level iterations. \looseness=-1%

\begin{table}[t]
    {
    \footnotesize
    \caption{Summary of Related Works in Bilevel Reinforcement Learning.}
    \renewcommand\arraystretch{1.5}
    \begin{center}
     \makebox[\textwidth]{
    \begin{tabular}{c|ccccc|cc|ccc}
        \toprule[1.5pt]
        {}
        &\multicolumn{2}{c}{Context}
        & \multirow{2}{*}{$r_x$}
        & \multirow{2}{*}{$P_x$}
        & \multirow{2}{*}{$\mu_x$}
        & Agnostic
        & Deter
        & Upper
        & \multicolumn{2}{c}{Lower} 
        \\

        {}
        & Multi
        & Side Info
        &
        & 
        &
        & Control
        & Stoch
        & Iterations
        & Iterations
        & Method
        \\
        \hline

        {\citep{chen2022adaptive}}
        & 
        & 
        & \bluecheck
        & \bluecheck
        & 
        & Control
        & Deter
        & $\cO(\delta^{-2})$*
        & $\cO(\log(\delta^{-1}))$
        & Soft-VI
        \\
        \hline

        {\citep{chakraborty2023parl}}
        & 
        & 
        & \bluecheck
        & 
        & 
        & Agnostic
        & Deter
        & $\cO(\delta^{-2})$
        & $\cO(\log(\delta^{-1}))$
        & PG
        \\
        \hline

        {\citep{shen2024principled}}
        & 
        & 
        & \bluecheck
        & \bluecheck
        & 
        & Agnostic
        & Deter
        & $\cO(\delta^{-2})$
        & $\cO(\log(\delta^{-1}))$
        & PMG
        \\
        \hline
     
        {\citep{yang2024bilevel}}
        & 
        &
        &{\bluecheck}
        &
        &
              & Agnostic
        & Deter
        &$\cO(\delta^{-2})$
        & $\cO(\log(\delta^{-1}))$
        & PMG
        \\

       \hline
        \hline

        \multirow{3}{*}{\salgo}
        & \multirow{3}{*}{\bluecheck}
        & \multirow{3}{*}{\bluecheck}
        & \multirow{3}{*}{\bluecheck}
        & \multirow{3}{*}{\bluecheck}
        & \multirow{3}{*}{\bluecheck}
        & \multirow{3}{*}{Agnostic}
        & \multirow{3}{*}{Stoch}
        & \multirow{3}{*}{$\cO(\delta^{-4})$}
        & $\cO(\log(\delta^{-1}))$ 
        & Soft-VI
        \\

        {}
        &
        &
        &
        &
        &
        &
        &
        &
        &$\cO(\log(\delta^{-1}))$ 
        & NPG
        \\

        {}
        &
        &
        &
        &
        &
        &
        &
        &
        &$ \Tilde\cO(\delta^{-2})$
        & Soft-Q
        \\
        \hline
        
        {\salgo}
        & \bluecheck
        & \bluecheck
        & \bluecheck
        & \bluecheck
        & \bluecheck
        & Control
        & Stoch
        & $\cO(\delta^{-4})$
        & $\cO(\log(\delta^{-1}))$ 
        & RT-Q
        \\

\hline
\hline

         \multicolumn{11}{p{\textwidth}}{
         \vspace{-0.3cm} %
            \parbox{\textwidth}{Multi: Multiple followers. Side Info: Side information. Context: CMDP instead of MDP. $r_x$, $P_x$, and $\mu_x$ denote the dependence of rewards, transitions, and initial state distribution on $x$. Agnostic vs. Control: Whether leader can influence the training of the follower(s). Deter vs. Stoch: Requiring full knowledge of hypergradient or estimating it from samples. Complexity is based on $\|\nabla F(x)\|^2 \leq \delta^2$ instead of $\|\nabla F(x)\|^2 \leq \delta$. *\citep{chen2022adaptive} assumes convexity of $F$ in $x$ and considers $F(x) - \min F(x) \leq \delta$. VI: Value Iteration. PMG: Policy Mirror Gradient. PI: Policy Iteration. NPG: Natural Policy Gradient. Q: Q-learning. RT-Q: Randomly Truncated Soft Q-learning.}
         } \\
         \bottomrule[1.5pt]
    \end{tabular} 
    }
    \label{tab:summary}
    \end{center}
    }
    \end{table}

\noindent \textbf{Our Contributions}
\begin{itemize}[leftmargin=2em]
    \item We introduce \emph{\prob\ (\sprob)} that captures a wide range of important applications (Sec. \ref{sec:problem}). It is the first bilevel reinforcement learning framework that through the context $\xi$ allows multiple followers and side information. %
    We summarize the key differences to the previous literature in \Cref{tab:summary}.
    \item We propose a stochastic \emph{\algo\ (\salgo)} algorithm that performs stochastic gradient descent on the upper-level objective (Sec. \ref{sec:oracle}). Importantly, we are the first to estimate the hypergradient from lower-level trajectory samples instead of computing it exactly, while further providing convergence guarantees. Furthermore, our approach \emph{is agnostic of the learning dynamics of the agent}, enabling followers to utilize a wide range of algorithms to solve the lower-level CMDPs. We only assume the leader can sample lower-level trajectories from an inexact oracle. For several widely-used RL algorithms, we explicitly show how to use them to build the inexact oracle needed by \salgo. Noteably, we are the first to consider stochastic lower-level learning algorithms, such as soft Q-learning.
    \item We establish the non-asymptotic convergence rate of \salgo~to a stationary point of the overall objective---despite the nonconvex lower-level problem (Sec. \ref{sec:oracle}).  When assuming full hypergradient information, i.e., deterministic updates, the outer iteration complexity of \salgo\  reduces to $\cO(\delta^{-2})$, recovering previous results. Moreover, we discuss how to estimate the hypergradient if the upper-level loss function admits a specific form of cumulative costs (Sec. \ref{sec:linear_loss_function}).
    \item When the leader is allowed to control the follwers' learning dynamics (Sec \ref{sec:full_access}), we propose a stochastic accelerated algorithm denoted as \salgo\ RT-Q (Alg \ref{alg:hpgdrtq}). It greatly reduces the number of lower-level soft Q-learning iterations from $\tilde\cO(\delta^{-2})$ to $\cO(\log(\delta^{-1}))$, such that we recover the rate for deterministic lower-level updates. For this result we leverage several techniques, such as mini-batches, multilevel Monte Carlo \cite{Giles2015Multilevel,hu2024contextual}, and importance sampling.  
    \item We demonstrate the performance of \salgo\ for principal agent reward shaping and tax design (Sec. \ref{sec:numerical_experiments}).  In certain cases, the stochastic updates of \salgo\ are beneficial as they avoid local minima. Moreover, \salgo~ needs fewer outer iterations compared to benchmark algorithms.
    
\end{itemize}

\section{Related Work}
\label{sec:related work}
\textbf{Stochastic bilevel optimization} has been extensively explored in the literature \cite{dempe2002foundations, bard2013practical}. In recent years, there is a pivotal shift to non-asymptotic analysis of stochastic gradient methods \cite{ghadimi2018approximation, chen2021closing, khanduri2021near, kwon2023fully,kwon2024complexity, Liu2021Value}. \cite{hu2024contextual} propose contextual stochastic bilevel optimization where the lower level solves a static contextual optimization. Our work generalizes to the lower level solving a contextual MDP. This poses unique challenges in terms of hypergradient estimation and sample generation. Comparing to bilevel optimization, leveraging the special structure of \sprob, we avoid Hessian and Jacobian estimation of the lower-level MDP when computing the hyper policy gradient, which is crucial for scalability.

\textbf{Configurable MDP} \citep{metelli2018configurable} is an extension of a traditional MDP 
allowing external parameters or settings to be adjusted by the decision-maker, often referred to as the \textit{configurator}. 
Only recently some works studied the case where the configurator has a different objective than the agent~\cite{ramponi2021learning}. However, that work assumes access to a finite number of parameters that the configurator can control, while our model goes beyond this assumption.
In addition, our model captures the variability and uncertainty that the agent could face in the same configuration environment.

\textbf{Stackelberg games} are a game theoretic framework, where a leader takes actions to which one or multiple followers choose the best response \cite{Stackelberg1934Marktform}. Several existing lines of work have studied solving variants of Stackelberg games.
 Examples include Stackelberg equilibrium solvers~\cite{Fiez2020Implicit,Gerstgrasser2023oracles}, opponent shaping \cite{Foerster2018Learning,Yang2020Learninga}, mathematical programs with equilibrium constraints \cite{Liu2022Inducing,Wang2023Differentiable,wang2022coordinating}, inducing cooperation~\cite{Baumann2020Adaptive,Balaguer2022Good}, steering economic simulations~\cite{Curry2023Learning,zheng2022ai}. These works are either too general with limited implications for our problem or consider entirely distinct settings.

\textbf{Multi-agent RL}~(MARL) studies multiple agents interacting in a joint environment, i.e., their actions together determine the next state~\citep{zhang2021multi,shi2024sample}. In \sprob\, the lower level CMDPs can be seen as a special instance of MARL where the interactions of the followers are restricted to jointly influencing the decision of the leader.

\textbf{Bilevel RL} studies how to design additional rewards or change the underlying MDP to achieve desirable learning outcomes. Many applications are formulated as bilevel RL, such as environment design for generalization \cite{Dennis2020Emergent,Diaz2022Generalization,Yang2022Game}, reward shaping \cite{HadfieldMenell2017Inversea,hu2020learning,Huang2024Learning, Wu2024Robust}, safe reinforcement learning~\cite{turchetta2020safe}, optmizing conditional value at risk\cite{Xia2023Risk‐sensitive}, and model design \cite{chen2022adaptive,zhang2018learning,Brown2024Markov}. Previous work on bilevel RL considers a special case of our setting when there is only one lower-level MDP~\citep{chen2022adaptive,chakraborty2023parl,shen2024principled,yang2024bilevel}. In particular, \cite{chen2022adaptive} focus on the case when the leader has control on the follower's training procedure. \cite{shen2024principled} further extend from one single lower-level MDP to a lower-level min-max game. \cite{chakraborty2023parl} and the concurrent work of \cite{yang2024bilevel} focus on the case when the leader can only influence the reward of the MDP. 

The introduction of the context makes \sprob\ harder to solve as there can be many followers, each with its own preferences, and their best response policies change even for the same leader decision $x$ when facing different contextual uncertainties. Multiple followers and additional side information are very common, which highlights the practical relevance of our work. In addition, the algorithms in the aforementioned works focus on deterministic updates on the upper and lower level decisions, i.e., assuming access to the full hypergradient and performing exact policy gradient/value iteration, which is both computationally hard and not feasible for large-scale practical applications. To the best of our knowledge, we are the first to provide a convergence analysis for the stochastic case, when the hypergradient is estimated from samples and the lower level uses a stochastic update rule. %

\section{Problem Formulation and Applications}
\label{sec:problem}
In this section, we formalize \sprob\ and illustrate its versatility using concrete applications, including RLHF, dynamic mechanism design, tax design, and principal-agent problems.
\subsection{Problem Formulation}

We consider a bilevel optimization problem, where the follower solves a Contextual Markov Decision Process (CMPD) and the leader controls the configuration of the CMDP.
In particular, the leader chooses a parameter $x \in X \subseteq \RR^{d}$ and nature chooses a random context $\xi$ according to a distribution $\PP_\xi$. Together $(x,\xi)$ parameterizes an MDP $\cM_{x,\xi}$, which the follower aims to solve. $\cM_{x,\xi}$ is defined by a tuple $(\cS,\cA, r_{x,\xi}, P_{x,\xi}, \mu_{x,\xi}, \gamma)$, where $\cS$ denotes the state space, $\cA$ denotes the action space, $r_{x,\xi}(\cdot,\cdot):\cS\times\cA\rightarrow \RR$ is the reward function, $P_{x,\xi}(\cdot;\cdot,\cdot): \cS\times\cS\times\cA\rightarrow [0,1]$ denotes the transition kernel, $\mu_{x,\xi}$ indicates the initial state distribution, and $\gamma$ is the discount factor. The subscript $x,\xi$ implies that rewards, transitions, and initial state distribution depend on the leader's decision $x$ and the context $\xi$. 
Connecting to previous works, for a fixed $x$, $\cM_{x,\xi}$  is a \emph{contextual MDP}~\cite{hallak2015contextual} with respect to $\xi$. For a fixed $\xi$, $\cM_{x,\xi}$ generalizes a \emph{configurable MDP}~\cite{metelli2018configurable}. 
Given $\cM_{x,\xi}$, the follower maximizes an entropy-regularized objective by choosing a policy $\pi_{x,\xi}$, where $\pi_{x,\xi}(a;s)$ denotes the probability of choosing action $a$ in state $s$.
\begin{equation}
\label{eq:mdp_problem}
\begin{aligned}
    \max_\pi J_{\lambda,x,\xi}(\pi)=\EE_{s_0}\sbr{V^\pi_{\lambda,x,\xi}(s_0)} = \EE_{s_0}\sbr{\EE^\pi_{P_{x,\xi}} \sbr{\sum_{t=0}^\infty \gamma^t \left(r_{x,\xi}(s_t, a_t) + {\lambda} H(\pi;s_t)\right)}},
\end{aligned}
\end{equation}
where $s_0\sim\mu_{x,\xi}$, $a_t \sim \pi(\cdot;s_t), s_{t+1} \sim P_{x,\xi}(\cdot;s_t, a_t)$ and $H(\pi;s)=-\sum_a \pi(a;s)\log\pi(a;s)$.
We call $\lambda> 0$ the regularization parameter and $V^\pi_{\lambda,x,\xi}$ the value function. As standard in RL literature, we define the related Q and advantage functions as:
\begin{align}
    Q^\pi_{\lambda,x,\xi}(s,a)&=r_{x,\xi}(s,a)+\gamma\EE_{s'\sim P_{x,\xi}(\cdot;s,a)}\sbr{V^\pi_{\lambda,x,\xi}(s')} \nonumber\\
 A^\pi_{\lambda,x,\xi}(s,a)   &=Q^\pi_{\lambda,x,\xi}(s,a)-V^\pi_{\lambda,x,\xi}(s)=Q^\pi_{\lambda,x,\xi}(s,a)-\sum_{a'}\pi(a';s) Q^\pi_{\lambda,x,\xi}(s,a')
 \label{eq:adv_def}.
\end{align}
The unique optimal policy for \eqref{eq:mdp_problem} is given by $\pi^*_{x,\xi}(s;a)\propto \exp(Q^*_{\lambda,x,\xi}(s,a)/\lambda)$, i.e., the softmax of the optimal Q-function \cite{Nachum2017Bridging}.\footnote{For brevity, we notationally drop the dependence of $\pi_{x,\xi}$ on $\lambda$.} Given $x,\pi^*_{x,\xi},\xi$, the leader in turn incurs a loss $f(x,\pi^*_{x,\xi},\xi) \in \RR$, which it wants to minimize in expectation over $\PP_\xi$. To do so, it needs to choose $x$ to align the follower's policy $\pi^*_{x,\xi}$ with the leader's objective. \sprob\ can thus be formulated as the following stochastic bilevel optimization problem: 
\begin{equation}
    \label{problem:multi_lower_level_conditional}
     \begin{aligned}
    \min_{x} \quad & F(x):=\EE_{\xi} [f(x,\pi^*_{x,\xi}, \xi)] & \text{(leader, upper-level)} \\
    \mathrm{where} \quad  & \pi^*_{x,\xi} = \argmax_{\pi} J_{\lambda,x,\xi}(\pi). & \text{(follower, lower-level)}
    \end{aligned}
\end{equation}
\Cref{problem:multi_lower_level_conditional} is well-defined due to entropy regularization ensuring the uniqueness of $\pi^*_{x,\xi}$. It further ensures $\pi^*_{x,\xi}$ is differentiable, stabilizes learning, and appears in previous works~\cite{chen2022adaptive}. %

\subsection{Applications: RLHF, Tax Design, Reward Shaping, Contract Design, and Mechanism Design}
\label{sec:applications}

\sprob~ captures several important real-world problems, which we discuss below. For a clearer exposition, we omit the entropy regularization term at the lower level. However, we stress that some of the referenced works explicit use entropy regularization \cite{Curry2024Automated,chen2022adaptive} or make overlapping assumptions such as unique optimal policies \cite{chakraborty2023parl}. For problems without explicit entropy regularization we refer to \cite{chen2022adaptive, Mei2020Global,Dai2018SBEED,geist2019theory} who have shown that entropy-regularized RL approximates the unregularized problem both in the upper and lower level as $\lambda\rightarrow0$.

\noindent \textbf{Reinforcement Learning from human feedback (RLHF)} considers the setting where an agent tries to learn a task from human feedback. The difficulty stems from the fact that the human feedback is given as preferences over two possible trajectories and not directly as a reward \cite{Christiano2017Deep}. The feedback is of the form $\{\tau_0,\tau_1,l\}$ where $\tau_0,\tau_1$ are two trajectories and $l\in{0,1}$ indicates whether $\tau_0$ is preferred over $\tau_1$ or vice versa.
It has been shown that this problem can be framed as \sprob, where the upper-level tries to learn rewards that minimize the cross-entropy loss, between the actual and predicted labels, using the Bradley-Terry Model and the lower level finds the optimal policy with respect to that reward function \cite{chakraborty2023parl,shen2024principled},
\begin{equation*}
\label{eq:rlhf}
\begin{aligned}
    \max_{x} \ &\EE_{\tau_0,\tau_1\sim \cD(\pi^*_{x}),l}\left[(1-l)\log \PP \left(\tau_0 \succ \tau_1|r_x\right)+l\log\PP\left(\tau_1\succ \tau_0|r_x\right)\right]
   \\ &\text{s.t.} \ \pi_{x}^{*}(\cdot) = \argmax_\pi \EE \left[ \sum_{t=0}^H \gamma^t r_x(s_t,a_t) \right]. \\
\end{aligned}
\end{equation*}
Here $H$ is the time horizon. $D$ is the sampling distribution of trajectories using $\pi^*_{x,\xi}$ and \begin{equation*}
    \PP \left(\tau_0 \succ \tau_1|r_x\right) =\frac{\exp \sum_{t=0}^H \gamma^t r_x(s^0_t,a^0_t)}{\exp \sum_{t=0}^H \gamma^t r_x(s^0_t,a^0_t)+\exp \sum_{t=0}^H \gamma^t r_x(s^1_t,a^1_t)}.
\end{equation*}
Note in the case of standard RLHF the context becomes trivial.

\noindent \textbf{Tax Design for Macroeconomic Modeling} considers a public entity setting tax rates and representative households responding optimally by balancing their short-term utility of consumption and long-term wealth accumulation \cite{hill2021solving, chen2022adaptive, zheng2022ai}. A potential formulation of this problem as \sprob\ is
\begin{equation*}
\label{eq:tax_design_problem}
\begin{aligned}
    \max_{x, y} \EE_\xi\left[\phi(x, y, \pi^*_{x,y,\xi},\xi)\right]~
   \text{s.t.} \ \pi_{x,y,\xi}^{*}(\cdot) = \argmax_\pi \EE \left[ \sum_{t=0}^\infty \gamma^t \Big(r_{\xi}^W(s_t) + r^C_{y, \xi}(\pi(s_t))\Big) \right], \\
\end{aligned}
\end{equation*}
where $x$ and $y$ denote the income and value-added tax rates, respectively, and $\phi$ defines the social welfare objective of the leader. %
The state $s_t$ defines the wealth of a household while their actions decide their working hours and consumption in each time step. The reward function $r^W_{\xi}$ and $r^C_{y, \xi}$ define the households' utility functions for wealth and consumption, respectively.
The value-added tax rate $y$ affects the consumption utility function $r^C_{y, \xi}$ while the income tax $x$ changes the transition kernel modeling wealth accumulation, i.e., $s_{t+1} \sim P_{x,\xi}(\cdot;s_t, a_t)$.
$\xi$ represents the preferences of the households over several consumption goods and their productivity in this problem formulation.

\noindent \textbf{Population Principal-Agent Reward Shaping} considers a principal aiming to craft a non-negative bonus reward function $r^B_x$, parameterized by $x$, to motivate an agent \citep{ben2023principal, yu2022environment, zhang2008valuebased}. Commonly, a principal faces multiple agents that form a distribution. Each agent has its own individual reward function $r_\xi$. This scenario, termed \textit{population principal-agent reward shaping} is captured by our \sprob\ framework.
    \begin{equation*}
    \label{eq:principal_agent_problem}
    \begin{aligned}
        \max_x \EE_{\xi}^{\pi_{x,\xi}^{*}}\left[\sum_{t=0}^\infty \gamma^t \overline{r}(s_t, \pi_{x,\xi}^*(s_t))\right]\ 
       \text{s.t.} \ \pi_{x,\xi}^{*}(\cdot) = \argmax_\pi \EE_\pi \left[ \sum_{t=0}^\infty \gamma^t \Big(r_\xi(s_t, \pi(s_t))+r^B_x(s_t,\pi(s_t))\Big) \right]. \\
    \end{aligned}
    \end{equation*}
Here $\mathbb{E}_\xi$ denotes the expectation over the distribution of agents and the trajectories. The policy $\pi_{x,\xi}^*(\cdot)$ is the optimal response of the $\xi$-th agent to the composite reward function $r_\xi + r^B_x$. The principal's reward is $\overline{r}(s_t, a_t)$ when the agent visits the state action pair $(s_t,a_t)$. 

\noindent\textbf{Dynamic Contract Design} studied by
\cite{Ivanov2024Principal,Wu2024Contractual} is similar to the above reward shaping. Generalizing it to a contextual setting, the problem consists of an agent of type $\xi$ that incurs a cost $c_{\xi}(s,a)$ for taking action $a$ in state $s$. The principal in turn gets a reward $r(s,s')$ for transitioning from state $s$ to $s'$, but cannot observe the agent's action. It can however offer contracts $x(s,s')$ that get paid if the MDP transitions from state $s$ to $s'$. These contracts are positive payments by the principal and are thus added to the lower-level objective and substracted from the upper-level objective.
    \begin{equation*}
    \label{eq:dynamic_contract_design}
    \begin{aligned}
        \max_x &\EE_{\xi,\pi_{x,\xi}^{*}}\left[\sum_{t=0}^\infty \gamma^t \left({r}(s_t, s_{t+1})-x(s_t,s_{t+1})\right)\right]\ \\
        &\text{s.t.} \ \pi_{x,\xi}^{*}(\cdot) = \argmax_\pi \EE_\pi \left[ \sum_{t=0}^\infty \gamma^t \Big(x(s_t, s_{t+1})-c_{\xi}(s_t,a_t)\Big) \right]. \\
    \end{aligned}
    \end{equation*}

\noindent\textbf{Dynamic Mechanism Design} considers the problem of a mechanism designer controlling an MDP for a group of $n$ bidders, who get a reward based on the observed trajectories~\cite{Curry2024Automated}. The context $\xi$ parameterizes the bidders' reward functions $r_{i,\xi}$, which they report to the mechanism designer. The latter wants to learn a policy for the MDP and charge payments to the bidders, to ensure eliciting truthful reward reports and also maximizize an objective $\mathcal{L}$, e.g. the total sum of payments. 
In this setting, \cite{Curry2024Automated} propose to search for such a mechanism within the class of affine maximizers, as they guarantee truthful reports by all bidders. In these mechanisms, a set of agent-dependent weights $x_{w,i}$ and state-action dependent boosts $x_{b}$ is chosen by the mechanism designer, then a policy $\pi$ is learned to maximize the corresponding affinely transformed social welfare $\mathbb{E}_{s_t, a_t \sim \pi}\left[\sum_{t=0}^T  \left(\sum_{i=1}^n x_{w,i} r_{i,\xi}(s_t, a_t)\right)+x_b(s_t, a_t) \right]$ and bidders are charged for the learned policy depending on their reported reward functions. Searching for the optimal mechanism parameters $x_{w,i}$ and $x_{b}$ to maximize $\mathcal{L}$ in expectation over $\xi$, subject to the constraint that the mechanism's policy maximizes affine social welfare can be formulated as \sprob. In this case $x_{w,i}$  and $x_{b}$ are the decision variable, the context parameterizes the bidders' preferences and the affinely transformed social welfare at each time step is the reward function of the lower-level MDP, as shown below:
\begin{align*}
      \min_{x_w,x_b}  &\mathbb{E}_{\xi}[\mathcal{L}\left(\pi^*_{\xi,x_w,x_b},x_w,x_b\right)] \\
      \quad &\text { s.t. } 
  \pi^*_{\xi,x_w,x_b} = \arg \max _\pi \mathbb{E}_{s_t, a_t \sim \pi}\left[\sum_{t=0}^T  \left(\sum_{i=1}^n x_{w,i} r_{i,\xi}(s_t, a_t)\right)+x_b(s_t, a_t) \right].
\end{align*}
\looseness=-1
Note, that all previous works in these application areas have either focused on the setting with a single representative follower~\cite{ben2023principal,chen2022adaptive,Ivanov2024Principal,Wu2024Contractual} or presented a problem-specific algorithm that cannot capture our \sprob~framework in its full generality~\cite{ben2023principal,Curry2024Automated}.

The \sprob~framework is also related to \textbf{Meta reinforcement learning (Meta RL)}, that aims to leverage the similarity of several RL tasks to learn common knowledge and use it on new unseen tasks \citep{beck2023survey}. One way to formulate Meta RL problems is to find a common regularization policy $\Tilde \pi$ for multiple tasks.
\begin{equation*}
    \begin{aligned}
        \max_{\Tilde{\pi}} \EE_{\xi} \left[
            \sum_{t=0}^\infty \gamma^t r_\xi(s_t, \pi^*_{\Tilde{\pi}, \xi}(s_t))
        \right] \mathrm{s.t.}\  \pi^*_{\Tilde{\pi}, \xi}(\cdot) = \argmax_{\pi} \left[
            \sum_{t=0}^\infty \gamma^t r_\xi(s_t, \pi(s_t)) 
            - \frac{\lambda}{2} KL(\pi(s_t) || \Tilde{\pi}(s_t))
        \right],
    \end{aligned}
\end{equation*}
where $\xi$ represents the distribution of multiple RL tasks and $r_\xi$ is the reward for the task indexed by $\xi$. Although the upper-level regularizer is different from entropy-regularization it still results in a unique softmax policy of the form $\pi_{s_{\xi}}^*(s, a) \propto \exp\left(\frac{Q_{s_{\xi}, \pi}(s, a)}{\lambda} + \log(\tilde{\pi}(s, a))\right)$ \cite{Vieillard2020Leverage}. Moreover, in Meta RL the leader does not change the transitions or rewards, but the target policy $\tilde{\pi}$, that goes into the KL regularization term of the followers.

\section{\algo\ Algorithm for \sprob} \label{sec:algorithm}
In this section, we derive a simple expression for the hypergradient of \sprob. We present \salgo\ and prove non-asymptotic convergence. \salgo\ can be combined with a large class of lower-level MDP solvers satisfying a mild inexact oracle assumption. We show this is the case for several popular RL algorithms. Furthermore, we present results for two important special cases of our problem: (1) when the upper-level objective decomposes as a discounted sum of rewards over the lower-level trajectories, and  (2) when the leader can direct the lower-level algorithm. We defer all proofs to \Cref{app:proofs}.
and make the following standard assumptions on how $x$ and $\xi$ influence the setup of the CMDP.\looseness=-1 

\begin{assumption}
\label{assum:theory}
We assume the following conditions:
    \begin{itemize} 
        \item $f$ is $L_f$-Lipschitz continuous and $S_f$-smooth in $x$ and $\pi$, uniformly for all $\xi$, i.e. 
           \begin{align*}
   \norm{ f(x_1, \pi_1, \xi) - f(x_2, \pi_2, \xi) }_\infty &\leq L_f \left( \| x_1 - x_2 \|_\infty + \| \pi_1 - \pi_2 \|_\infty \right)\\
   \|\partial_x f(x_1, \pi_1, \xi) - \partial_x f(x_2, \pi_2, \xi) \|_\infty &\leq S_f \left( \| x_1 - x_2 \|_\infty + \| \pi_1 - \pi_2 \|_\infty \right)\\
   \|\partial_\pi f(x_1, \pi_1, \xi) - \partial_\pi f(x_2, \pi_2, \xi) \|_\infty &\leq S_f \left( \| x_1 - x_2 \|_\infty + \| \pi_1 - \pi_2 \|_\infty \right)
   \end{align*}
        \item $\forall x,\xi : |r_{x,\xi}(s,a)|<\overline{R}$, $\norm{\partial_x \log P_{x,\xi}(s';s,a)}_\infty<K_1$, $\norm{\partial_x r_{x,\xi}(s,a)}_\infty<K_2$.
    \end{itemize}
\end{assumption}
\subsection{Hypergradient derivation}
\label{sec:hpg}

The leader's loss $f$ depends on $x$ directly and indirectly through $\pi^*_{x,\xi}$. Therefore, the derivative of $f$ with respect to $x$ is commonly referred to as the \textit{hypergradient} to highlight this nested dependency. It is possible to obtain a closed-form expression of the hypergradient, using the implicit function theorem~\citep{ghadimi2018approximation}. However, this involves computing and inverting the Hessian of the follower's value function, which can be computationally expensive and unstable~\cite {Fiez2020Implicit,Liu2022Inducing}. Instead, we leverage the fact that $\pi^*_{x,\xi}$ is a softmax function to explicitly compute its derivative with respect to $x$, which is given by $\frac{d \pi^*_{x,\xi}(s,a)}{d x}= \frac{1}{\lambda} \pi^*_{x,\xi}(a;s) \partial_x {A^{\pi^*_{x,\xi}}_{\lambda,x,\xi}(s,a)}$~\cite{chen2022adaptive}, where $\partial_x{A}=(\partial_{x_1}A,\dots,\partial_{x_d}A)$. Applying the Dominated Convergence Theorem to switch derivative and expectation, we arrive at \Cref{thm:Fgrad}.
\begin{theorem}
    \label{thm:Fgrad} Under \Cref{assum:theory}, $F$ is differentiable and the hypergradient is given by
    \begin{equation} 
    \label{eq:leadergrad}
        \frac{d F(x)}{dx}= \EE_{\xi}\left[\frac{\partial_1 f(x, \pi^*_{x,\xi},\xi)}{\partial x} + \EE_{s\sim \nu, a\sim \pi^*_{x,\xi}} \left[ \frac{1}{\lambda \nu(s)} \frac{\partial_2 f(x, \pi^*_{x,\xi},\xi)}{\partial \pi^*_{x,\xi}(a;s)}   {\partial_x{A^{\pi^*_{x,\xi}}_{\lambda,x,\xi}(s,a)}}\right] \right],
    \end{equation}
    where $\nu$ is any sampling distribution with full support on the state space $\cS$.
\end{theorem}

The first term captures the direct influence of $x$ on $f$, and the second the indirect influence through $\pi^*_{x,\xi}$. For now we assume the leader knows $\partial_1 f(\cdot,\pi,\xi)$ and $\partial_2 f(x,\cdot,\xi)$. It remains to compute ${\partial_x{A^{\pi^*_{x,\xi}}_{\lambda,x,\xi}(s,a)}}$, i.e. the partial derivative with respect to $x$ evaluated for a given policy. For this, cf. \eqref{eq:adv_def}, we need ${\partial_x{Q^{\pi^*_{x,\xi}}_{\lambda,x,\xi}(s,a)}}$. We derive an expression for the latter in  \Cref{prop:policygradient}. The proof adapts the analysis of the policy gradient theorem to account for the dependence of $P_{x,\xi},\mu_{x,\xi}$ and $r_{x,\xi}$ on $x$.
\begin{theorem}
\label{prop:policygradient} For given $\pi,x,\xi$, it holds that:
     \begin{align*}
        \partial_x{Q^{\pi}_{\lambda,x,\xi}(s_0,a_0)}      &= \EE_{s,a}^\pi \left[ \sum_{t=0}^\infty  \gamma^t \frac{d r_{x,\xi}(s_t,a_t)}{dx} + \gamma^{t+1} \frac{d \log P_{x,\xi}(s_{t+1};s_t,a_t)}{dx}V^{\pi}_{\lambda,x,\xi}(s_{t+1}) \right].
    \end{align*}
\end{theorem}
Note, \Cref{thm:Fgrad,prop:policygradient} generalize existing results in model design for MDPs to CMDPs~\cite{chen2022adaptive,zhang2018learning}.\looseness=-1

\subsection{\salgo\ Algorithm and Convergence Analysis}
\label{sec:oracle}
Computing the exact hypergradient is computationally expensive and thus infeasible in larger settings. Instead, to minimize $F(x)$, one would ideally sample unbiased estimates of the hypergradient in \Cref{eq:leadergrad} and run stochastic gradient descent (SGD). However, the leader does not have access to $\pi^*_{x,\xi}$ and generally no control over the training procedure of the lower level. Instead, we assume the follower adapts any preferred algorithm to solve the MDP up to a certain precision $\delta$ and the leader can observe trajectories from the follower's policy. Such a setting is well-motivated by economic applications. \looseness=-1
\begin{assumption}
\label{assumption:algorithm}
For any $\cM_{x,\xi}$, the leader has access to an oracle $o$, which returns trajectories sampled from a policy $\pi^{o}_{x,\xi}$ such that $\forall x,\forall\xi:\EE_{o}\left[\norm{\pi^*_{x,\xi}-\pi^o_{x,\xi}}^2_{\infty}\right]\leq \delta^2$. %
\end{assumption}
We will show that \Cref{assumption:algorithm} is relatively mild and holds for a variety of RL algorithms. Given access to trajectories generated by $\pi^o_{x,\xi}$, the leader can construct an estimator of $\partial_x A^{\pi^o_{x,\xi}}_{\lambda,x,\xi}(s,a)$ by rolling out $\pi^o_{x,\xi}$ for $T$ steps, where $T\sim \text{Geo}(1-\gamma)$. We defer the construction (\Cref{alg:sample_gradient_estimation}) and proof of unbiasedness (\Cref{prop:unbiasedgradient}) to the Appendix.  Using this estimator, we introduce \salgo\ in \Cref{alg:oracleSGD}. As $F$ is generally nonconvex due to the bilevel structure~\citep{ghadimi2018approximation}, we demonstrate non-asymptotic convergence to a stationary point of $F$, which matches the lower bound for solving stochastic smooth nonconvex  optimization~\citep{arjevani2023lower}.
\setlength{\textfloatsep}{7pt}
\begin{algorithm}[!t]
\caption{\algo~(\salgo)}
\label{alg:oracleSGD}
\begin{algorithmic}[H]
\STATE \textbf{Input:} Iterations $T$, Learning rate $\alpha$, Regularization $\lambda$, Trajectory oracle $o$, Initial point $x_0$
\FOR{$t = 0$ to $T-1$}
    \STATE $\xi\sim \mathbb{P}_\xi$, $s \sim \nu$ and $a \sim \pi^o_{x,\xi}(\cdot;s)$
    \STATE $\widehat{\partial_x A^{\pi^o_{x,\xi}}_{\lambda,x,\xi}}(s,a) \gets \texttt{GradientEstimator}(\xi,x_t,s,a,o)$  (\Cref{alg:sample_gradient_estimation})
    \STATE $\widehat{\frac{dF}{dx}} \gets \frac{\partial_1 f(x_t, \pi^o_{x_t,\xi},\xi)}{\partial x}  + \frac{1}{\lambda \nu(s)} \frac{\partial_2 f(x_t, \pi^o_{x_t,\xi},\xi))}{\partial \pi(s,a)} \widehat{\partial_x A^{\pi^o_{x,\xi}}_{\lambda,x,\xi}}(s,a)$
    \STATE $x_{t+1}\gets x_t-\alpha \widehat{\frac{dF}{dx}}$
\ENDFOR
\STATE \textbf{Output:} $\hat{x}_T \sim \textrm{Uniform}(\{x_0,\dots,x_{T-1}\})$
\end{algorithmic}
\end{algorithm}
\begin{theorem}
   \label{thm:convergence}
   Using HPGD, under Assumptions \ref{assum:theory} and \ref{assumption:algorithm}, for $\alpha\leq 1/(2S_f)$, we have the following:
  \begin{equation}
  \label{eq:convtheorem}
      \EE\norm{{\frac{dF(\hat{x}_T)}{dx}}}^2 = \cO \Big(\frac{1}{\alpha T}+\delta+\alpha \Big).
  \end{equation}
   For $\alpha =\cO(1/\sqrt{T})$ and $\delta=\cO(1/\sqrt{T})$, \salgo\ converges to a stationary point at rate $\cO(1/\sqrt{T})$.
\end{theorem}
\noindent\textit{Proof sketch}.~~
Using the smoothness of $F$ and the fact that $\hat{x}_T$ is uniformly sampled from all iterates, we upper bound the left side of \eqref{eq:convtheorem} by the sum of three terms. The first is $|F(x_0)-\min_x F(x)|/(\alpha T)$. The second depends on the bias of our gradient estimate, which we show is linear in $\delta$. The last term depends on $\alpha$ times the variance of our estimator, which is bounded. \qed\looseness=-1

To the best of our knowledge, \Cref{thm:convergence} is the first result that shows convergence when using stochastic estimates for the hypergradient. When the hypergradient can be computed exactly, the last $\cO(\alpha)$ term vanishes and we recover the deterministic convergence rates of previous works~\cite{chen2022adaptive,chakraborty2023parl,shen2024principled}.

\looseness=-1 
Another major advantage of \salgo\ is that the follower can use any (possibly stochastic) algorithm satisfying \Cref{assumption:algorithm} to solve the lower-level MDP, while the leader only needs access to generated trajectories. 
While \Cref{assumption:algorithm} certainly holds if the follower solves the MDP exactly, for example with an LP-solver, we are interested in verifying it for common RL algorithms, which can scale to larger state and action spaces.
In \Cref{app:convergenceRL}, we prove non-asymptotic convergence to $\pi^*_{x,\xi}$ for Soft Value Iteration, which converges at rate $\mathcal{O}(\log 1/\delta)$ (\Cref{prop:valueiterconv}); Q-learning, which converges at rate of $\mathcal{O}(\log (1/\delta)/\delta^2)$ (\Cref{prop:qlearningconv}) and Natural Policy Gradient, which converges at rate of $\mathcal{O}(\log 1/\delta)$ (\Cref{prop:npgconv}). Additionaly, we show Vanilla Policy Gradient converges asymptotically in \Cref{prop:pggconv}.
All these Algorithms thus satisfy \Cref{assumption:algorithm}, which makes \salgo\ scalable and widely applicable to settings where followers might use a variety of model-free or model-based algorithms.

\subsection{Upper-Level Discounted Reward Objective}
\label{sec:linear_loss_function}
So far we assumed the leader knows $\partial_1 f(\cdot,\pi,\xi)$ and $\partial_2 f(x,\cdot,\xi)$. In this subsection, instead, we assume $f$ can be written as the negative expected sum of discounted rewards over the lower-level trajectories and show how to estimate the hypergradient from trajectory samples without explicit knowledge of $\partial_1 f(\cdot,\pi,\xi)$ and $\partial_2 f(x,\cdot,\xi)$. In many practical applications, such as reward shaping, or dynamic mechanism design~(cf. \Cref{sec:applications}), the loss $f$ satisfies:
\begin{equation}
\label{eq:fdecomp}
    f(x,\pi^*_{x;\xi}, \xi)=-\EE^{\pi^*_{x,\xi}}_{s_0\sim \mu} \Big[ \sum\nolimits_t \gamma^t \overline{r}_{x,\xi} (s_t,a_t) \Big].
\end{equation}
 Here $\overline{r}_{x,\xi}$ represents the reward of the leader, which is generally distinct from the follower's reward. The expectation is taken over trajectories induced by the lower-level $\pi^*_{x,\xi}$. In this case, the leader does not know the partial derivatives of $f$ but can still estimate the hypergradient from trajectory samples. The following proposition follows from a similar anlysis as the policy gradient theorem.
\begin{proposition}
    \label{prop:polgradlinear}
    If $f$ decomposes as in \Cref{eq:fdecomp}, then $\frac{d F(x)}{dx}$ can be expressed as follows: 
    \begin{equation}
    \label{eq:decomposedhypergrad}
        \begin{aligned}
        \frac{d F(x)}{dx} =  \EE_\xi\Bigg[\EE^{\pi^*_{x,\xi}}_{s_0\sim\mu_{x,\xi}} \Bigg[& \sum_{t=0}^\infty \gamma^t \Biggl( \frac{1}{{\lambda}}\partial_x A_{\lambda,x,\xi}^{\pi^*_{x,\xi}}(s_t,a_t) \overline{Q}_{x,\xi}(s_t,a_t)\\& +\frac{d\overline{r}_{x,\xi}(s_t,a_t)}{dx}  + \partial_x \log P_{x,\xi}(s_t;s_{t-1},a_{t-1}) \overline{V}_{x,\xi}(s_t)\Biggr) \Bigg] \Bigg],
        \end{aligned}
        \end{equation}
    where for compactness, we slightly abuse notation to express $\mu_{x,\xi}(s_0)$ as $P_{x,\xi}(s_0,a_{-1},s_{-1})$.
\end{proposition}
\looseness=-1
Here $\overline{V}_{x,\xi},\overline{Q}_{x,\xi}$ are the (unregularized) value and state action value functions with respect to $\overline{r}_{x,\xi}$. Comparing to \Cref{thm:Fgrad}, note that the expectation is over trajectories with starting states distributed according to the actual initial distribution $\mu_{x,\xi}$ instead of some $\nu$. We discuss how to construct estimators for \eqref{eq:decomposedhypergrad} in \Cref{alg:decomp_gradient_estimation} (\Cref{app:algorithms}) and prove unbiasedness in \Cref{prop:unbiaseddecomposedgradient} (\Cref{app:auxresults}). A special case of  \Cref{eq:decomposedhypergrad} appeared in~\cite{chen2022adaptive}, where they consider model design for MDPs, that does not take into account contextual uncertainty or the possibility of multiple followers, i.e when the support of $\xi$ is a singleton.

\section{Accelerated \salgo\ with Full Lower-Level Access}
\label{sec:full_access}
Previously, we assumed that the leader does not know the solver used in the lower level and queries trajectories from an oracle. However, in certain settings, such as model design~\cite{chen2022adaptive}, and dynamic mechanism design~\cite{Curry2024Automated}, the leader can additionally influence how the followers solve the CMDP. In this section, we focus on the case when the followers use a stochastic tranining procedure, which usually require a polynomial number of steps in terms of $\delta^{-1}$, to learn the optimal lower-level policy. We argue that if the leader has influence on the followers' training procedure, we can greatly reduce the number of lower-level iterations.
\begin{table}
 \caption{Bias, variance and lower-level iteration complexity of hypergradient estimators when using vanilla soft Q-learning and RT-Q.}
 \label{tab:RT-Q}
    \centering
  \begin{tabular}{@{}lll@{}}
\toprule
 & Vanilla & RT-Q \\ \midrule
Bias  & $\cO(2^{-K/2})$ & $\cO(2^{-K/2})$  \\

Variance & $\cO(1)$ & $\cO(K)$ \\

 Complexity & $\cO(K 2^{K})$ & $\cO(K^2)$ \\
 \bottomrule
\end{tabular}
    \label{tab:rt-q}
\end{table}

Let us assume that the lower level is solved using vanilla soft Q-learning (\Cref{alg:qlearning} in \Cref{app:algorithms}). According to Proposition \ref{prop:qlearningconv}~(\Cref{app:auxresults}), the follower needs to run $T= \cO(K2^K)$ iterations to ensure that $\EE\|\pi^{T}_{x,\xi}-\pi^*_{x,\xi}\|_\infty^2 \leq 2^{-K}$ and thus  $
    \norm{\EE\left[\frac{dF(x)}{dx}  - \widehat{\frac{dF_T}{dx}}\right]}_\infty=\cO(2^{-K/2}),
$ where $\pi^{T}$ denotes the learned policy after running $T$-th Q-learning iterations and $\widehat{\frac{dF_T}{dx}}$ denotes the corresponding hypergradient estimator.

To reduce the lower-level iteration complexity, we propose a randomized early stopping scheme over the lower-level soft Q-learning iterations, denoted as randomly-truncated soft Q-learning (RT-Q). The pseudocode is given in Algorithm \ref{alg:hpgdrtq}~(\Cref{app:prooffaster}). We illustrate the high-level idea below. 

Without loss of generality, consider a subsequence $t_k:= \cO (k2^k)$ such that $t_K:=T$. Let $\frac{d}{dx} F_{T}$ denote the hypergradient estimator, based on the $T$-th policy iterate $\pi^{T}$. It holds that:
\begin{equation*}
    {\frac{d}{dx} F_{T}} = {\frac{d}{dx} F_{t_K}} = {\frac{d}{dx} F_{t_1}} + \sum_{k=1}^{K-1} \left({\frac{d}{dx} F_{t_{k+1}}}-{\frac{d}{dx} F_{t_k}}\right)={\frac{d}{dx} F_{t_1}} + \EE_{k\sim p_{{k}}} \left[\frac{\frac{d}{dx} F_{t_{k+1}}-\frac{d}{dx} F_{t_{k}}}{p_{k}}\right],
\end{equation*} 
where $p_{k}$ denotes a truncated geometric distribution, such that $p_{k}\propto 2^{-k}$. The above shows that ${\frac{d}{dx} F_{t_1}} + p_{k}^{-1}\Big[\frac{d}{dx} F_{t_{k+1}}-\frac{d}{dx} F_{t_{k}}\Big]$ with $k\sim p_{k}$ is an unbiased estimator of ${\frac{d}{dx} F_{T}}$. Using this estimator, the follower does not need to run $ \cO(K2^K)$ soft Q-learning iterations but in expectation only $\sum_{k=1}^{K-1} p_{k} t_k =  \cO(K^2)$ iterations. This implies that if the leader can direct how the followers learn and observe behaviors sampled from their learned policies, we can generate a hypergradient estimator with the same bias as $\frac{d}{dx} F_{T}$ but a much smaller lower-level iteration complexity. We formalize our results in the following Theorem.

\begin{theorem}[Improved iteration complexity using RT-Q]
    \label{thm:fasterconv}
    Using Randomly Truncated soft Q-learning (RT-Q) instead of vanilla soft Q-learning to estimate the hypergradient, we achieve the bias, variance, and lower-level iteration complexity results summarized in \Cref{tab:RT-Q}.
\end{theorem}

The idea has been previously studied for contextual bilevel optimization under the name randomly truncated multilevel Monte-Carlo~\cite{Giles2015Multilevel,Hu2021Bias,hu2024contextual}. The reduction in the iteration complexity generally comes at the expense of an increased variance of the hypergradient estimator. In~\cite{hu2024contextual}, this  increase is logarithmic as the lower-level problem is a static optimization problem and samples generated to estimate the hypergradient are independent from the lower-level decision variable. This structure is crucial for controlling the increased variance of the hypergradient estimator. However, for \sprob, rollouts generated from $\pi^t_{x,\xi}$ are used to estimate the hypergradient. These trajectory samples thus depend on the lower-level decision and it is not possible to control the variance as in~\cite{hu2024contextual}. 
To address this issue, we notice that the major source of randomness in our hypergradient estimators stems from the estimator $\widehat{ \partial_x A^{\pi^{t_k}}}(s,a)$ computed by \texttt{GradientEstimator}~(\Cref{alg:sample_gradient_estimation}). We control this randomness by sampling multiple trajectories with a random length, which is in expectation $\cO(1)$. We further sample an action $a$ once from $\pi^{t_{k+1}}$ and then use it to compute both $\widehat{ \partial_x A^{\pi^{t_{k+1}}}}(s,a)$ and $\widehat{ \partial_x A^{\pi^{t_k}}}(s,a)$, using importance sampling. Combining all these tricks with multi-level Monte Carlo, RT-Q achieves a variance of $\cO(K)$, where $K=\cO(\log(\delta^{-1}))$.

\begin{figure}
  \centering
  \begin{subfigure}[b]{0.34\textwidth}
    \centering
    \includegraphics[width=\textwidth]{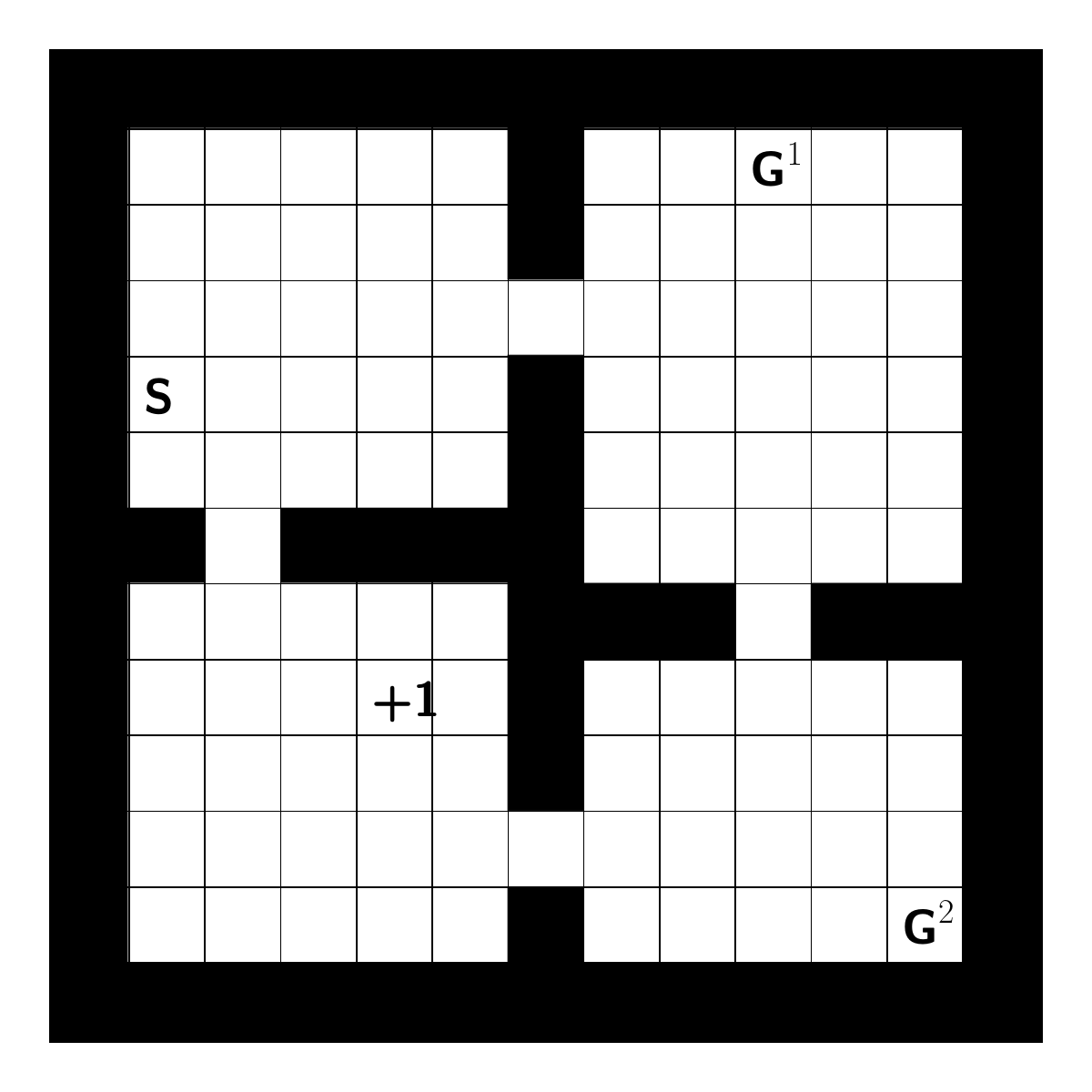}
    \caption{Four Room State Space}
    \label{fig:4rooms}
  \end{subfigure}
  \hfill
  \begin{subfigure}[b]{0.65\textwidth}
      \includegraphics[width=\textwidth]{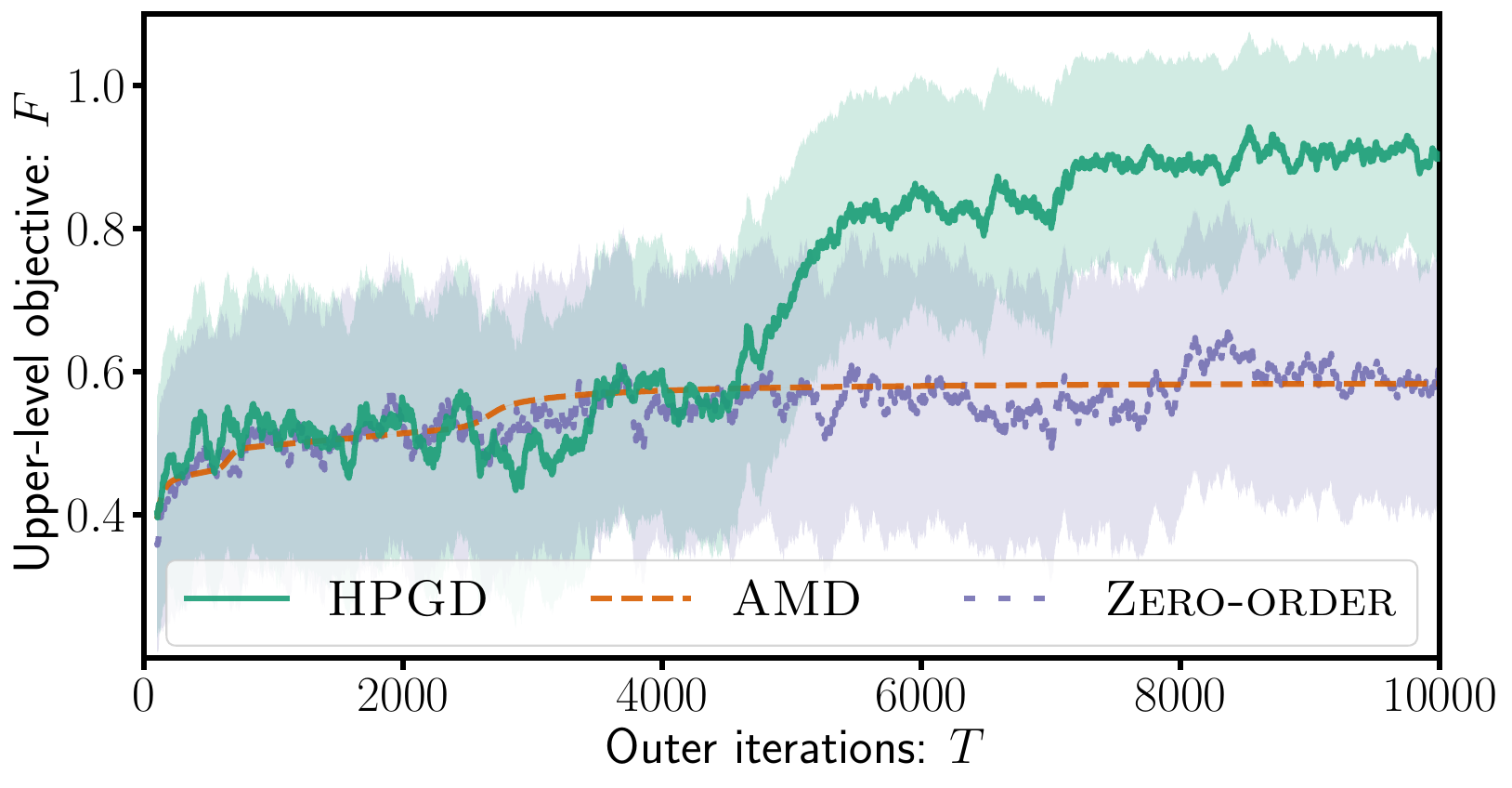}
    \caption{Upper-level objective values, $F$, over the number of outer iterations}
    \label{fig:4rooms_convergence}
  \end{subfigure}
  \caption{Four-Rooms State Space and Performance. \textbf{Left}: $S$ denotes the start state while $G^1$ and $G^2$ denote goal states that are considered separate tasks. $+1$ denotes the target cell to which the upper-level aims to steer the lower-level MDP. \textbf{Right:} \salgo\ escapes local optima achieving higher performance than comparison algorithms.}
\end{figure}

\section{Numerical Experiments}
\label{sec:numerical_experiments}
We illustrate the performance of \salgo\ in the Four-Rooms environment and on the Tax Design for Macroeconomic Model problem \citep{hill2021solving, chen2022adaptive} that we extend to multiple households with diverse preferences.
We use Adaptive Model Design (AMD)~\citep{chen2022adaptive} and a zeroth-order gradient approximation algorithm for comparison. Details on the zeroth-order algorithm are deferred to \Cref{appendix:zero_order_algorithm}.
We note that AMD is not directly applicable to \sprob\ as it was designed for solving an MDP without context. We apply it with modifications described in \Cref{appendix:amd_algorithm}.
To our knowledge, such a zeroth-order gradient method is also the first of its kind for \sprob.
The main distinction between the algorithms is the zeroth-order algorithm requires two oracle queries for each gradient calculation while \salgo\ and AMD require only one. However, the zeroth-order method only needs to observe the function value of the upper level while the latter two require first-order information about the lower-level contextual MDP.
In particular, AMD assumes complete access to the MDP to calculate the exact hypergradient while \salgo\ relies only on trajectory samples.
Technical details about the implementation \footnote{
    We implemented our experiments end-to-end in JAX \citep{jax2018github} for its runtime benefits and ease of experimentation.
    The code is available at \url{https://github.com/lasgroup/HPGD}.
} are deferred to Appendix~(\ref{appendix:implementation_four_rooms}).\looseness=-1

\subsection{Four-Rooms Environment}
Figure~(\ref{fig:4rooms}) depicts the lower-level CMDP for the Four-Rooms environment. $S$ denotes the initial position while $G^1$ and $G^2$ are goal states.
We consider the two goal states as separate tasks and define $\xi$ in \Cref{problem:multi_lower_level_conditional} to be the uniform distribution over the set of tasks, i.e., $\xi \sim \mathrm{Uniform}(\{1, 2\})$.
We denote the goal state in each task by $G^\xi$.
The state space $\cS$ is defined by the cells of the grid world while the actions are the movements in the four directions. In each step $t$, with probability $2/3$, the agent moves to $s_{t+1}$ following the chosen direction $a_t$ while it takes a random movement with probability $1/3$. The reward is always zero except when $s_t = G^\xi$ where $r(s_t, a_t) = 1$, and the episode resets. To incentivize taking the shortest path, we set the discount factor as $\gamma = 0.99$.

For the upper level, we let $x$ parameterize an additive penalty function $\Tilde{r}_{x}: \cS \times \cA \to [-0.2, 0.0]$ \footnote{The parametrization of this function is described in \Cref{appendix:four_rooms_implementation_details}.}, such that the follower receives a reward of $r+\Tilde{r}_{x}$, as in the Principal-Agent problem \cite{ben2023principal}.
The goal of the leader is to steer the followers through the cell marked with $+1$ in \Cref{fig:4rooms}, denoted by $s^{+1}$, while keeping the penalties allocated to states to their minimum. We define $\overline{r}$ in \Cref{eq:fdecomp} as
\begin{equation*}
\looseness -1    \overline{r}_{x, \xi}(s_t, a_t) = \mathds{1}_{\{s_t = s^{+1}\}} - \beta \mathds{1}_{\{s_t = G^\xi\}}\sum\nolimits_{s, a}\Tilde{r}_{x}(s,a),
\end{equation*}
where $\mathds{1}$ is the indicator function and the second term defines the cost associated with implementing the penalties for the lower level.
Note that there is a trade-off between the terms in $\overline{r}$ depending on the context variable $\xi$. If $\xi=2$, the desired change in the follower's policy can be achieved with small interventions since the shortest path from $S$ to $G^2$ is already going through the bottom-left room.
When $\xi = 1$, the leader must completely block the shortest path from $S$ to $G^1$ to divert the follower through the desired state.
An efficient algorithm for this \sprob\ problem therefore must avoid the local optimum of setting $\Tilde{r} = 0$ and find the balance between the follower visiting state $s^{+1}$ and implementing too large penalties in the CMDP.\looseness=-1

\begin{figure}
    \centering
    \includegraphics[width=\textwidth]{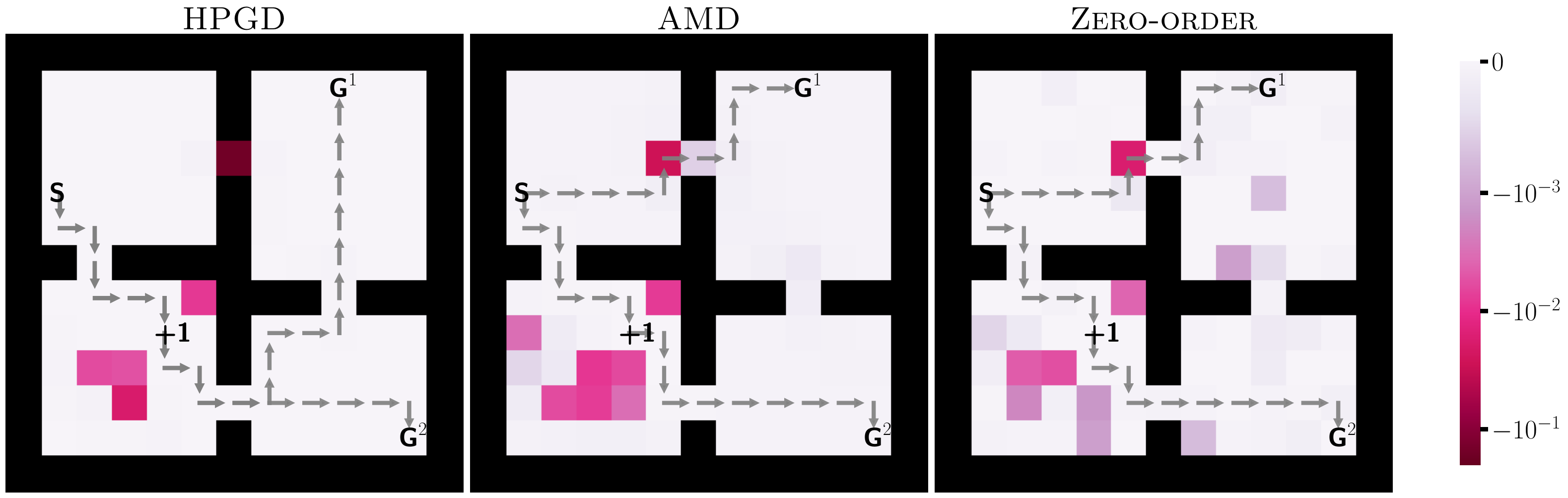}
    \caption{Reward penalties given to the lower-level agent in each state of the Four-Rooms problem optimized by the \salgo, AMD, and Zero-Order, respectively. \salgo\ efficiently steers the lower-level MDP when the task is to reach $G^1$ while others are only successful in the case of $G^2$.}
    \label{fig:4rooms_incentives}
\end{figure}

Figure~(\ref{fig:4rooms_convergence}) depicts the upper-level's objective function over the learning iterations $t$ with hyperparameters $\lambda = 0.001$ and $\beta = 1.0$.
\salgo\ outperforms both AMD and the Zero-Order algorithms in this instance in terms of overall performance. The major difference in their performances is that \salgo\ successfully escapes the local optimum of $\Tilde{r} = 0$ after about $5000$ steps and assigns all the additive penalty budget to states in the gridworld. On the contrary, AMD and Zero-Order converge to the local optimum of minimizing the implementation penalty term in $\overline{r}$. They only utilize $38\%$ and $26\%$ of the available budget of $-0.2$ to divert the follower when $\xi=2$ but neglect the goal state $G^1$.

\Cref{fig:4rooms_incentives} shows the value of additive penalties $\Tilde{r}$ in the state space with the highest probability paths for the goal states. \salgo\ successfully blocks the follower when $\xi=1$ and diverts its shortest path from $S$ to $G^1$ along the other rooms, while AMD and Zero-Order fail to assign sufficient penalty to the upper corridor to cause the same effect. All algorithms are successful in ensuring that the shortest path through the bottom-left room is going through the marked state.\looseness=-1

The parameters $\lambda$ and $\beta$ were chosen for demonstration purposes to highlight the capability of \salgo\ to escape local minima, as has been observed for SGD \citep{xie2021diffusion}. However, we emphasize that in the majority of the cases, the three algorithms perform equally as shown in \Cref{table:4rooms_parameter_comparison}. We provide the figures for the remaining hyperparameters in \Cref{appendix:four_rooms_additional_figures}.
The slightly higher performance of AMD and low standard error among initializations is expected since this algorithm calculates the gradient of $f$ deterministically while \salgo\ and Zero-Order rely on stochastic estimates yielding more variations, especially for the Zero-Order approach.

\begin{table}[t]
  \caption{Performance over hyperparameters $\beta$ and $\lambda$ for the Four Rooms Problem averaged over $10$ random seeds with standard errors. Algorithms perform on-par for most hyperparameters while \salgo\ outperforms others in few. AMD enjoys low variance due to the non-stochastic gradient updates while Zero-Order suffers from the most variation.}
  \label{table:4rooms_parameter_comparison}
  \centering
  \begin{tabular}{ccccc}
    \toprule
    \multicolumn{2}{c}{Parameters} & \multicolumn{3}{c}{Algorithms} \\
    $\lambda$ & $\beta$ & \salgo\  & AMD & Zero-Order\\
    \midrule
    $0.001$ & $1$ & $\textbf{0.91} \pm 0.088$ & $0.58 \pm 0.000$ & $0.59 \pm 0.059$ \\
    $0.001$ & $3$ & $0.51 \pm 0.006$ & $0.51 \pm 0.000$ & $0.50 \pm 0.005$ \\
    $0.001$ & $5$ & $0.46 \pm 0.006$ & $0.46 \pm 0.003$ & $0.46 \pm 0.007$ \\

    $0.003$ & $1$ & $0.95 \pm 0.002$ & $1.00 \pm 0.000$ & $0.91 \pm 0.048$ \\
    $0.003$ & $3$ & $\textbf{0.73} \pm 0.001$ & $0.39 \pm 0.000$ & $0.40 \pm 0.028$ \\
    $0.003$ & $5$ & $0.29 \pm 0.003$ & $0.32 \pm 0.000$ & $0.32 \pm 0.002$ \\

    $0.005$ & $1$ & $1.17 \pm 0.011$ & $1.28 \pm 0.003$ & $1.15 \pm 0.026$ \\
    $0.005$ & $3$ & $1.01 \pm 0.002$ & $1.13 \pm 0.004$ & $1.02 \pm 0.027$ \\
    $0.005$ & $5$ & $0.87 \pm 0.003$ & $0.97 \pm 0.009$ & $0.79 \pm 0.027$ \\
    \bottomrule
  \end{tabular}
\end{table}

\subsection{Tax Design for Macroeconomic Models}
We test \salgo~on a bilevel macroeconomic model evaluating taxation schemes based on \citep{hill2021solving, chen2022adaptive}.
The economic model consists of a finite set of consumption goods $i \in \{1,\dots, M\}$ each with price $p_i$. We choose $M=3$ with unit prices.
On the lower-level, at each time step $t$, $s_t$ denotes the accumulated assets of a household which in turn chooses the number of hours worked $n_t$ and consumption $c_{i,t}$.
The environment updates the accumulated assets of the household as $s_{t+1} = s_t + (1-x)wn_t - \sum_{i=1}^M c_{i,t} + \epsilon$ where $\epsilon$ is sampled from a normal distribution with mean $0$ and standard deviation $\varsigma$. We clip the accumulated assets to the range $[-100, 100]$ for numerical stability.
The utility of a household is given by $u_t = \sigma(s_t) - \theta n_t^2 + \prod_{i=1}^M (c_{i,t} / (p(1+y_i)))^{\alpha_i}$ where the product-of-consumption term corresponds to the Cobb-Douglas function \citep{roth2016watch} and $\sigma(\cdot)$ is the value of accumulated assets.
We define $\mathbb{P}_\xi$ as the distribution of households representing different socio-economic groups in the economy.
In particular, we define two equal-sized groups with $\alpha = (0.6, 0.3, 0.1)$ and $(0.1, 0.7, 0.2)$ that model their different preferences over the goods in the economy and optimize their consumption behavior using regularized deep Q-learning~\citep{geist2019theory} with $\lambda = 0.2$.
On the upper-level, a tax designer is optimizing the discounted sum of social welfare by setting the income tax $x \in [0, 3]$ and value-added tax $y_i \in [0,3]$ for $i \in \{1,\dots,M\}$.
In each time step $t$, the social welfare is defined as $v_t = \omega(s_t) + \sum_{i=1}^M c_{i,t} / (1+y_i) + \phi \log (\sum_{i=1}^M c_{i,t} y_i / (1+y_i) + wxn_t)$ where $\omega(\cdot)$ is the utility of accumulated assets and $\phi$ is a positive hyperparameter.

\Cref{fig:tax_design_performance} demonstrates that \salgo\ can successfully optimize the hyperparameters even in continuous complex problems. Benefiting from first-order information, \salgo\ converges faster than the Zero-Order algorithm and increases the social-welfare by about $25\%$. Additionally, 
\salgo\ shows better qualitative results for the optimised tax rates in \Cref{fig:tax_design_tax_outcomes}. \salgo\ swiftly increases the income tax to improve its objective by increasing tax revenues while sets value-added tax rates according to the preferences of the household groups. Households on average prefer the second good the most, therefore, $y_2$ is set low to increase consumption and therefore maximize the second term in $v_t$ while the third good is the least preferred for which the tax rate is increased\footnote{$y_2 = 1.9$ after $1000$ iterations. The figure is cropped for better readability.} since consumption is already low.
While Zero-Order shows equivalent performance on \Cref{fig:tax_design_performance} to \salgo\, it fails to distinguish goods when setting value-added tax rates.
We defer further details on the implementation, hyperparameter choices, and additional results to \Cref{appendix:tax_design_implementation_details}.

\begin{figure}
    \begin{subfigure}[T]{0.4\textwidth}
        \centering
        \includegraphics[width=\textwidth]{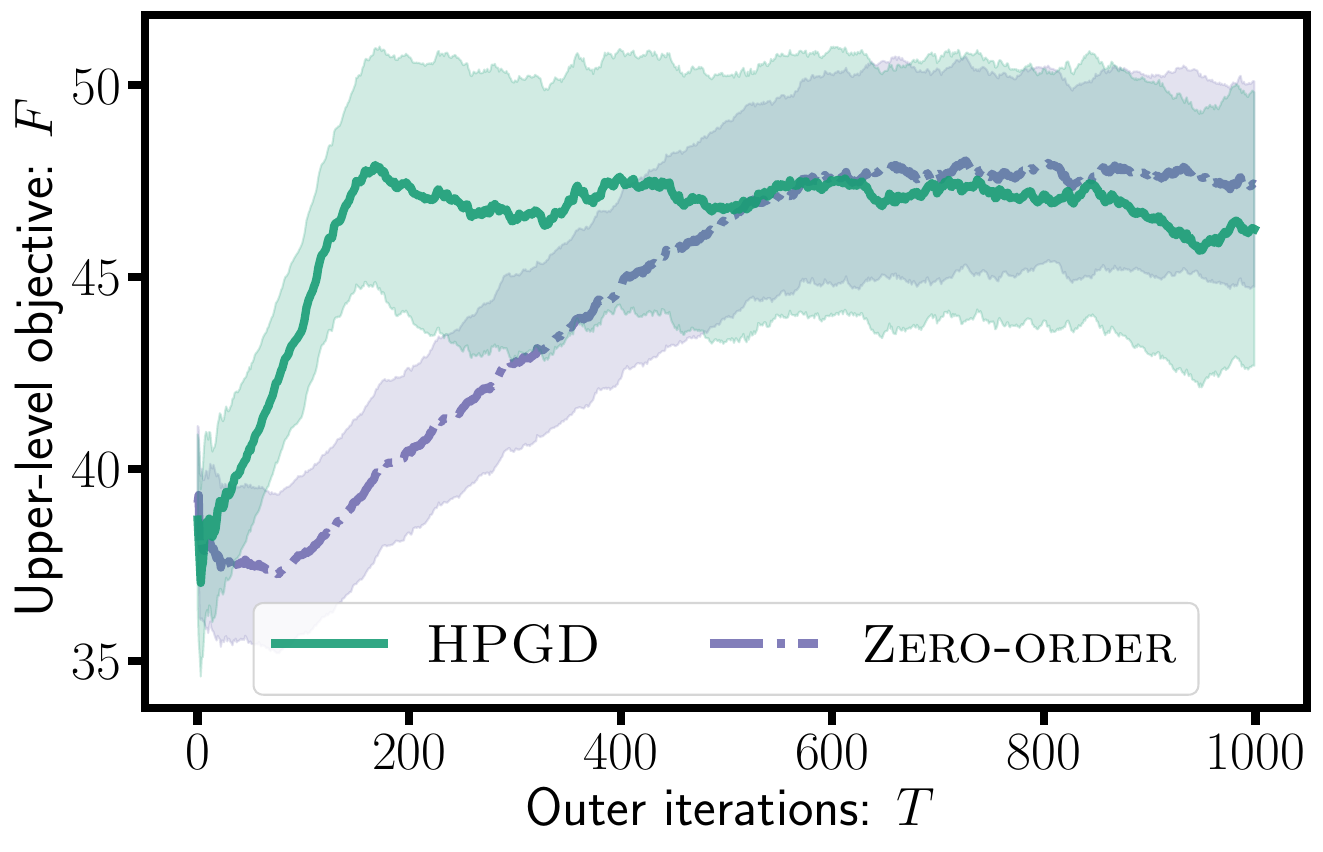}
        \caption{Performance on the Tax Design problem. \salgo\ quickly learn optimal tax policies while Zero-Order takes more iterations.}
        \label{fig:tax_design_performance}
    \end{subfigure}
    \hspace{2pt}
    \begin{subfigure}[T]{0.6\textwidth}
        \centering
        \includegraphics[width=\textwidth]{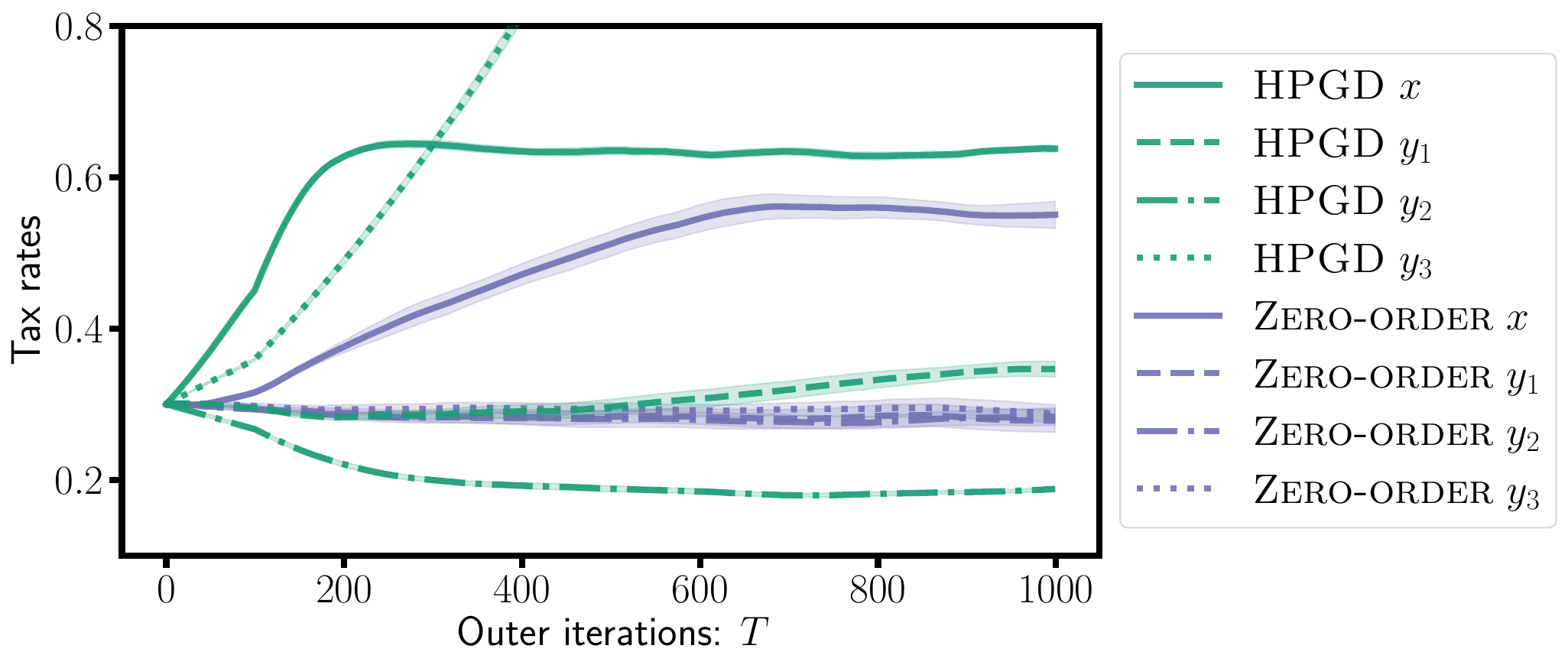}
        \caption{Tax rates over the outer iterations. \salgo\ increases income tax quickly and distinguishes VAT rates according to preferences.}
        \label{fig:tax_design_tax_outcomes}
    \end{subfigure}
\end{figure}

\section{Conclusion}
\looseness-1
We introduce \sprob, a class of stochastic bilevel optimization problems with lower-level contextual MDPs that capture a wide range of important applications, where a leader aims to design environments and incentive structures that align the followers' policies with the leader's upper-level objective. We propose \salgo, an oracle-based algorithmic framework, and analyze its convergence. Importantly, \salgo\ works with any existing algorithm that solves the lower-level CMDP to near-optimality, making it suitable in various regimes when the leader can only observe trajectories of followers. Moreover, \salgo\ is the first provably convergent algorithm in this area, which uses stochastic estimates of the hypergradient. We further propose RT-Q, a more efficient algorithm and study its bias, variance, and cost when the leader can fully control the followers' training. Numerical results further validate the expressiveness of the proposed model and the performance of our algorithm. Future directions include 1) applying HPGD in various real-world applications, 2) studying the setting when the lower-level problem is a game, and 3) studying single-loop algorithms for bilevel reinforcement learning with when the lower-level is just an MDP. %

\section*{Acknowledgements}
The authors acknowledge constructive suggestions from Niao He and Sven Seuken. V. Thoma and  B. Pasztor are supported by an ETH AI Center Doctoral Fellowship. V. Thoma acknowledges funding from the Swiss National Science Foundation (SNSF) Project Funding No. 200021-207343. This work was supported as a part of NCCR Automation, a National Centre of Competence (or Excellence) in Research, funded by the SNSF (grant number 51NF40\_225155).

\medskip
\newpage
\bibliographystyle{plain}
\bibliography{ref}
\newpage
\tableofcontents
\newpage
\appendix

\newpage
\section{Frequently-Used Notation}
\begin{table}[h]
\caption{Table of notation used in the paper.}
    \centering
    \begin{tabular}{ll}
        \toprule
        \textbf{Notation} & \textbf{Description} \\
        \midrule
        $x$ & Upper-level decision variable/leader's decision\\
        $\xi$ & Contextual variable\\
        $\mathcal{M}_{x,\xi}$ & MDP parameterized by $x$ and $\xi$\\
        $\mathcal{S}$ & State Space\\
        $\mathcal{A}$ & Action Space\\
        $r_{x,\xi}$ & Reward function\\
        $P_{x,\xi}$ & Transition kernel\\
        $\mu_{x,\xi}$ & Initial state distribution\\
        $\gamma$ & Discount factor\\
        $\tau$ & Trajectories of  MDP\\
        $\pi_{x,\xi}$ & Follower policy for $\mathcal{M}_{x,\xi}$\\
        $J_{\lambda,x,\xi}(\pi)$ & Entropy-regularized lower-level objective function\\
        $s_0$ & Initial state\\
        $H(\pi;s)$ & Entropy of policy $\pi$ at state $s$\\
        $V^\pi_{\lambda,x,\xi}(s)$ & Value function\\
        $Q^\pi_{\lambda,x,\xi}(s,a)$ & Q-function\\
        $A^\pi_{\lambda,x,\xi}(s,a)$ & Advantage function\\
        $\pi^*_{x,\xi}$ & Optimal policy maximizing $J_{\lambda,x,\xi}(\pi)$ (dependence on $\lambda$ is droped)\\
         $\pi^o_{x,\xi}$ & Oracle policy with distance $\delta$ from $\pi^*_{x,\xi}$\\
          $\pi^{t_k}_{x,\xi}$ & Follower policy after performing $t_k$ learning steps\\
        $f(x,\pi_{x,\xi},\xi)$ & Upper-level loss for specific context $\xi$\\
        $F(x)$ & Upper-level loss function\\
        $\lambda$ & Regularization parameter\\
        $L_f$ & Lipschitz continuity parameter of $f$\\
        $S_f$ & Smoothness parameter of $f$\\
        $\overline{R}$ & Upper bound on absolute value of reward function\\
        $K_1,K_2$ & $\norm{\partial_x \log P_{x,\xi}(s';s,a)}_\infty<K_1$, $\norm{\partial_x r_{x,\xi}(s,a)}_\infty<K_2$.\\
        $\delta$ & Oracle inaccuracy, such that $\forall x,\forall\xi:\EE_{o}\left[\norm{\pi^*_{x,\xi}-\pi^o_{x,\xi}}^2_{\infty}\right]\leq \delta^2$\\
        RT-Q & Randomnly-Truncated Soft Q-learning\\
        \sprob & \prob\\
        \salgo & \algo\\
        $K$ & Used to define bias,variance and complexity of RT-Qneu\\
        $\nu$ & Sampling distribution to estimate hypergradient\\
        $m$ & $m:=\min_s \nu(s)$\\
        $a$ & test\\
        $\widehat{ \partial_x A^{\pi^{t_k}}}(s,a)$ & Estimate of Advantage derivative, obtained from \Cref{alg:sample_gradient_estimation}\\
         $\widehat{ \partial_x Q^{\pi^{t_k}}}(s,a)$ & Estimate of Q derivative, obtained from \Cref{alg:sample_gradient_estimation}.\\
        ${\widetilde{\partial_x Q^{\pi^{t_k}}}}$ & Smoothed Q derivative setimate, ${\widetilde{\partial_x Q^{\pi^{t_k}}}}:=\frac{1}{2^{k}}\sum_{l=1}^{2^{k}}  \widehat{ \partial_x Q^{\pi^{t_k}}}(\tau_l)$  \\
        $\overline{r}$ & Upper-Level reward function for special loss function (cf. \Cref{sec:linear_loss_function})\\
        $\overline{V},\overline{Q}$ & Unregularized upper-level value and Q functions for $\overline{r}$\\
        Geo$(1-\gamma)$& Geometric distribution with parameter $1-\gamma$\\
        $\mathcal{T}^*_\lambda$& Soft Bellman optimality operator\\
        \bottomrule
    \end{tabular}
    \label{tab:notation}
\end{table}
\newpage
\section{Algorithms}
\label{app:algorithms}
We give the pseudocode to certain algorithms/routines/procedures mentioned in the main text.

\begin{algorithm}[H]
\caption{\texttt{GradientEstimator}$(\xi,x,s,a,o)$}
\label{alg:sample_gradient_estimation}
\begin{algorithmic}[H]
\STATE \textbf{Input: } $\xi$, $x$ state $s$, action $a$, trajectory oracle $o$
\STATE $T_Q,T_V \sim \text{Geo}(1-\gamma)$, $T'_Q,T'_V \sim \text{Geo}(1-\gamma^{0.5})$\\
\STATE $\tau_{Q}\gets \texttt{SampleTrajectory}(o,\texttt{start}=(s,a),\texttt{length}=T_Q+T'_Q+1)$ 
\STATE $\tau_{V}\gets \texttt{SampleTrajectory}(o,\texttt{start}=s,\texttt{length}=T_V+T'_V+1)$ 
\STATE $\widehat{\frac{d}{dx} Q}(s, a)\gets \sum_{t=0}^{T_Q} \frac{d}{dx} r(s^{\tau_Q}_t,a^{\tau_Q}_t)+$\\
\STATE$\frac{\gamma}{1-\gamma} \frac{d}{dx} \log P(s^{\tau_Q}_{T_Q+1};s^{\tau_Q}_{T_Q},a^{\tau_Q}_{T_Q}) \sum_{t=T_Q+1}^{T_Q+T'_Q+1} \gamma^{(t-T_Q-1)/2} \left(r(s^{\tau_Q}_t,a^{\tau_Q}_t)+\lambda H(\pi(\cdot ;s_t))\right)$
\STATE $\widehat{\partial_x V}(s)\gets \sum_{t=0}^{T_V} \partial_x r(s^{\tau_V}_t,a^{\tau_V}_t)+$\\
\STATE$\frac{\gamma}{1-\gamma} \partial_x \log P(s^{\tau_V}_{T_V+1};s^{\tau_V}_{T_V},a^{\tau_V}_{T_V}) \sum_{t=T_V+1}^{T_V+T'_V+1} \gamma^{(t-T_V-1)/2} \left(r(s^{\tau_V}_t,a^{\tau_V}_t)+\lambda H(\pi(\cdot ;s_t))\right)$
\STATE \textbf{Output:} $\widehat{\partial_x A(s,a)}\gets\widehat{\partial_x Q}(s, a)-\widehat{\partial_x V}(s)$

\end{algorithmic}
\end{algorithm}

\begin{algorithm}[H]
\caption{Soft Value Iteration}
\begin{algorithmic}[1]
\label{alg:valueiteration}
\STATE \textbf{Input:} Number of iterations $T$
\STATE \textbf{Result:} Approximation $V_{\lambda} \approx V^*_{\lambda}$, policy $\pi_{\lambda} \approx \pi^*_{\lambda}$
\STATE Initialize $V_{\lambda} = 0$
\FOR{$t=0$ \TO $T$}
    \FOR{$s \in \cS$}
    \FOR{$a \in \mathcal{A}$}
        \STATE $Q_{\lambda}(s,a) = r(s, a) + \gamma \mathbb{E}_{s' \mid s, a}\left[V_{\lambda}(s')\right]$
    \ENDFOR
        \STATE $V_{\text{new},\lambda}(s) = \lambda \log \left(\sum_{a \in \mathcal{A}} \exp \left(\frac{Q_{\lambda}(s,a)}{\lambda}\right)\right)$
    \ENDFOR
     \STATE set $V_{\lambda} := V_{\text{new},\lambda}$
\ENDFOR

\STATE $ \pi_{\lambda}^o \leftarrow \frac{\exp(Q_{\lambda}(s,a)/\lambda)}{\sum_a \exp(Q_{\lambda}(s,a)/\lambda)}$
\STATE \textbf{return} $V_{\lambda} $ and $\pi_{\lambda}^o $
\end{algorithmic}
\end{algorithm}

\begin{algorithm}[H]
\caption{\texttt{SoftQlearning}$(T,\pi_B,\{\alpha_{t}\}_{t\geq 0})$}
\begin{algorithmic}[1]
\label{alg:qlearning}
\STATE \textbf{Input:} Number of iterations $T$, Behavioural Policy $\pi_B$, Stepsizes $\{\alpha_{t}\}_{t\geq 0}$
\STATE \textbf{Result:} Approximation $Q_{\lambda} \approx Q^*_{\lambda}$, policy $\pi_{\lambda} \approx \pi^*_{\lambda}$
\STATE Initialize $Q_{\lambda} = 0$
 \STATE Initialise $s_0$
\FOR{$t=0$ \TO $T$}
    \STATE Sample $a \sim \pi_B(\cdot; s_t)$
    \STATE Observe next reward $r(s_t,a)$ and state $s_{t+1}\sim P(\cdot|s_t,a)$
    \STATE $Q_{\lambda}(s_t,a) = Q_{\lambda}(s_t,a) +\alpha_{t} \left( r(s_t, a) + \gamma \lambda \log \left(\sum_{a' \in \mathcal{A}} \exp \left(\frac{Q_\lambda(s_{t+1},a')}{\lambda}\right)\right)\right) $
\ENDFOR
\STATE $ \pi_{\lambda}^o(a;s) \leftarrow \frac{\exp(Q_{\lambda}(a|s)/\lambda)}{\sum_{a'} \exp(Q_{\lambda}(s,a')/\lambda)}$
\STATE \textbf{return} $Q_\lambda$ and $\pi_{\lambda}^o $
\end{algorithmic}
\end{algorithm}

\begin{algorithm}[H]
\caption{\texttt{DecomposableGradientEstimator}}
\label{alg:decomp_gradient_estimation}
\begin{algorithmic}[H]
\STATE \textbf{Input: } $\xi$, $x$, initial distribution $\mu_{x,\xi}$, oracle $o$
\STATE $T, \sim \text{Geo}(1-\gamma)$, $T' \sim \text{Geo}(1-\gamma^{0.5})$\\
\STATE $(s_0,a_0,\dots,s_{T+T'},a_{T+T'})\gets \texttt{SampleTrajectory}(o,\texttt{start}=\mu_{x,\xi},\texttt{length}=T+T')$ 
\STATE $\widehat{A^{\pi^o_{x,\xi}}_{\lambda,x,\xi}}(s_T,a_T)\gets \texttt{GradientEstimator}(\xi,x,s_T,a_T,o)$
\STATE $\widehat{\frac{dF}{dx}}=\left(\sum_{t=0}^T \frac{d}{dx} \overline{r}(s_t,a_t) \right)  + \frac{1}{\lambda(1-\gamma)} \partial_x \widehat{A^{\pi^o_{x,\xi}}_{\lambda,x,\xi}} (s_T,a_T) \sum_{t'=T}^{T+T'} \gamma^{(t-T)/2} \overline{r}(s_{t'},a_{t'})$
\STATE\quad $+\frac{1}{1-\gamma} \partial_x \log P(s_T,a_{T-1},s_{T-1}) \sum_{t'=T}^{T+T'} \gamma^{(t'-T)/2}\overline{r}(s_{t'},a_{t'})$
\STATE \textbf{Output:} $\widehat{\frac{dF}{dx}}$
\end{algorithmic}
\end{algorithm}

\begin{algorithm}[H]
\caption{Vanilla Policy Gradient Algorithm}
\label{alg:pg}
\begin{algorithmic}[H]
\STATE \textbf{Data:} Initial parameter $\theta_0$, initial state $s$
\STATE \textbf{Result:} Approximate policy $\pi_{\theta_L}$
\FOR{$l = 0$ \TO $L$}
    \STATE Sample $T \sim \text{Geo}(1-\gamma)$
    \STATE Sample trajectory $(s_0, a_0, s_1, \ldots, a_{T-1}, s_{T}, r_{T}, a_{T})$ using policy $\pi_{\theta_{l}}$
    \STATE Sample $T' \sim \text{Geo}(1-\gamma^2)$
    \STATE Set $\tilde{s}_0 = s_{T'}$ and $\tilde{a}_0 = a_{T}$
    \STATE Sample trajectory $(\tilde{s}_0, \tilde{a}_0, \tilde{s}_1, \ldots, \tilde{a}_{T'-1}, \tilde{s}_{T'}, \tilde{r}_{T'}, \tilde{a}_{T'})$ using policy $\pi_{\theta_{l}}$
    \STATE Determine step-size $\alpha$.
    \STATE $\widehat{\nabla J}_s(\theta_{l}) = \frac{1}{1-\gamma} \nabla \log \pi_{\theta_{l}}(a_{T} | s_{T}) \sum_{t'=0}^{T'-1} \gamma^{t'/2} \tilde{r}_{t'+1}$
    \STATE $\theta_{l+1} = \theta_{l} - \alpha \widehat{\nabla J}_s(\theta_{l})$
\ENDFOR
\end{algorithmic}
\end{algorithm}

\section{Proofs}
\label{app:proofs}
\subsection{Overview}
In this section, we provide proofs for the presented theorems and propositions. We provide the proof of \Cref{thm:Fgrad}, deriving the hypergradient of function $F(x)$; the proof of \Cref{prop:policygradient}, deriving the derivative of the action-value function with respect to $x$; the proof of our main result, \Cref{thm:convergence}, which shows convergence of \salgo\ to a stationary point of $F(x)$.

For the propositions, we show how to estimate the upper-level gradient if $f$ is decomposable in the proof of \Cref{prop:polgradlinear}; In \Cref{prop:policy_gradient_configuration} we show how to compute the gradient of the optimal policy with respect to $x$; the proof of \Cref{prop:unbiasedgradient}, which shows we can achieve unbiased estimates of the advantage hypergradient; and the proof of \Cref{prop:unbiaseddecomposedgradient}, which shows the same for the special case when $f$ decomposes.

We state and proof Propositions \ref{prop:valueiterconv} to \ref{prop:npgconv} which show convergence in $L_2$ to the optimal policy of soft value iteration, soft Q-learning, Vanilla Policy Gradient and Natural Policy Gradient respectively. 

Lastly, we prove \Cref{thm:fasterconv}, regarding the reduced iteration complexity of RT-Q claimed in \Cref{sec:full_access}.

\subsection{{Proof of }\texorpdfstring{Theorem \ref{thm:Fgrad}}{}}
\begin{proof}
The proof relies on on three main ideas. 
\begin{enumerate}
    \item Show that Dominated Convergence applies, i.e. all derivatives are uniformly bounded by an integrable function. Thus we can exchange derivative and expectation and compute the derivative of $f$ instead of $F$.
    \item Use \Cref{prop:policy_gradient_configuration} to get an expression for the derivative of the optimal policy with respect to $x$.
    \item Use importance sampling with some distribution $\nu$ to get an expression of the hypergradient that we can cheaply sample, instead of having to multiply two matrices with size $|\cS|\times |\cA|$.
\end{enumerate}
    By Proposition \ref{prop:policygradient}, it follows that
\begin{align*}
    \norm{\frac{\partial \pi^*_{x,\xi}}{\partial x}}_{\infty} &\leq \frac{2}{\lambda} \norm{\partial_x Q^{\pi^*_{x,\xi}}_{\lambda,x,\xi}(s,a)}_{\infty}\\
    &\leq  \frac{2}{\lambda} \frac{K_2}{1-\gamma} \frac{K_1 \overline{R}K_1}{(1-\gamma)^2}.
\end{align*}
As the partial derivatives of $f$ are bounded by $L_f$ (cf. \Cref{assum:theory}), we can apply the Dominated Convergence Theorem to get 

\begin{align}
\partial_x \EE\left[ f(x, \pi^*_{x,\xi}, \xi) \right] &=  \EE\left[\partial_x f(x, \pi^*_{x,\xi}, \xi) \right] \nonumber \\
&= \EE\left[\frac{\partial_1 f(x, \pi^*_{x,\xi},\xi)}{\partial x} + \frac{\partial_2 f(x, \pi^*_{x,\xi},\xi)}{\partial \pi^*_{x,\xi}} \frac{\partial \pi^*_{x,\xi}}{\partial x}\right]\nonumber \\
&= \EE\left[\frac{\partial_1 f(x, \pi^*_{x,\xi},\xi)}{\partial x} + \sum_{s,a}\frac{\partial_2 f(x, \pi^*_{x,\xi},\xi)}{\partial \pi^*_{x,\xi}(a;s)} \frac{\partial \pi^*_{x,\xi}(a;s)}{\partial x}\right]\nonumber \\
&= \EE\left[\frac{\partial_1 f(x, \pi^*_{x,\xi},\xi)}{\partial x} + \sum_{s,a}\frac{\partial_2 f(x, \pi^*_{x,\xi},\xi)}{\partial \pi^*_{x,\xi}(a;s)} \frac{1}{\lambda} \pi^*_{x,\xi}(a;s) \partial_x{A^{\pi^*_{x,\xi}}_{\lambda,x,\xi}(s,a)}\right] \label{eq:hyper_proof_1}\\
&= \EE\left[\frac{\partial_1 f(x, \pi^*_{x,\xi},\xi)}{\partial x} + \EE_{s\sim \nu, a\sim \pi^*_{x,\xi}} \left[ \frac{1}{\lambda \nu(s)} \frac{\partial_2 f(x, \pi^*_{x,\xi},\xi)}{\partial \pi^*_{x,\xi}(a;s)}   \partial_x{A^{\pi^*_{x,\xi}}_{\lambda,x,\xi}(s,a)}\right] \right] \nonumber,
\end{align}
where we use \Cref{prop:policy_gradient_configuration} for \cref{eq:hyper_proof_1} and importance sampling with any distribution $\nu$ in the last equality, as long as $\nu$ has full support on $\cS$. Further, we note that $\frac{\partial_2 f(x, \pi^*_{x,\xi},\xi)}{\partial \pi^*_{x,\xi}} \in \text{Mat}_{1,|S|\times|A|}(\RR)$ and $\frac{\partial \pi^*_{x,\xi}}{\partial x}\in \text{Mat}_{|S|\times|A|,d}(\RR) $. Hence, we just explicitely write out the matrix multiplication for the second equality. The second equality follows from the multivariate chain rule and the first equality from the Dominated Convergence Theorem.
\end{proof}

\subsection{Proof of \texorpdfstring{\Cref{prop:policygradient}}{}}
\begin{proof}
\Cref{prop:policygradient} is important as $\partial_x{A^{\pi}_{\lambda,x,\xi}(s,a)}$ and thus $\partial_x{Q^{\pi}_{\lambda,x,\xi}(s,a)}$ allow us to compute $  \frac{d \pi^*_{x,\xi}}{d x}$ (cf. \Cref{prop:policy_gradient_configuration}). We will prove the theorem using an induction proof, which follows the analysis of the policy gradient theorem. However, instead of at each timestep $\pi(a_t;s_t)$ depending on a policy parameter $\theta$, it will be the transition $P_{x,\xi}(s_{t+1};s_t,a_t)$ and reward $r_{x,\xi}(s_t,a_t)$ depending on $x$. Also note that we only consider the partial derivative with respect to $x$, evaluated at some policy $\pi$, which in the case of $\pi^*_{x,\xi}$

In the following, we will show by induction that \[\frac{d{Q^{\pi}_{\lambda,x,\xi}(s,a)}}{dx}  = \sum_{t=0}^\infty \sum_{s',a'} \gamma^t p_{x,\xi}(s,a \rightarrow s',a';t,\pi) \left(\frac{d r_{x,\xi}(s',a') }{dx}+\gamma \sum_{s''} \frac{d P_{x,\xi}(s'';s',a') }{dx}V^{\pi}_{\lambda,x,\xi}(s'') \right),\]
where $p_{x,\xi}(s,a \rightarrow s',a';t,\pi)$ is the probability that the Markov Chain induced by $\pi$, starting from $s,a$ reaches $s',a'$ after $t$ steps. Note, the formulation is equivalent to the one stated in \Cref{prop:policygradient}.

The proof follows the analysis of the standard policy gradient theorem. We drop here the dependence on $x$ and $\xi$ to simplify the notation. Assuming that $Q^{\pi}_{\lambda}({s},{a})$ is differentiable for all $s,a$, we show by induction that for all $n \in \mathbb{N}$ it holds that
\begin{equation}
\begin{aligned}
    \frac{d{Q^{\pi}_{\lambda}(s,a)}}{dx}  =& \sum_{t=0}^n \sum_{s',a'} \gamma^t p(s,a \rightarrow s',a';t,\pi) \left(\frac{d r(s',a') }{dx}+ \gamma \sum_{s''} \frac{d P(s'';s',a') }{dx}V^{\pi}_{\lambda}(s'') \right) \label{eq:pgforinduction}\\
    &+\gamma^{n+1} \sum_{\Tilde{s},\Tilde{a}} p(s,a \rightarrow \Tilde{s},\Tilde{a};n+1,\pi) \frac{d{Q^{\pi}_{\lambda}(\Tilde{s},\Tilde{a})}}{dx}.
\end{aligned}
\end{equation}
The claim then follows as $n \rightarrow \infty$.
\paragraph{Base case $(n=0)$}
It is easy to check that
\begin{align*}
   \frac{d{Q^{\pi}_{\lambda}(s,a)}}{dx}  =& \frac{d}{dx} \left( r(s,a) + \gamma \sum_{s'} P(s';s,a) V_{\lambda}(s')\right) \\
   =& \frac{d}{dx}r(s,a) + \gamma \sum_{s'} \left(\frac{d}{dx}P(s';s,a) V^{\pi}_{\lambda}(s')+P(s';s,a) \frac{d}{dx} V^{\pi}_{\lambda}(s')\right) \\
   =& \frac{d}{dx}r(s,a) + \gamma \sum_{s'} \left(\frac{d}{dx}P(s';s,a) V^{\pi}_{\lambda}(s')+P(s';s,a) \sum_{a'} \pi(a';s')\frac{d}{dx} Q^{\pi}_{\lambda}(s',a')\right) \\
   =& \sum_{t=0}^0 \sum_{s',a'} \gamma^t p(s,a \rightarrow s',a';t,\pi) \left(\frac{d r(s',a') }{dx}+ \gamma \sum_{s''} \frac{d P(s'';s',a') }{dx}V^{\pi}_{\lambda}(s'') \right)\\
    & +\gamma^{1} \sum_{\Tilde{s},\Tilde{a}} p(s,a \rightarrow \Tilde{s},\Tilde{a};1,\pi) \frac{d{Q^{\pi}_{\lambda}(\Tilde{s},\Tilde{a})}}{dx}.
    \end{align*}
    
We use the definition of $Q^{\pi}_{\lambda}(s,a)$ in the first equality. The second follows by the product rule. The third equality follows by the definition of the value function. The last equality comes from rearranging terms.
\paragraph{Induction step $(n\implies n+1)$}
Assuming \cref{eq:pgforinduction} holds for $n$ we show it holds for $n+1$:
\begin{align*}
    \frac{d{Q^{\pi}_{\lambda}(s,a)}}{dx}  =& \sum_{t=0}^n \sum_{s',a'} \gamma^t p(s,a \rightarrow s',a';t,\pi) \left(\frac{d r(s',a') }{dx}+ \gamma \sum_{s''} \frac{d P(s'';s',a') }{dx}V^{\pi}_{\lambda}(s'') \right)\\
    &+\gamma^{n+1} \sum_{\Tilde{s},\Tilde{a}} p(s,a \rightarrow \Tilde{s},\Tilde{a};n+1,\pi) \frac{d{Q^{\pi}_{\lambda}(\Tilde{s},\Tilde{a})}}{dx}\\
    =& \sum_{t=0}^n \sum_{s',a'} \gamma^t p(s,a \rightarrow s',a';t,\pi) \left(\frac{d r(s',a') }{dx}+ \gamma \sum_{s''} \frac{d P(s'';s',a') }{dx}V^{\pi}_{\lambda}(s'') \right)\\
    &+\gamma^{n+1} \sum_{\Tilde{s},\Tilde{a}} p(s,a \rightarrow \Tilde{s},\Tilde{a};n+1,\pi) \frac{d}{dx}\left(r(\Tilde{s},\Tilde{a})+ \gamma \sum_{\tilde{s}'} P(\Tilde{s}';\Tilde{s},\Tilde{a}) V_\lambda(\Tilde{s}')\right)\\
    =& \sum_{t=0}^n \sum_{s',a'} \gamma^t p(s,a \rightarrow s',a';t,\pi) \left(\frac{d r(s',a') }{dx}+ \gamma \sum_{s''} \frac{d P(s'';s',a') }{dx}V^{\pi}_{\lambda}(s'') \right)\\
    &+\gamma^{n+1} \sum_{\Tilde{s},\Tilde{a}} p(s,a \rightarrow \Tilde{s},\Tilde{a};n+1,\pi) \Biggl(\frac{d}{dx}r(\Tilde{s},\Tilde{a})+ \gamma \sum_{\tilde{s}'} \frac{d}{dx}P(\Tilde{s}';\Tilde{s},\Tilde{a})V_\lambda(\Tilde{s}')\\
&+P(\Tilde{s}';\Tilde{s},\Tilde{a})\sum_{\tilde{a}'}\pi(\Tilde{a}';\Tilde{s}')\frac{d}{dx}Q_{\lambda}(\Tilde{s}',\Tilde{a}')\Biggr)\\
=&\sum_{t=0}^{n+1} \sum_{s',a'} \gamma^t p(s,a \rightarrow s',a';t,\pi) \left(\frac{d r(s',a') }{dx}+ \gamma \sum_{s''} \frac{d P(s'';s',a') }{dx}V^{\pi}_{\lambda}(s'') \right)\\
    &+\gamma^{n+2} \sum_{\Tilde{s},\Tilde{a}} p(s,a \rightarrow \Tilde{s},\Tilde{a};n+2,\pi) \frac{d{Q^{\pi}_{\lambda}(\Tilde{s},\Tilde{a})}}{dx}.\\
\end{align*}
The first equality is simply the base case. The second equality follows from the definition of the Q-function. The third equality follows from the multivariate chain rule and the definition of the value function and the last equality is again rearranging terms.
\end{proof}

\subsection{Proof of Theorem \texorpdfstring{\ref{thm:convergence}}{}}
\begin{proof}
By the smoothness of $f$, we use the following bound from \cite{hu2024contextual}[Lemma 1] for $\alpha\leq 1/(2S_f)$:
\begin{align}
    \EE\left[\norm{{\frac{dF(\hat{x}_T)}{dx}}}^2_\infty\right] 
    &\leq \underbrace{\frac{2 (F(x_1)-\min_x F(x))}{\alpha T}}_{\textcolor[RGB]{0,100,0}{\textbf{(1)}}}+\frac{2}{T} \sum_{t=1}^T \Biggl(L_f \underbrace{\norm{\EE\left[\frac{dF(x_t)}{dx}  - \widehat{\frac{dF(x_t)}{dx}}\right]}_\infty}_{\textcolor{red}{\textbf{(2)}}} \label{eq:graddecompositionbound}\\
    &\phantom{\leq} +S_f\alpha \underbrace{\EE\left[\norm{\frac{dF(x_t)}{dx} - \widehat{\frac{dF(x_t)}{dx}}}^2_\infty \right]}_{\textcolor{blue}{\textbf{(3)}}}\Biggr) \notag.
\end{align}

The error term naturally decomposes into an initial error divided by $T$ {\textcolor[RGB]{0,100,0}{\textbf{(1)}}}, a bias term {\textcolor{red}{\textbf{(2)}}}, and a variance term {\textcolor{blue}{\textbf{(3)}}}, which decreases with the stepsize $\alpha$.

For {\textcolor[RGB]{0,100,0}{\textbf{(1)}}} we do not need to simplify any further.

To prove our claim, we need to show that {\textcolor{blue}{\textbf{(3)}}} is $\cO(1)$, and that {\textcolor{red}{\textbf{(2)}}}  is $\cO(\delta)$. The latter is the main challenge of this proof. 

Let us begin by bounding the bias term {\textcolor{red}{\textbf{(2)}}}. The goal is to show that it can be upper bounded by a sum of terms, which are all linear in $\norm{\pi^o_{x,\xi}-\pi^*_{x,\xi}}_\infty$. 

\begin{align*}
    &\norm{\EE\left[\frac{dF(x_t)}{dx}  - \widehat{\frac{dF(x_t)}{dx}}\right]}_\infty \\
    =& \norm{\EE_{x_t}\left[ \frac{dF(x_t)}{dx} - \EE_{\xi,o}\left[ \frac{\partial_1 f(x_t, \pi^o_{x_t,\xi},\xi)}{\partial x}  + \EE^{\nu}_{a\sim \pi^o_{x_t,\xi}}\left[\frac{1}{\lambda \nu(s)} \frac{\partial_2 f(x_t, \pi^o_{x_t,\xi},\xi))}{\partial \pi(s,a)} \EE\left[\widehat{\frac{d}{dx} A^{\pi^o_{x_t,\xi}}_{\lambda,x,\xi}}(s,a)\right]\right] \right]\right]}_\infty \\
    \leq& \underbrace{\norm{\EE_{x_t,o,\xi}\left[\frac{\partial_1 f(x_t, \pi^*_{x_t,\xi},\xi)}{\partial x}-\frac{\partial_1 f(x_t, \pi^o_{x_t,\xi},\xi)}{\partial x}\right]}_\infty}_{\textcolor{cyan}{\textbf{(A)}}} \\
    &+ \underbrace{\norm{\EE_{x_t,o}^{\xi,\nu}\left[ \frac{1}{\lambda \nu(s)} \sum_{a} \left( \pi^*_{x,\xi}(a;s) \frac{\partial_2 f(x, \pi^*_{x,\xi},\xi)}{\partial \pi^*_{x,\xi}(a;s)}   \partial_x{A^{\pi^*_{x,\xi}}_{\lambda,x,\xi}(s,a)} - \pi^o_{x,\xi}(a;s)\frac{\partial_2 f(x, \pi^o_{x_t,\xi},\xi)}{\partial \pi^o_{x_t,\xi}(a;s)}   \partial_x{A^{\pi^o_{x_t,\xi}}_{\lambda,x,\xi}(s,a)} \right)\right] }_{\infty}}_{\textcolor{magenta}{\textbf{(B)}}},
\end{align*}
where the first equality is by definition and the first inequality follows from the trinalge inequality and we further use the fact that $\widehat{\frac{d}{dx} A^{\pi^o_{x_t,\xi}}_{\lambda,x,\xi}}$ is an unbiased estimator of ${\frac{d}{dx} A^{\pi^o_{x_t,\xi}}_{\lambda,x,\xi}}$ as shown in  \Cref{prop:unbiasedgradient}.

\textbf{\textcolor{cyan}{\textbf{{(A)}}}} is relatively easy to bound. Indeed by the smoothness of $f$ ( \Cref{assum:theory}), it immediately follows that
\begin{equation*}
    {\textcolor{cyan}{\textbf{(A)}}} \leq \EE_{x_t,o,\xi}\left[S_f \norm{\pi^*_{x_t,\xi}-\pi^o_{x_t,\xi}}_\infty\right] \leq S_f \delta.
\end{equation*}

To bound {\textcolor{magenta}{\textbf{(B)}}} we again use the triangle inequality to decompose:
\begin{align*}
   {\textcolor{magenta}{\textbf{(B)}}}=& \norm{\EE_{x_t,o}^{\xi,\nu}\left[ \frac{1}{\lambda \nu(s)} \sum_{a} \left( \pi^*_{x,\xi}(a;s) \frac{\partial_2 f(x, \pi^*_{x,\xi},\xi)}{\partial \pi^*_{x,\xi}(a;s)}   \partial_x{A^{\pi^*_{x,\xi}}_{\lambda,x,\xi}(s,a)} - \pi^o_{x,\xi}(a;s)\frac{\partial_2 f(x, \pi^o_{x_t,\xi},\xi)}{\partial \pi^o_{x_t,\xi}(a;s)}   \partial_x{A^{\pi^o_{x_t,\xi}}_{\lambda,x,\xi}(s,a)} \right)\right] }_{\infty}\\
    \leq& \underbrace{\EE_{x_t,o}^{\xi,\nu}\Biggl[ \frac{1}{\lambda \nu(s)} \sum_{a}  \norm{\pi^*_{x,\xi}(a;s)-\pi^o_{x,\xi}(a;s)}_\infty \norm{\frac{\partial_2 f(x, \pi^*_{x,\xi},\xi)}{\partial \pi^*_{x,\xi}(a;s)}   \partial_x{A^{\pi^*_{x,\xi}}_{\lambda,x,\xi}(s,a)}}_{\infty}\Biggr]}_{\textcolor{lime}{\textbf{(a)}}} \\
    &+ \underbrace{\EE_{x_t,o}^{\xi,\nu}\Biggl[ \frac{1}{\lambda \nu(s)} \sum_a \norm{\pi^o_{x,\xi}(a;s)}_{\infty} \norm{\frac{\partial_2 f(x, \pi^*_{x_t,\xi},\xi)}{\partial \pi^*_{x_t,\xi}(a;s)}   \partial_x{A^{\pi^*_{x_t,\xi}}_{\lambda,x,\xi}(s,a)} -\frac{\partial_2 f(x, \pi^o_{x_t,\xi},\xi)}{\partial \pi^o_{x_t,\xi}(a;s)}   \partial_x{A^{\pi^o_{x_t,\xi}}_{\lambda,x,\xi}(s,a)} }_{\infty}\Biggr]}_{\textcolor{orange}{\textbf{(b)}}}.
\end{align*}

\textcolor{lime}{\textbf{(a)}} is relatively easy to bound. We have
\begin{align*}
    \textcolor{lime}{\textbf{(a)}} 
&\leq 
    \EE_{x_t}^{\xi,\nu}\Biggl[ \frac{1}{\lambda \nu(s)} |\mathcal{A}|  \delta \norm{\frac{\partial_2 f(x, \pi^*_{x,\xi},\xi)}{\partial \pi^*_{x,\xi}(a;s)}   \partial_x{A^{\pi^*_{x,\xi}}_{\lambda,x,\xi}(s,a)}}_{\infty}\Biggr]\\
&\leq 
    \EE_{x_t}^{\xi,\nu}\Biggl[ \frac{1}{\lambda \nu(s)} |\mathcal{A}|  \delta L_f \norm{ \partial_x{A^{\pi^*_{x,\xi}}_{\lambda,x,\xi}(s,a)}}_{\infty}\Biggr]\\
    &\leq \EE_{x_t}^{\xi,\nu}\Biggl[ \frac{2}{\lambda \nu(s)} |\mathcal{A}|  \delta L_f 
\norm{ \partial_x{Q^{\pi^*_{x,\xi}}_{\lambda,x,\xi}(s,a)}}_{\infty}\Biggr],\\
\end{align*}
where we use the assumption on the oracle in the first inequality, that $f$ is Lipschitz continuous in the second inequality, and that $\norm{ \partial_x{V^{\pi^*_{x,\xi}}_{\lambda,x,\xi}(s)}}_{\infty}\leq \norm{ \partial_x{Q^{\pi^*_{x,\xi}}_{\lambda,x,\xi}(s,a)}}_{\infty}$ in the third inequality. Note that: 
\begin{equation*}
    \norm{V^{\pi}_{\lambda,x,\xi}(s)}_{\infty} \leq \frac{(\overline{R}+\lambda \log |\mathcal{A}|)}{1-\gamma},
\end{equation*} 
since the entropy of any policy is bound by $\log|A|$ (which follows from Jensen's inequality).
Using the definition that \[
\partial_x{Q^{\pi^*_{x,\xi}}_{\lambda,x,\xi}(s,a)}=\EE_{s,a}^{\pi^*_{x,\xi}} \left[ \sum_{t=0}^\infty  \gamma^t \frac{d r_{x,\xi}(s_t,a_t)}{dx} + \gamma^{t+1} \frac{d \log P_{x,\xi}(s_{t+1};s_t,a_t)}{dx}V^{\pi}_{\lambda,x,\xi}(s_{t+1})  \right],\] it thus holds that
\begin{equation}
\label{eq:boundqderivative}
    \norm{ \partial_x{Q^{\pi^*_{x,\xi}}_{\lambda,x,\xi}(s,a)}}_{\infty} \leq \left(\frac{K_2}{1-\gamma}+\frac{K_1(\overline{R}+\lambda \log |\mathcal{A}|)}{(1-\gamma)^2}\right).
\end{equation} 
Letting $m:=\min_s \nu(s)$, we thus have 
\begin{align*}
    \textcolor{lime}{\textbf{(a)}}  \leq \frac{2}{\lambda m} |\mathcal{A}|  \delta L_f \left(\frac{K_2}{1-\gamma}+\frac{K_1(\overline{R}+\lambda \log |\mathcal{A}|)}{(1-\gamma)^2}\right).
\end{align*}
For {\textcolor{orange}{\textbf{(b)}}} we further simplify using the triangle inequallity:
\begin{align*}
    {\textcolor{orange}{\textbf{(b)}}}  \leq& \frac{1}{\lambda m} \EE_{x_t,o}^{\xi}\Biggl[  \norm{\frac{\partial_2 f(x, \pi^*_{x_t,\xi},\xi)}{\partial \pi^*_{x_t,\xi}(a;s)}   \partial_x{A^{\pi^*_{x_t,\xi}}_{\lambda,x,\xi}(s,a)} -\frac{\partial_2 f(x, \pi^o_{x_t,\xi},\xi)}{\partial \pi^o_{x_t,\xi}(a;s)}   \partial_x{A^{\pi^o_{x_t,\xi}}_{\lambda,x,\xi}(s,a)} }_{\infty}\Biggr]\\
    \leq& \underbrace{\frac{1}{\lambda m} \EE_{x_t,o}^{\xi}\left[  \norm{\frac{\partial_2 f(x, \pi^*_{x_t,\xi},\xi)}{\partial \pi^*_{x_t,\xi}(a;s)}  -\frac{\partial_2 f(x, \pi^o_{x_t,\xi},\xi)}{\partial \pi^o_{x_t,\xi}(a;s)}}_{\infty} \norm{\partial_x{A^{\pi^*_{x_t,\xi}}_{\lambda,x,\xi}(s,a)} }_{\infty}\right]}_{\textcolor{brown}{\textbf{(i)}}}\\
    &+ \underbrace{\frac{1}{\lambda m} \EE_{x_t,o}^{\xi}\left[\norm{\frac{\partial_2 f(x, \pi^o_{x_t,\xi},\xi)}{\partial \pi^o_{x_t,\xi}(a;s)}   }_\infty \norm{\partial_x{A^{\pi^*_{x_t,\xi}}_{\lambda,x,\xi}(s,a)}-\partial_x{A^{\pi^o_{x_t,\xi}}_{\lambda,x,\xi}(s,a)}  }_{\infty}\right]}_{\textcolor{purple}{\textbf{(ii)}}}.
\end{align*}
Similar to \textcolor{lime}{\textbf{(a)}}, we can bound {\textcolor{brown}{\textbf{(i)}}} using the smoothness of $f$ (Assumption \ref{assum:theory}):
\begin{align*}
    {\textcolor{brown}{\textbf{(i)}}} \leq 2 \frac{S_f}{\lambda m} \delta  \left(\frac{K_2}{1-\gamma}+\frac{K_1(\overline{R}+\lambda \log |\mathcal{A}|)}{(1-\gamma)^2}\right).
\end{align*}

Bounding {\textcolor{purple}{\textbf{(ii)}}} and in particular $\norm{\partial_x{A^{\pi^*_{x_t,\xi}}_{\lambda,x,\xi}(s,a)}-\partial_x{A^{\pi^o_{x_t,\xi}}_{\lambda,x,\xi}(s,a)}  }_{\infty}$  is the tricky part of this proof. We first need to show two intermediate results. First, we bound the difference in entropy between two policies and the difference in the regularized value functions of two policies. Once we have that, we can tackle $\norm{\partial_x{A^{\pi^*_{x_t,\xi}}_{\lambda,x,\xi}(s,a)}-\partial_x{A^{\pi^o_{x_t,\xi}}_{\lambda,x,\xi}(s,a)}  }_{\infty}$. For the entropy, we denote by $l_1 := \min_{s,a,x,\xi} \pi^*_{x,\xi}(a;s)$ the minimum probability of playing an action in any state under the optimal policy. Recall that 
$\forall x,\xi,s,a : |r_{x,\xi}(s,a)|<\overline{R}$ and that $0\leq H(\pi;s)\leq \log|\cA|$ (by Jensen's inequality).
Thus we have
\[
\frac{-\overline{R}}{1-\gamma}\leq Q^*_{\lambda,x,\xi}(s,a)\leq \frac{\overline{R}+\lambda \log |\cA|}{1-\gamma}.
\]
Because $\pi^*_{x,\xi}(s;a)\propto \exp(Q^*_{\lambda,x,\xi}(s,a)/\lambda)$, it follows that
\[
l_1 \geq \frac{\exp(\frac{-\overline{R}}{\lambda(1-\gamma)})}{|\mathcal{A}|\exp(\frac{\overline{R}+\lambda \log|\cA|}{\lambda(1-\gamma)})}>0.
\]
We assume now that $\delta$ is sufficiently small, i.e. $\delta \leq l_1/2$, such that $l_1/2 \leq \min_{s,a,x,\xi} \pi^o_{x,\xi}(a;s)$. We need to have such a lower bound on the policies, touse the fact that the $\log$ function is Lipschitz continuous with parameter $\frac{1}{a}$ on an interval $[a,\infty)$ for any $a>0$. Hence we have 
\begin{align*}
    \norm{H(\pi^*_{x,\xi}|s)-H(\pi^o_{x,\xi}|s)}_{\infty}    &= \norm{\sum_a \pi^*_{x,\xi}(a;s)\log  \pi^*_{x,\xi}(a;s) - \sum_a \pi^o_{x,\xi}(a;s)\log  \pi^o_{x,\xi}(a;s)}_\infty\\
    &\leq \left(\sum_a \norm{\pi^*_{x,\xi} - \pi^o_{x,\xi}}_\infty \norm{\log \pi^*_{x,\xi}}_\infty\right) + \norm{\log \pi^*_{x,\xi}-\log \pi^o_{x,\xi}}_\infty \\
    &\leq |\mathcal{A}||\log l_1| \delta + \frac{2}{l_1}\delta.
\end{align*}

We use this to bound the difference in the value functions of $\pi^*_{x,\xi}$ and $\pi^o_{x,\xi}$.
\begin{equation}
\label{eq:bound_v_difference}
\begin{aligned}
    &\norm{V_\lambda^{\pi^*_{x,\xi}}(s)-V_{\lambda}^{\pi^o_{x,\xi}}(s)}_{\infty} \\
    \leq&
    \norm{\sum_a\left(\pi^*_{x,\xi}(a;s)Q_\lambda^{\pi^*_{x,\xi}}(s,a)  - \pi^o_{x,\xi}(a;s)Q_\lambda^{\pi^o_{x,\xi}}(s,a)  \right)}_{\infty} + \lambda \norm{H(\pi^*_{x,\xi}|s)-H(\pi^o_{x,\xi}|s)}_{\infty}\\
    \leq& \norm{\sum_a\left(\pi^*_{x,\xi}(a;s)Q_\lambda^{\pi^*_{x,\xi}}(s,a)  - \pi^o_{x,\xi}(a;s)Q_\lambda^{\pi^o_{x,\xi}}(s,a)  \right)}_{\infty} +\lambda \delta\left(|\mathcal{A}||\log l_1|  + \frac{2}{l_1}\right)\\
    \leq& \norm{\sum_a\pi^*_{x,\xi}(a;s)}_\infty\norm{Q_\lambda^{\pi^*_{x,\xi}}(s,a)  - Q_\lambda^{\pi^o_{x,\xi}}(s,a)  }_{\infty} \\&+ \sum_a \norm{\pi^*_{x,\xi}(a;s)-\pi^o_{x,\xi}(a;s)}_\infty \norm{Q_\lambda^{\pi^o_{x,\xi}}(s,a) }_\infty + \lambda \delta\left(|\mathcal{A}||\log l_1|  + \frac{2}{l_1}\right)\\
   \leq & \lambda \delta\left(|\mathcal{A}||\log l_1|  + \frac{2}{l_1}\right)+ \delta |\mathcal{A}|\frac{\overline{R}}{1-\gamma} + \gamma \norm{V_\lambda^{\pi^*_{x,\xi}}(s)  - V_\lambda^{\pi^o_{x,\xi}}(s)  }_{\infty}\\
   \leq& \frac{\lambda \delta\left(|\mathcal{A}||\log l_1|  + \frac{2}{l_1}\right)}{1-\gamma}+ \frac{\delta |\mathcal{A}|\overline{R}}{(1-\gamma)^2} .
\end{aligned}
\end{equation}
The first inequality follows from the definition of the regularized value function and the triangle inequality. The second inequality follows from our derived bound on the difference of entropies. The third  inequality is agian using the triangle inequality and the fourth inequality is bounding the Q-function using that the rewards are upper bounded by $\overline{R}$ and that $ Q^\pi_{\lambda,x,\xi}(s,a)=r_{x,\xi}(s,a)+\gamma\EE_{s'\sim P_{x,\xi}(\cdot;s,a)}\sbr{V^\pi_{\lambda,x,\xi}(s')}$. The last inequality then follows by iteratively plugging in the inequality for the term $\norm{V_\lambda^{\pi^*_{x,\xi}}(s)  - V_\lambda^{\pi^o_{x,\xi}}(s)  }_{\infty}$, which gives a geometric sum.
We employ a similar technique to bound {\textcolor{purple}{\textbf{(ii)}}} using the above results:

\begin{align*}
    &\frac{1}{\lambda m} \EE_{x_t,o}^{\xi}\left[\norm{\frac{\partial_2 f(x, \pi^o_{x_t,\xi},\xi)}{\partial \pi^o_{x_t,\xi}(a;s)}   }_\infty \norm{\partial_x{A^{\pi^*_{x_t,\xi}}_{\lambda,x,\xi}(s,a)}-\partial_x{A^{\pi^o_{x_t,\xi}}_{\lambda,x,\xi}(s,a)}  }_{\infty}\right]\\
    \leq& \frac{L_f}{\lambda m} \EE_{x_t,o}^{\xi}\left[\norm{\partial_x{A^{\pi^*_{x_t,\xi}}_{\lambda,x,\xi}(s,a)}-\partial_x{A^{\pi^o_{x_t,\xi}}_{\lambda,x,\xi}(s,a)}  }_{\infty}\right]\\
    \leq& \frac{L_f}{\lambda m} 2 \EE_{x_t,o}^{\xi}\left[ \norm{\partial_x{Q^{\pi^*_{x_t,\xi}}_{\lambda,x,\xi}(s,a)}-\partial_x{Q^{\pi^o_{x_t,\xi}}_{\lambda,x,\xi}(s,a)}  }_{\infty}+|\cA|\left(\frac{K_2}{1-\gamma}+\frac{K_1(\overline{R}+\lambda \log |\mathcal{A}|)}{(1-\gamma)^2}\right)\norm{\pi^o_{x,\xi}-\pi^*_{x,\xi}}_\infty\right].
\end{align*}
Here the first inequality follows from the Lipschitz continuity of $f$. For second inequality we use the definition of the advantage function, the triangle inequality and the following inequality:
\begin{equation}
\label{eq:diffvsmallerdiffq}
    \begin{aligned}
        &\norm{\partial_x V^{\pi^o_{x_t,\xi}}_{\lambda,x,\xi}(s)-\partial_x{V^{\pi^*_{x_t,\xi}}_{\lambda,x,\xi}(s)}}_\infty\\
        =&\norm{\sum_a \pi^o_{x_t,\xi}(a;s) \partial_x  Q^{\pi^o_{x_t,\xi}}_{\lambda,x,\xi}(s,a)-\sum_a \pi^*_{x_t,\xi}(a;s) \partial_x{Q^{\pi^*_{x_t,\xi}}_{\lambda,x,\xi}(s,a)}}_\infty\\
        =& \norm{\sum_a \left(\pi^o_{x_t,\xi}(a;s)-\pi^*_{x_t,\xi}(a;s)\right) \partial_x  Q^{\pi^o_{x_t,\xi}}_{\lambda,x,\xi}(s,a)-\sum_a \pi^*_{x_t,\xi}(a;s) \left(\partial_x{Q^{\pi^*_{x_t,\xi}}_{\lambda,x,\xi}(s,a)}-\partial_x{Q^{\pi^o_{x_t,\xi}}_{\lambda,x,\xi}(s,a)}\right)}_\infty\\
        \leq &|\mathcal{A}|\frac{K_1(\overline{R}+\lambda \log |\mathcal{A}|)}{(1-\gamma)^2} \norm{\pi^o_{x_t,\xi}-\pi^*_{x_t,\xi}}_\infty+\norm{\partial_x{Q^{\pi^*_{x_t,\xi}}_{\lambda,x,\xi}(s,a)}-\partial_x{Q^{\pi^o_{x_t,\xi}}_{\lambda,x,\xi}(s,a)}}_\infty
        ,
    \end{aligned}
\end{equation}
where the last inequality follows from the trianlge inequality and \Cref{eq:boundqderivative}.

We bound the difference in Q-function derivatives as follows: 
\begin{align*}
    & \norm{\partial_x{Q^{\pi^*_{x_t,\xi}}_{\lambda,x,\xi}(s,a)}-\partial_x{Q^{\pi^o_{x_t,\xi}}_{\lambda,x,\xi}(s,a)}  }_{\infty} \\
    =&\Bigg|\Bigg|\sum_{t=0}^\infty \sum_{s',a'} \gamma^t p_{x,\xi}(s,a \rightarrow s',a';t,\pi^*_{x_t,\xi}) \left(\frac{d r_{x,\xi}(s',a') }{dx}+\gamma \sum_{s''} \frac{d P_{x,\xi}(s'';s',a') }{dx}V^{\pi^*_{x_t,\xi}}_{\lambda,x,\xi}(s'') \right) \\
    &- \sum_{t=0}^\infty \sum_{s',a'} \gamma^t p_{x,\xi}(s,a \rightarrow s',a';t,\pi^o_{x_t,\xi}) \left(\frac{d r_{x,\xi}(s',a') }{dx}+\gamma \sum_{s''} \frac{d P_{x,\xi}(s'';s',a') }{dx}V^{\pi^o_{x_t,\xi}}_{\lambda,x,\xi}(s'') \right) \Bigg|\Bigg|_{\infty}  \\
        =&\Bigg|\Bigg|\frac{d r_{x,\xi}(s,a) }{dx} + \gamma \sum_{s'} \frac{d P_{x,\xi}(s';s,a) }{dx}V^{\pi^*_{x_t,\xi}}_{\lambda,x,\xi}(s')\\
        &+ \sum_{t=1}^\infty \sum_{s',a'} \gamma^t p_{x,\xi}(s,a \rightarrow s',a';t,\pi^*_{x_t,\xi}) \left(\frac{d r_{x,\xi}(s',a') }{dx}+\gamma \sum_{s''} \frac{d P_{x,\xi}(s'';s',a') }{dx}V^{\pi^*_{x_t,\xi}}_{\lambda,x,\xi}(s'') \right) \\
        &- \frac{d r_{x,\xi}(s,a) }{dx} - \gamma \sum_{s'} \frac{d P_{x,\xi}(s';s,a) }{dx}V^{\pi^o_{x_t,\xi}}_{\lambda,x,\xi}(s')\\
    &- \sum_{t=1}^\infty \sum_{s',a'} \gamma^t p_{x,\xi}(s,a \rightarrow s',a';t,\pi^o_{x_t,\xi}) \left(\frac{d r_{x,\xi}(s',a') }{dx}+\gamma \sum_{s''} \frac{d P_{x,\xi}(s'';s',a') }{dx}V^{\pi^o_{x_t,\xi}}_{\lambda,x,\xi}(s'') \right) \Bigg|\Bigg|_{\infty}  \\
    \leq& \gamma \sum_{s'} \norm{\frac{d P_{x,\xi}(s';s,a) }{dx}}_\infty \norm{V^{\pi^o_{x_t,\xi}}_{\lambda,x,\xi}(s')-V^{\pi^o_{x_t,\xi}}_{\lambda,x,\xi}(s') }_{\infty}\\
    &+ \gamma \Bigg|\Bigg| \sum_{s',a'}  P(s';s,a)  \pi^*_{x_t,\xi}(a',s') \sum_{t=0}^\infty \sum_{s',a'} \gamma^t p_{x,\xi}(s',a' \rightarrow s'',a'';t,\pi^*_{x_t,\xi}) \dots \\
    & \left(\frac{d r_{x,\xi}(s'',a'') }{dx}+\gamma \sum_{s'''} \frac{d P_{x,\xi}(s''';s'',a'') }{dx}V^{\pi^*_{x_t,\xi}}_{\lambda,x,\xi}(s'') \right)\\
    &-   \sum_{s',a'}  P(s';s,a)  \pi^o_{x_t,\xi}(a',s') \sum_{t=0}^\infty \sum_{s',a'} \gamma^t p_{x,\xi}(s',a' \rightarrow s'',a'';t,\pi^o_{x_t,\xi}) \dots \\
    & \left(\frac{d r_{x,\xi}(s'',a'') }{dx}+\gamma \sum_{s'''} \frac{d P_{x,\xi}(s''';s'',a'') }{dx}V^{\pi^o_{x_t,\xi}}_{\lambda,x,\xi}(s'') \right) \Bigg|\Bigg|_{\infty}\\
\leq& \gamma \sum_{s'} \norm{\frac{d P_{x,\xi}(s';s,a) }{dx}}_\infty \norm{V^{\pi^o_{x_t,\xi}}_{\lambda,x,\xi}(s')-V^{\pi^o_{x_t,\xi}}_{\lambda,x,\xi}(s') }_{\infty}\\
    &+ \gamma \Bigg|\Bigg| \sum_{s',a'}  P(s';s,a)  \pi^*_{x_t,\xi}(a',s')  \partial_x{Q^{\pi^*_{x_t,\xi}}_{\lambda,x,\xi}(s,a)}\\
    &-   \sum_{s',a'}  P(s';s,a)  \pi^o_{x_t,\xi}(a',s')  \partial_x{Q^{\pi^o_{x_t,\xi}}_{\lambda,x,\xi}(s,a)} \Bigg|\Bigg|_{\infty}\\    
    \leq& \gamma \sum_{s'} \norm{\frac{d P_{x,\xi}(s';s,a) }{dx}}_\infty \norm{V^{\pi^o_{x_t,\xi}}_{\lambda,x,\xi}(s')-V^{\pi^o_{x_t,\xi}}_{\lambda,x,\xi}(s') }_{\infty}\\
    &+\gamma \norm{\partial_xQ^{\pi^*_{x_t,\xi}}_{\lambda,x,\xi}(s',a')}_{\infty}\norm{\sum_{s',a'}  P(s';s,a) }_\infty \norm{ \pi^*_{x_t,\xi}(a',s')-  \pi^o_{x_t,\xi}(a',s')}_\infty\\
    &+ \gamma \sum_{s',a'}  P(s';s,a)  \pi^o_{x_t,\xi}(a',s') \norm{\partial_x{Q^{\pi^*_{x_t,\xi}}_{\lambda,x,\xi}(s',a')}-\partial_x{Q^{\pi^o_{x_t,\xi}}_{\lambda,x,\xi}(s',a')}  }_\infty,
    \end{align*}
    where the dots indicate multiplication over the linebreak. The first equality follows from plugging in the result from \Cref{prop:policygradient}. The second equality follows by taking out all terms with $t=0$. The first inequality uses the triangle inequality. The second inequality plugs back in the definition from \Cref{prop:policygradient}. The last inequality follows from the triangle inequality again. 
    
    Taking the expectation, we thus get:
    \begin{align*}
     &\EE_{x_t,o}^{\xi}\left[ \norm{\partial_x{Q^{\pi^*_{x_t,\xi}}_{\lambda,x,\xi}(s,a)}-\partial_x{Q^{\pi^o_{x_t,\xi}}_{\lambda,x,\xi}(s,a)}  }_{\infty}\right]\\
    \leq& \gamma \left( |\mathcal{S}|K_1\left(\frac{\lambda \delta\left(|\mathcal{A}||\log l_1|  + \frac{2}{l_1}\right)}{1-\gamma}+ \frac{\delta |\mathcal{A}|\overline{R}}{(1-\gamma)^2}  \right) + \delta \left(\frac{K_2}{1-\gamma}+\frac{K_1(\overline{R}+\lambda \log |\mathcal{A}|)}{(1-\gamma)^2}\right)\right) \\
    &+\gamma \EE_{x_t,o}^{\xi}\left[ \norm{\partial_x{Q^{\pi^*_{x_t,\xi}}_{\lambda,x,\xi}(s',a')}-\partial_x{Q^{\pi^o_{x_t,\xi}}_{\lambda,x,\xi}(s',a')}  }_\infty \right]\\
    \leq& \frac{\gamma}{1-\gamma} \left( |\mathcal{S}|K_1\left(\frac{\lambda \delta\left(|\mathcal{A}||\log l_1|  + \frac{2}{l_1}\right)}{1-\gamma}+ \frac{\delta |\mathcal{A}|\overline{R}}{(1-\gamma)^2}  \right) + \delta \left(\frac{K_2}{1-\gamma}+\frac{K_1(\overline{R}+\lambda \log |\mathcal{A}|)}{(1-\gamma)^2}\right)\right)\\
    =& \frac{\delta \gamma}{1-\gamma} \left( |\mathcal{S}|K_1\left(\frac{\lambda \left(|\mathcal{A}||\log l_1|  + \frac{2}{l_1}\right)}{1-\gamma}+ \frac{ |\mathcal{A}|\overline{R}}{(1-\gamma)^2}  \right) +  \left(\frac{K_2}{1-\gamma}+\frac{K_1(\overline{R}+\lambda \log |\mathcal{A}|)}{(1-\gamma)^2}\right)\right),
\end{align*}
 where we use the intermediate result from before to bound the difference between the value functions, the upper bound on the Q-function derivative shown in \Cref{eq:boundqderivative} and the assumption on the oracle to get the first inequality. The second inequality follows from the the resulting geometric sum and the last equality is just rearranging terms to show the linearity in $\delta$.

Using this result, we can now bound {\textcolor{purple}{\textbf{(ii)}}}:
\begin{equation*}
\begin{aligned}
       {\textcolor{purple}{\textbf{(ii)}}} \leq&  \frac{2 L_f \delta \gamma}{\lambda m(1-\gamma)} \Bigg( |\mathcal{S}|K_1\left(\frac{\lambda \left(|\mathcal{A}||\log l_1|  + \frac{2}{l_1}\right)}{1-\gamma}+ \frac{ |\mathcal{A}|\overline{R}}{(1-\gamma)^2}  \right) +  \left(\frac{K_2}{1-\gamma}+\frac{K_1(\overline{R}+\lambda \log |\mathcal{A}|)}{(1-\gamma)^2}\right)\\
    &+|\cA|\left(\frac{K_2}{\gamma}+\frac{K_1(\overline{R}+\lambda \log |\mathcal{A}|)}{(1-\gamma)\gamma}\right)\Bigg).
\end{aligned}
\end{equation*}
With that we are done decomposing {\textcolor{red}{\textbf{(2)}}}. Combining everything, we have the following bound:
\begin{align*}
    {\textcolor{red}{\textbf{(2)}}} =&\norm{\EE\left[\frac{dF(x_t)}{dx}  - \widehat{\frac{dF(x_t)}{dx}}\right]}_\infty\\
    \leq &{\textcolor{cyan}{\textbf{(A)}}}+ {\textcolor{magenta}{\textbf{(B)}}} \\
    \leq & {\textcolor{cyan}{\textbf{(A)}}}+ {\textcolor{lime}{\textbf{(a)}}}+{\textcolor{orange}{\textbf{(b)}}}\\
    \leq& {\textcolor{cyan}{\textbf{(A)}}}+ {\textcolor{lime}{\textbf{(a)}}}+{\textcolor{brown}{\textbf{(i)}}}+ {\textcolor{purple}{\textbf{(ii)}}}\\
    \leq&  S_f \delta + \frac{2}{\lambda m} |\mathcal{A}|  \delta L_f \left(\frac{K_2}{1-\gamma}+\frac{K_1(\overline{R}+\lambda \log |\mathcal{A}|)}{(1-\gamma)^2}\right)+ \frac{2S_f}{\lambda m} \delta  \left(\frac{K_2}{1-\gamma}+\frac{K_1(\overline{R}+\lambda \log |\mathcal{A}|)}{(1-\gamma)^2}\right)\\
    &+ \frac{2 L_f \delta \gamma}{\lambda m(1-\gamma)} \Bigg( |\mathcal{S}|K_1\left(\frac{\lambda \left(|\mathcal{A}||\log l_1|  + \frac{2}{l_1}\right)}{1-\gamma}+ \frac{ |\mathcal{A}|\overline{R}}{(1-\gamma)^2}  \right) +  \left(\frac{K_2}{1-\gamma}+\frac{K_1(\overline{R}+\lambda \log |\mathcal{A}|)}{(1-\gamma)^2}\right)\\
    &+|\cA|\left(\frac{K_2}{\gamma}+\frac{K_1(\overline{R}+\lambda \log |\mathcal{A}|)}{(1-\gamma)\gamma}\right)\Bigg)\\
    =& \mathcal{O}(\delta).
\end{align*}

With that we have tackled terms {\textcolor[RGB]{0,100,0}{\textbf{(1)}}} and {\textcolor{red}{\textbf{(2)}}} in \Cref{eq:graddecompositionbound}. It remains to bound the variance, i.e. term {\textcolor{blue}{\textbf{(3)}}}. If we can show that ${\textcolor{blue}{\textbf{(3)}}}=\cO(1)$ then the last term in \Cref{eq:graddecompositionbound} is $\cO(\alpha)$ as claimed and we are done. Indeed, bounding {\textcolor{blue}{\textbf{(3)}}} is relatively easy, as all important terms are bounded by \Cref{assum:theory}. We have:

\begin{align*}
    \EE\left[\norm{\frac{dF(x_t)}{dx} - \widehat{\frac{dF(x_t)}{dx}}}^2_\infty \right]
    \leq& 2 \norm{\frac{dF(x_t)}{dx} - \EE\left[{\widehat{\frac{dF(x_t)}{dx}}}\right]}^2_\infty + 2 \EE\left[\norm{ \widehat{\frac{dF(x_t)}{dx}}-\EE\left[{\widehat{\frac{dF(x_t)}{dx}}}\right]}^2_\infty\right]\\
    \leq &2 \mathcal{O}(\delta^2) + 2 \EE\left[\norm{ \widehat{\frac{dF(x_t)}{dx}}-\EE\left[{\widehat{\frac{dF(x_t)}{dx}}}\right]}^2_\infty\right]\\
    \leq &\mathcal{O}(\delta^2) + 2 \left(\EE\left[\norm{ \widehat{\frac{dF(x_t)}{dx}}}^2_\infty\right] -\norm{\EE\left[{\widehat{\frac{dF(x_t)}{dx}}}\right]}^2_\infty\right)\\
    \leq& \mathcal{O}(\delta^2) + 2 \EE\left[\norm{ \widehat{\frac{dF(x_t)}{dx}}}^2_\infty\right],
\end{align*}
where the third inequality follows directly, the second inequality uses the definition of the variance and the first inequality uses that \[\norm{a+b}^2= \norm{a}^2+2a^\top b+\norm{b}^2\leq\norm{a}^2+2\norm{a} \norm{b}+\norm{b}^2\leq 2\norm{a}+2\norm{b}^2.\]

By the above, it suffices to bound the second moment:
\begin{align*}
    \EE\left[\norm{ \widehat{\frac{dF(x_t)}{dx}}}^2_\infty\right]&\leq  \EE\left[\norm{ \frac{\partial_1 f(x_t, \pi^o_{x_t,\xi},\xi)}{\partial x}  + \frac{1}{\lambda \nu(s)} \frac{\partial_2 f(x_t, \pi^o_{x_t,\xi},\xi))}{\partial \pi(s,a)} \widehat{\partial_x A^{\pi^o_{x_t,\xi}}_{\lambda,x,\xi}}(s,a)}^2_\infty\right]\\
    &\leq  \EE\left[\norm{ L_f  + \frac{1}{\lambda m} L_f  \widehat{\partial_x A^{\pi^o_{x_t,\xi}}_{\lambda,x,\xi}}(s,a)}^2_\infty\right].\\
\end{align*}
To proceed, we upper bound $\EE\left[\norm{\widehat{\partial_x A^{\pi^o_{x_t,\xi}}_{\lambda,x,\xi}}(s,a)}_\infty^2\right]$ by
\begin{equation}
\begin{aligned}
 & \EE\left[\norm{\widehat{\partial_x A^{\pi^o_{x_t,\xi}}_{\lambda,x,\xi}}(s,a)}_\infty^2\right]\\
\leq  & 4\EE\Bigg[\Bigg\|\Bigg(\sum_{t=0}^{T_Q} \frac{d}{dx} r(s^{\tau_Q}_t,a^{\tau_Q}_t)\\
&+\frac{\gamma}{1-\gamma} \frac{d}{dx} \log P(s^{\tau_Q}_{T_Q+1};s^{\tau_Q}_{T_Q},a^{\tau_Q}_{T_Q}) \sum_{t=T_Q+1}^{T_Q+T'_Q+1} \gamma^{(t-T_Q-1)/2} \left(r(s^{\tau_Q}_t,a^{\tau_Q}_t)+\lambda H(\pi(\cdot ;s_t))\right)\Bigg)\Bigg\|_\infty^2\Bigg]\\
\leq & 4\EE\left[\norm{\left(T_Q K_2 +\frac{\gamma}{1-\gamma} K_1 \frac{\overline{R}+\lambda \log |\mathcal{A}|}{1-\gamma^{0.5}}\right)}_\infty^2\right], 
\end{aligned}
\label{eq:boundonadvantageestimate}
\end{equation}
where $T_Q$ is the random variable defined in \Cref{alg:sample_gradient_estimation}. The first inequality uses the definition of the advantage estimate (cf. \Cref{alg:sample_gradient_estimation}), the i.i.d. property of $T_Q$ and $T_V$, as well as of $T_Q'$ and $T_V'$, and \Cref{eq:triangleinequalitysquared}. The second inequality follows from \Cref{assum:theory}.

Plugging in the above, we thus get:
\begin{align*}
    \EE\left[\norm{ \widehat{\frac{dF(x_t)}{dx}}}^2_\infty\right]&=\EE\left[\norm{ L_f  + \frac{1}{\lambda m} L_f  \widehat{\partial_x A^{\pi^o_{x_t,\xi}}_{\lambda,x,\xi}}(s,a)}^2_\infty\right]\\
    &\leq \EE\left[2\norm{ L_f}_\infty^2  +2\left( \frac{L_f}{\lambda m}\right)^2 \norm{\widehat{\partial_x A^{\pi^o_{x_t,\xi}}_{\lambda,x,\xi}}(s,a)}^2_\infty\right]\\
    &\leq 2\EE\left[\norm{ L_f}_\infty^2\right]  +8\left( \frac{L_f}{\lambda m}\right)^2 \EE\left[\norm{\left(T_Q K_2 +\frac{\gamma}{1-\gamma} K_1 \frac{\overline{R}+\lambda \log |\mathcal{A}|}{1-\gamma^{0.5}}\right)}_\infty^2\right]\\
    &\leq2\EE\left[\norm{ L_f}_\infty^2\right]  +16\left( \frac{L_f}{\lambda m}\right)^2 \EE\left[\norm{T_Q K_2}_\infty^2\right]  +16\EE\left[\norm{\frac{\gamma}{1-\gamma} K_1 \frac{\overline{R}+\lambda \log |\mathcal{A}|}{1-\gamma^{0.5}}}_\infty^2\right]\\
    &=\cO(1),
\end{align*}
where we repeatedly apply \Cref{eq:triangleinequalitysquared} and the fact that the second moment of a geometric random variable is finite.

Now we can plug all our bounds back into \Cref{eq:graddecompositionbound} to get the result of \Cref{thm:convergence}.
\begin{align*}
        \EE\left[\norm{{\frac{dF(\hat{x}_T)}{dx}}}^2_\infty\right] 
    &\leq {\textcolor[RGB]{0,100,0}{\textbf{(1)}}}+{\textcolor{red}{\textbf{(2)}}} +{\textcolor{blue}{\textbf{(3)}}}\\
    &\leq \mathcal{O}(\frac{1}{\alpha T})+ \mathcal{O}(\delta) + 2S_f\alpha \left(\mathcal{O}(\delta^2 + 1)\right)\\
    &\leq \mathcal{O}(\frac{1}{\alpha T})+ \mathcal{O}(\delta) + \mathcal{O}(\alpha).
\end{align*}

\end{proof}

\subsection{Proof of \texorpdfstring{\Cref{thm:fasterconv}}{}}
\label{app:prooffaster}

\textbf{\textit{Vanilla soft Q-learning} }

We give a brief overview of how HPGD is combined with vanilla soft Q-learning (\Cref{alg:qlearning}) to get a bias of $\cO(2^{-K/2})$ in \Cref{alg:hpdgvanillaq}. In the algorithm we refer to $t_K=\cO(K2^{K})$ as the number of iterations soft Q-learning needs to achieve $\lambda \EE{\norm{\pi^{t_K}-\pi^*}^2_\infty}\leq \EE{\norm{Q^{t_K}-Q^*_{\lambda}}^2_\infty}\leq 2^{-K}$ (cf. \Cref{prop:qlearningconv}). Note we slightly abuse notation, when we pass a policy instead of an oracle to \Cref{alg:sample_gradient_estimation}. However, the policy can be equivalently used to sample trajectories. Note that the idea of \salgo\ with soft Q-learning is a double-loop methods that aims to find an $\delta$-oracle. Similar a $\delta$-oracle idea also appears in conditional stochastic optimization~\citep{hu2020biased,hu2020sample}, in bilevel optimization~\citep{ghadimi2018approximation}.

\begin{algorithm}
\caption{\salgo\ with vanilla soft Q-learning}
\label{alg:hpdgvanillaq}
\begin{algorithmic}[H]
\STATE \textbf{Input:} Iterations $T$, Precision param. $K$, Learning rate $\alpha$, Regularization $\lambda$, Initial point $x_0$, behavioural policy $\pi_B$, Q-learning rates $\{\alpha_{t}\}_{t\geq 0}$
\FOR{$t = 0$ to $T-1$}
    \STATE $\xi\sim \mathbb{P}_\xi$
    \STATE $\pi^{t_K}\gets \texttt{SoftQlearning}_{x_t,\xi} (t_K,\pi_B,\{\alpha_{t}\}_{t\geq 0})$ (\Cref{alg:qlearning})
    \STATE $ s \sim \nu$ and $a \sim \pi^{t_K}(\cdot;s)$
    \STATE $\widehat{\partial_x A^{\pi^{t_K}}}(s,a) \gets \texttt{GradientEstimator}(\xi,x_t,s,a,\pi^{t_K})$  (\Cref{alg:sample_gradient_estimation})
    \STATE $\widehat{\frac{dF}{dx}} \gets \frac{\partial_1 f(x_t, \pi^{t_K},\xi)}{\partial x}  + \frac{1}{\lambda \nu(s)} \frac{\partial_2 f(x_t, \pi^{t_K},\xi)}{\partial \pi(s,a)} \widehat{\partial_x A^{\pi^{t_K}}}(s,a)$
    \STATE $x_{t+1}\gets x_t-\alpha \widehat{\frac{dF}{dx}}$
\ENDFOR
\STATE \textbf{Output:} $\hat{x}_T \sim \textrm{Uniform}(\{x_0,\dots,x_{T-1}\})$
\end{algorithmic}
\end{algorithm}

\textbf{\textit{Randomly-Truncated soft Q-learning (RT-Q)}} 

Let us now turn to RT-Q, for which we provide the pseudocode in \Cref{alg:hpgdrtq}. As above, we denote by $t_k = \cO(k2^{k})$ the number of iterations soft Q-learning needs to achieve $\lambda \EE{\norm{\pi^{t_k}-\pi^*}^2_\infty}\leq \EE{\norm{Q^{t_k}-Q^*_{\lambda}}^2_\infty}\leq 2^{-k}$. We slightly abuse notation in the Pseudocode, such that we do not just return the last soft Q-learning iteration but also the second last and the first. Moreover we denote by $p_k=\frac{2^{-k}}{1-2^{-K}}$ and we generally use $\widehat{x}$ to denote estimates from a single sample and $\widetilde{x}$ to denote averaged estimates from multiple samples. %

\begin{algorithm}
\caption{\salgo\ with RT-Q}
\label{alg:hpgdrtq}
\begin{algorithmic}[H]
\STATE \textbf{Input:} Iterations $T$, Precision param. $K$, Learning rate $\alpha$, Regularization $\lambda$, Initial point $x_0$, behavioural policy $\pi_B$, Q-learning rates $\{\alpha_{t}\}_{t\geq 0}$
\FOR{$t = 0$ to $T-1$}
    \STATE $\xi\sim \mathbb{P}_\xi$
    \STATE $k \sim p_k $
    \STATE $\pi^{t_{k+1}},\pi^{t_{k}},\pi^{t_{1}}\gets \texttt{SoftQlearning}_{x_t,\xi} (t_{k+1},\pi_B,\{\alpha_{t}\}_{t\geq 0})$ (\Cref{alg:qlearning})
    \STATE $ s \sim \nu$, $a \sim \pi^{t_{k+1}}(\cdot;s)$ and $a'\sim \pi^{t_1}(\cdot;s)$
\STATE $\widetilde{ \partial_x A^{\pi^{t_k}}}(s,a)\gets\frac{1}{2^{k}}\sum_{l=1}^{2^{k}}  \texttt{GradientEstimator} (\xi,x_t,s,a,\pi^{t_k})$(\Cref{alg:sample_gradient_estimation})
    \STATE $\widetilde{ \partial_x A^{\pi^{t_{k+1}}}}(s,a)\gets\frac{1}{2^{k}}\sum_{l=1}^{2^{k}}  \texttt{GradientEstimator}(\xi,x_t,s,a,\pi^{t_{k+1}})$
        \STATE $\widehat{ \partial_x A^{\pi^{t_{1}}}}(s,a')\gets \texttt{GradientEstimator}(\xi,x_t,s,a',\pi^{t_{1}})$
    \STATE  $\widetilde{\frac{dF_{t_{k+1}}}{dx}} \gets \frac{\partial_1 f(x, \pi^{t_{k+1}},\xi)}{\partial x}  + \frac{1}{\lambda \nu(s)} \frac{\partial_2 f(x, \pi^{t_{k+1}},\xi))}{\partial \pi(s,a)}\widetilde{ \partial_x A^{\pi^{t_{k+1}}}}(s,a)$
    \STATE $ \widetilde{\frac{dF^{t_{k+1}}_{t_{k}}}{dx}} = \frac{\partial_1 f(x, \pi^{t_{k}},\xi)}{\partial x}  + \frac{\pi^{t_k}(a;s)}{\pi^{t_{k+1}}(a;s)}\frac{1}{\lambda \nu(s)} \frac{\partial_2 f(x, \pi^{t_{k}},\xi))}{\partial \pi(s,a)} 
   \widetilde{ \partial_x A^{\pi^{t_{k}}}}(s,a)$
    \STATE 
    $ \widehat{\frac{dF_{t_{1}}}{dx}} = \frac{\partial_1 f(x, \pi^{t_{1}},\xi)}{\partial x}  + \frac{1}{\lambda \nu(s)} \frac{\partial_2 f(x, \pi^{t_{1}},\xi))}{\partial \pi(s,a')} \widehat{ \partial_x A^{\pi^{t_{1}}}}(s,a')$
    \STATE ${\frac{dF^{RT}_{t_K}}{dx}}=\widehat{\frac{dF_{t_{1}}}{dx}}+\frac{ \widetilde{\frac{dF_{t_{k+1}}}{dx}}-\widetilde{\frac{dF^{t_{k+1}}_{t_{k}}}{dx}}}{p_k}$
    \STATE $x_{t+1}\gets x_t-\alpha \frac{dF^{RT}_{t_K}}{dx}$
\ENDFOR
\STATE \textbf{Output:} $\hat{x}_T \sim \textrm{Uniform}(\{x_0,\dots,x_{T-1}\})$
\end{algorithmic}
\end{algorithm}

\begin{proof}

First let us specify that “bias” refers to the bias of the hypergradient estimator, i.e. \begin{equation*}
    \norm{\EE\left[\frac{dF(x_t)}{dx}  - \widehat{\frac{dF(x_t)}{dx}}\right]}_\infty.
\end{equation*}
Equivalently for “variance” we mean the variance of the estimator i.e.
\begin{equation*}
     \EE\left[\norm{\widehat{\frac{dF(x_t)}{dx}}-\EE\left[\widehat{\frac{dF(x_t)}{dx}}\right]}_\infty^2\right],
\end{equation*}

Moreover, the “iteration complexity” is the number of soft Q-learning iterations needed to solve the lower level.

For Vanilla soft Q-learning, we can just combine previous results to compute the iteration complexity and variance to achieve a bias of $2^{-K/2}$. For RT-Q, we formalize the intuition of \Cref{sec:full_access} to show we can achieve the same bias with only $\cO(K^2)$ iterations because we rarely perform many soft Q-learning iterations but assign a relatively higher magnitude to these few accurate hypergradient estimates. This necessarily increases variance and most of the proof will be spent on how to bound it. Here we “divide and conquer” the variance until we have easy terms that depend linearly on 
\begin{equation*}
     \EE\left[\norm{\pi^{t_k}_{x,\xi}-\pi^*_{x,\xi}}^2_\infty\right] = \cO(2^{-k})
\end{equation*}
or related terms we can easily bound. To decompose the variance, our main tool will be the following identity
\begin{equation}
\label{eq:triangleinequalitysquared}
    \norm{a+b}^2= \norm{a}^2+2\norm{a}\norm{b}+\norm{b}^2\leq2\norm{a}+2\norm{b}^2,
\end{equation}
which we have already derived and used in the proof of \Cref{thm:convergence} and also follows from the Parallelogram Law.

\textbf{\textit{Vanilla soft Q-learning} }

We start with the analysis of using HPGD with vanilla soft Q-learning to estimate $\frac{dF(x)}{dx}$. 
In \Cref{thm:convergence}, we showed that 
\begin{equation*}
    \norm{\EE\left[\frac{dF(x_t)}{dx}  - \widehat{\frac{dF(x_t)}{dx}}\right]}_\infty=\cO(\delta),
\end{equation*}
where
\begin{equation*}
    \EE_{o}\left[\norm{\pi^*_{x,\xi}-\pi^o_{x,\xi}}^2_{\infty}\right]\leq \delta^2.
\end{equation*}
In \Cref{prop:qlearningconv}, we show that after $t_K=\cO(K2^K )$ soft Q-learning iterations, it holds that
\begin{equation*}
    \EE\left[\norm{\pi^{t_K}_{x,\xi}-\pi^*_{x,\xi}}^2_\infty\right] = \cO(2^{-K}), 
\end{equation*}
where $\pi^{t_K}_{x,\xi}$ denotes the $t_K$-th iterate of the soft Q-learning algorithm. The results for complexity and bias follow directly.
It remains to bound the variance. However, we have already shown in \Cref{thm:convergence} that
\begin{align*}
    \EE\left[\norm{ \widehat{\frac{dF_{t_K}}{dx}}-\EE\left[{\widehat{\frac{dF_{t_K}}{dx}}}\right]}^2_\infty\right]
    \leq \EE\left[\norm{ \widehat{\frac{dF(x_t)}{dx}}}^2_\infty\right]=\cO(1).
\end{align*}

\textbf{\textit{Randomly-Truncated soft Q-learning (RT-Q)}} 

We first show that ${\frac{dF^{RT}_{t_K}}{dx}}$ has the same bias as its vanilla counterpart $\widehat{\frac{dF_{t_K}}{dx}}$. For this, observe that the following estimators have the same mean:
\begin{align*}
   \forall k:  \EE\left[\widetilde{\frac{dF^{t_{k+1}}_{t_{k}}}{dx}}\right]=\EE\left[\widetilde{\frac{dF_{t_{k}}}{dx}}\right]=\EE\left[\widehat{\frac{dF_{t_{k}}}{dx}}\right].
\end{align*}
The first equality holds because of importance sampling and the second inequality holds by the linearity of expectation. Plugging in these identities, we get

\begin{align*}
  &  \EE\left[ \widetilde{\frac{dF^{RT}_{t_K}}{dx}}\right]\\
=&\EE\left[\widehat{ \frac{dF_{t_1}}{dx}}\right]+\sum_{k=1}^{K} p_k \frac{1}{p_k}\EE\left[\widetilde{\frac{dF_{t_{k+1}}}{dx}}-\widetilde{\frac{dF^{t_{k+1}}_{t_{k}}}{dx}}\right]\\
=&\EE\left[ \widehat{\frac{dF_{t_1}}{dx}}\right]+\sum_{k=1}^{K} p_k \frac{1}{p_k}\left(\EE\left[\widehat{\frac{dF_{t_{k+1}}}{dx}}\right]-\EE\left[\widehat{\frac{dF_{t_{k}}}{dx}}\right]\right)\\
= &\EE\left[\widehat{\frac{dF_{t_{K}}}{dx}}\right].
\end{align*}
It follows directly that the hypergradient estimators obtained by RT-Q and vanilla soft Q-learning must have the same bias.

For a sampled $k\in \{1,\dots,K\}$, soft Q-learning has an iteration complexity $c_k$ of $\cO(k2^k )$ to build the following hypergradient estimator:
\begin{equation*}
   {\frac{dF^{RT}_{t_K}}{dx}}=\widetilde{\frac{dF_{t_1}}{dx}}+\frac{ \widetilde{\frac{dF_{t_{k+1}}}{dx}}-\widetilde{\frac{dF^{t_{k+1}}_{t_{k}}}{dx}}}{p_k}. 
\end{equation*}
As we sample any $k$ with probability $p_k=\frac{2^{-k}}{1-2^{-K}}$, the expected iteration complexity is then given by 

\begin{align*}
&\sum_{k=1}^K c_k p_k\\
    =&\sum_{k=1}^K c_k \frac{2^{-k}}{1-2^{-K}}\\
    =& \sum_{k=1}^K \cO\left(\frac{k}{2^{-k}}\right)\frac{2^{-k}}{1-2^{-K}}\\
    \leq& \cO\left(K\right) \sum_{k=1}^K \cO(\frac{1}{1-2^{-K}})\\
    =& \cO(K^2),
\end{align*}
which proves our claim. The attentive reader will note that RT-Q runs \texttt{GradientEstimator} $2^k$ times instead of once and thus samples $2^k$ trajectories to estimate the advantage derivative instead of one like vanilla soft Q-learning. Nonetheless, RT-Q has the better sample complexity as both methods need to sample $\cO(k 2^k)$ state action pairs for a given $k$ to run soft Q-learning and then estimate the advantage derivative. The same analysis as for the iteration complexity thus shows that the sample complexity of vanilla soft Q-learning is $ \cO(K2^K)$ and $\cO(K^2)$ for RT-Q.

It remains to show that the variance is of order $\cO(K)$. This is the most challenging part of the proof. As outlined previously, we will iteratively decompose the variance until we can bound all terms by $\cO(K)$.

For better readability we introduce the following notation for a given pair $s,a$:
\begin{align*}
    H_k(1)&=\frac{\partial_1 f(x, \pi^{t_k},\xi)}{\partial x}  \\
H_k(2) &=
\begin{cases}
\frac{1}{\lambda \nu(s)} \frac{\partial_2 f(x, \pi^{t_{k}}, \xi)}{\partial \pi(s, a)} \widehat{ \partial_x A^{\pi^{t_{k}}}}(s, a) & \text{if } k = 1 \\
\frac{1}{\lambda \nu(s)} \frac{\partial_2 f(x, \pi^{t_{k}}, \xi)}{\partial \pi(s, a)} \widetilde{ \partial_x A^{\pi^{t_{k}}}}(s, a) & \text{if } k > 1
\end{cases}\\
          H_k^{k+1}(2)&=\frac{1}{\lambda \nu(s)} \frac{\pi^{t_k}(a;s)}{\pi^{t_{k+1}}(a;s)}\frac{\partial_2 f(x, \pi^{t_{k}},\xi))}{\partial \pi(s,a)} \widetilde{ \partial_x A^{\pi^{t_{k}}}}(s,a)\\
        H^*(1)&=\frac{\partial_1 f(x, \pi^*_{x,\xi},\xi)}{\partial x}  \\  
    H^*(2)&=\frac{1}{\lambda \nu(s)} \frac{\partial_2 f(x, \pi^*_{x,\xi},\xi))}{\partial \pi(s,a)} {\partial_x A^{\pi^{*}_{x,\xi}}_{\lambda,x,\xi}}(s,a).
\end{align*}
Then for a given $\xi,s,a$ (which are sampled at the beginning of RT-Q) we have that 
\begin{equation*}
    {\frac{d}{dx} F^{RT}_{t_K}}=H_1(1)+H_1(2) + \frac{H_{k+1}(1)+H_{k+1}(2)-H_{k}(1)-H^{k+1}_{k}(2)}{p_{k}}.
\end{equation*}

Now let us decompose the variance of the RT-Q hypergradient estimator using the newly introduced notation:
\begin{align*}
   & \EE\left[\norm{\frac{d}{dx} F^{RT}_{t_K}-\EE\left[\frac{d}{dx} F^{RT}_{t_K}\right]}_\infty^2\right]\\
  \leq &\EE\left[\norm{\frac{d}{dx} F^{RT}_{t_K}-\frac{d}{dx} F(x)}_\infty^2\right]\\
   \leq& 2 \EE\left[\norm{\frac{d}{dx} F^{RT}_{t_K}-H_1(1)-H_1(2)}_\infty^2\right]+2\EE\left[\norm{H_1(1)+H_1(2)-\frac{d}{dx} F(x)}_\infty^2\right]\\
   \leq & 4 \underbrace{\EE\left[\norm{\frac{H_{k+1}(1)-H_{k}(1)}{p_{k}}}_\infty^2\right]}_{\textcolor{red}{\textbf{(1)}}}+4\underbrace{ \EE\left[\norm{\frac{H_{k+1}(2)-H^{k+1}_{k}(2)}{p_{k}}}_\infty^2\right]}_{\textcolor{green}{\textbf{(2)}}}+2\underbrace{\EE\left[\norm{H_1(1)+H_1(2)-\frac{d}{dx} F(x)}_\infty^2\right]}_{\textcolor{blue}{\textbf{(3)}}}.
\end{align*}
We proceed to bound the individual terms.

Note that \textcolor{blue}{\textbf{(3)}} is indepedent of $k$. Using \Cref{eq:boundonadvantageestimate} in the analysis of \Cref{thm:convergence}, we can bound it as
\begin{align*}
    \textcolor{blue}{\textbf{(3)}}&\leq 4\left(\left(\frac{K_2}{1-\gamma}+\frac{K_1(\overline{R}+\lambda \log |\mathcal{A}|)}{(1-\gamma)^2}\right)+\frac{1}{1-\gamma} K_2 +\frac{\gamma}{1-\gamma} K_1 \frac{\overline{R}+\lambda \log |\mathcal{A}|}{1-\gamma^{0.5}}\right)^2\\
    &=\cO(1).
\end{align*}

{\textcolor{red}{\textbf{(1)}}} is also relatively easy to bound, as shown below:
\begin{align*}
    {\textcolor{red}{\textbf{(1)}}}&= \sum_{k=1}^K \frac{1}{p_k}\EE\left[\norm{H_{k+1}(1)-H_{k}(1)}_\infty^2\right]\\
    &=\sum_{k=1}^K \frac{1}{p_k}\EE\left[\norm{\frac{\partial_1 f(x, \pi^{t_{k+1}},\xi)}{\partial x}-\frac{\partial_1 f(x, \pi^{t_{k}},\xi)}{\partial x}}_\infty^2\right] \\
   &\leq \sum_{k=1}^K \frac{1}{p_k} S_f^2 \EE\left[ \norm{\pi^{t_{k+1}}-\pi^{t_{k}}}_\infty^2\right] \\
   &\leq \sum_{k=1}^K \frac{1}{p_k} S_f^2 2\left(\EE\left[ \norm{\pi^{t_{k+1}}-\pi^{*}_{x,\xi}}_\infty^2\right]+\EE\left[ \norm{\pi^{t_{k}}-\pi^{*}_{x,\xi}}_\infty^2\right] \right)\\
   &\leq \sum_{k=1}^K \frac{1}{p_k} S_f^2 2\left(\frac{1}{2^{k+1}}+\frac{1}{2^{k}} \right)\\
      &= \sum_{k=1}^K 2^k S_f^2 \frac{3}{2^{k}} \\
      &= 3 S_f^2 K\\
      &=\cO(K).
\end{align*}
In the first equality we simply use that $k$ is sampled with probability $p_k=\frac{2^{-k}}{1-2^{-K}}$. In the second equality we plug in the definitions of $H_{k+1}(1)$ and $H_{k}(1)$. In the first inequality we use the smoothness of $f$. The second inequality uses \Cref{eq:triangleinequalitysquared}. The third inequality follows from the fact that $t_{k+1}$ is chosen to guarantee an expected distance of at most $\frac{1}{2^{k+1}}$ to $\pi^*_{x,\xi}$. The remaining equalities follow from plugging in $p_k$ and rearranging terms.

Now we want to repeat the same analysis again for {\textcolor{green}{\textbf{(2)}}}. As for {\textcolor{red}{\textbf{(1)}}} we get a sum over $k$ with the factor $\frac{1}{p_k}$, which we need to compensate by bounding $\EE\left[\norm{{H_{k+1}(2)-H^{k+1}_{k}(2)}}_\infty^2\right]$ by $\cO(2^{-k})$. However, bounding the latter is more invovled than our analysis for {\textcolor{red}{\textbf{(1)}}}. In the following pages, we will iteratively apply \Cref{eq:triangleinequalitysquared} until the terms get easy enough, such that we can use one of the following two facts about RT-Q:
\begin{enumerate}
    \item  $t_k = \cO(k2^{k})$ is chosen sucht that $\lambda \EE{\norm{\pi^{t_k}-\pi^*}^2_\infty}\leq \EE{\norm{Q^{t_k}-Q^*_{\lambda}}^2_\infty}\leq 2^{-k}$
    \item $\widetilde{ \partial_x A^{\pi^{t_{k+1}}}}(s,a)$ and related terms are an average estimate over $2^k$ trajectory samples, such that their variance (with respect to these random rollouts) is $\cO(\frac{1}{2^k})$.
\end{enumerate}

We briefly note that in the analysis below we will sometimes write $Q(s,a)$ or $A(s,a)$ inside the infinity norm for ease of exposition. When we do so, the infinity norm is still interpreted as the maximum over all possible $s,a$ and not as the absolute value of the Q-function or advantage for a specific $s,a$. 

Let us start decomposing the numerator of {\textcolor{green}{\textbf{(2)}}}:

\begin{align*}
&\EE\left[\norm{{H_{k+1}(2)-H^{k+1}_{k}(2)}}_\infty^2\right]\\
    =&\EE\Biggl[\norm{\frac{1}{\lambda \nu(s)} \frac{\partial_2 f(x, \pi^{t_k},\xi))}{\partial \pi(s,a)} \widetilde{\partial_x A^{\pi^{t_k}}}(s,a)-\frac{1}{\lambda \nu(s)} \frac{\pi^{t_k}(a;s)}{\pi^{t_{k+1}}(a;s)}\frac{\partial_2 f(x, \pi^{t_{k}},\xi))}{\partial \pi(s,a)} \widetilde{ \partial_x A^{\pi^{t_{k}}}}(s,a)}_{\infty}^2\Biggr]\\
    \leq & 2\EE\Biggl[\norm{\frac{1}{\lambda \nu(s)} \frac{\partial_2 f(x, \pi^{t_k},\xi))}{\partial \pi(s,a)} \widetilde{\partial_x A^{\pi^{t_k}}}(s,a)}_{\infty}^2\norm{\frac{\pi^{t_{k+1}}(a;s)}{\pi^{t_{k+1}}(a;s)}-\frac{\pi^{t_k}(a;s)}{\pi^{t_{k+1}}(a;s)}}_\infty^2\Biggr]\\
    &+2 \EE\Biggl[\norm{\frac{\pi^{t_k}(a;s)}{\pi^{t_{k+1}}(a;s)}}_\infty^2\norm{\frac{1}{\lambda \nu(s)} \frac{\partial_2 f(x, \pi^{t_{k+1}},\xi))}{\partial \pi(s,a)} \widetilde{\partial_x A^{\pi^{t_{k+1}}}}(s,a)-\frac{1}{\lambda \nu(s)} \frac{\partial_2 f(x, \pi^{t_{k}},\xi))}{\partial \pi(s,a)} \widetilde{ \partial_x A^{\pi^{t_{k}}}}(s,a)}_{\infty}^2\Biggr]\\
 \leq   & 2\underbrace{\EE\Biggl[\norm{\frac{1}{\lambda \nu(s)} \frac{\partial_2 f(x, \pi^{t_k},\xi))}{\partial \pi(s,a)} \widetilde{\partial_x A^{\pi^{t_k}}}(s,a)}_{\infty}^2\norm{{\pi^{t_{k+1}}(a;s)}-\pi^{t_k}(a;s)}_\infty^2\norm{\frac{1}{\pi^{t_{k+1}}(a;s)}}_\infty^2\Biggr]}_{\textcolor{yellow}{\textbf{(i)}}}\\
    &+2 \underbrace{\EE\Biggl[\norm{\frac{\pi^{t_k}(a;s)}{\pi^{t_{k+1}}(a;s)}}_\infty^2\norm{\frac{1}{\lambda \nu(s)} \frac{\partial_2 f(x, \pi^{t_{k+1}},\xi))}{\partial \pi(s,a)} \widetilde{\partial_x A^{\pi^{t_{k+1}}}}(s,a)-\frac{1}{\lambda \nu(s)} \frac{\partial_2 f(x, \pi^{t_{k}},\xi))}{\partial \pi(s,a)} \widetilde{ \partial_x A^{\pi^{t_{k}}}}(s,a)}_{\infty}^2\Biggr]}_{\textcolor{violet}{\textbf{(ii)}}}
\end{align*}
The first equality is just plugging in definitions. The first inequality follows from \Cref{eq:triangleinequalitysquared} and the second inequality comes from the Cauchy-Schwarz inequality.

Let us begin with bounding {\textcolor{yellow}{\textbf{(i)}}}. First we want to tackle the fractions of the policies. For this, recall that 
$\forall x,\xi,s,a : |r_{x,\xi}(s,a)|<\overline{R}$ and that $0\leq H(\pi;s)\leq \log|\cA|$ (by Jensen's inequality).
Thus we have that:
\[
\frac{-\overline{R}}{1-\gamma}\leq Q_{\lambda}(s,a)\leq \frac{\overline{R}+\lambda \log |\cA|}{1-\gamma}.
\]
The above bounds extend to any soft Q-learning estimate, which can be obtained by \Cref{alg:qlearning}, since by \Cref{assum:theory} the algorithm cannot observe rewards with greater magnitude than $\overline{R}$. As \Cref{alg:qlearning} returns a softmax policy, i.e. \[
\pi^{t_k}(a;s) = \frac{\exp(Q^{t_k}(a;s)/\lambda)}{\sum_{a'} \exp(Q^{t_k}(a';s)/\lambda)},
\]it holds that:
\begin{equation}
\label{eq:minmax_softmaxpolicy}
\forall t_k, \forall s, \forall a: \frac{\exp(\frac{\overline{R}+\lambda \log |\cA|}{\lambda(1-\gamma)})}{|\mathcal{A}|\exp(\frac{-\overline{R}}{\lambda(1-\gamma)})} \geq \pi^{t_k}(a;s)\geq \frac{\exp(\frac{-\overline{R}}{\lambda(1-\gamma)})}{|\mathcal{A}|\exp(\frac{\overline{R}+\lambda \log |\cA|}{\lambda(1-\gamma)})}.
\end{equation}
Therefore we have:
\begin{align*}
\exists M_1<\infty,\forall t_k:& \norm{\frac{1}{\pi^{t_{k}}(a;s)}}_\infty^2 \leq M_1 -
\\
\exists M_2<\infty,\forall t_k:& \norm{\frac{\pi^{t_k}(a;s)}{\pi^{t_{k+1}}(a;s)}}_\infty^2 \leq M_2.
\end{align*}
Using this we can bound {\textcolor{yellow}{\textbf{(i)}}}. Let $m=\min_s \nu(s)$, then: 
\begin{align*}
{\textcolor{yellow}{\textbf{(i)}}}&\leq \frac{M_1}{\lambda m} 2\left(\frac{ K_2}{1-\gamma} +\frac{\gamma}{1-\gamma} K_1 \frac{\overline{R}+\lambda \log |\mathcal{A}|}{1-\gamma^{0.5}}\right) \EE\left[\norm{{\pi^{t_{k+1}}(a;s)}-\pi^{t_k}(a;s)}_\infty^2 \right]\\
&\leq \cO\left(2\EE\left[\norm{{\pi^{t_{k+1}}(a;s)}-\pi^*_{x,\xi}(a;s)}_\infty^2 \right]+2\EE\left[\norm{{\pi^{t_{k}}(a;s)}-\pi^*_{x,\xi}(a;s)}_\infty^2 \right]\right)\\
&\leq \cO\left(\frac{1}{2^k}\right).
\end{align*}
The first inequality uses \Cref{eq:boundonadvantageestimate}. The second inequality uses \Cref{eq:triangleinequalitysquared} and the last equality uses the convergence of soft Q-learning. 

Now we turn to bound {\textcolor{violet}{\textbf{(ii)}}}:
\begin{align*}
   {\textcolor{violet}{\textbf{(ii)}}} \leq& 2\frac{M_2}{\lambda m} \EE\left[\norm{ \frac{\partial_2 f(x, \pi^{t_{k+1}},\xi))}{\partial \pi(s,a)} \widetilde{\partial_x A^{\pi^{t_{k+1}}}}(s,a)- \frac{\partial_2 f(x, \pi^{t_{k}},\xi))}{\partial \pi(s,a)} \widetilde{ \partial_x A^{\pi^{t_{k}}}}(s,a)}_{\infty}^2\right]\\
   \leq &2\frac{M_2}{\lambda m} \EE\Bigg[2\norm{ \frac{\partial_2 f(x, \pi^{t_{k+1}},\xi))}{\partial \pi(s,a)} \widetilde{\partial_x A^{\pi^{t_{k+1}}}}(s,a)- \frac{\partial_2 f(x, \pi^{t_{k+1}},\xi))}{\partial \pi(s,a)} \widetilde{ \partial_x A^{\pi^{t_{k}}}}(s,a)}_{\infty}^2\\
   &+ 2\norm{ \frac{\partial_2 f(x, \pi^{t_{k+1}},\xi))}{\partial \pi(s,a)} \widetilde{\partial_x A^{\pi^t_{k}}}(s,a)- \frac{\partial_2 f(x, \pi^{t_{k}},\xi))}{\partial \pi(s,a)} \widetilde{ \partial_x A^{\pi^{t_{k}}}}(s,a)}_{\infty}^2\Bigg]\\
   \leq &4\frac{M_2}{\lambda m} \Bigg( \EE\Bigg[\norm{ \frac{\partial_2 f(x, \pi^{t_{k+1}},\xi))}{\partial \pi(s,a)}}_\infty^2 \norm{\widetilde{\partial_x A^{\pi^{t_{k+1}}}}(s,a)-  \widetilde{ \partial_x A^{\pi^{t_{k}}}}(s,a)}_{\infty}^2\\
   &+ \norm{ \widetilde{ \partial_x A^{\pi^{t_{k}}}}(s,a)}_\infty^2 \norm{\frac{\partial_2 f(x, \pi^{t_{k+1}},\xi))}{\partial \pi(s,a)} - \frac{\partial_2 f(x, \pi^{t_{k}},\xi))}{\partial \pi(s,a)} }_{\infty}^2\Bigg]\Bigg)\\
   \leq &4\frac{M_2}{\lambda m}\Bigg( \underbrace{\EE\Bigg[ \norm{ \widetilde{ \partial_x A^{\pi^{t_{k}}}}(s,a)}_\infty^2 \norm{\frac{\partial_2 f(x, \pi^{t_{k+1}},\xi))}{\partial \pi(s,a)} - \frac{\partial_2 f(x, \pi^{t_{k}},\xi))}{\partial \pi(s,a)} }_{\infty}^2\Bigg]}_{\textcolor{cyan}{\textbf{(A)}}}\\
   &+L_f^2 \underbrace{\EE\Bigg[\norm{\widetilde{\partial_x A^{\pi^{t_{k+1}}}}(s,a)-  \widetilde{ \partial_x A^{\pi^{t_{k}}}}(s,a)}_{\infty}^2\Bigg] }_{\textcolor{magenta}{\textbf{(B)}}}\Bigg). \\
   \end{align*}
The first inequality uses $M_2$ to bound the policy fraction the definition of $m=\min_s \nu(s)$ and Cauchy-Schwarz. The second inequality uses \Cref{eq:triangleinequalitysquared} and Cauchy-Schwarz. For the final inequality, we rearrange and use the Lipschitz continuity of $f$ (cf. \Cref{assum:theory}).

{\textcolor{cyan}{\textbf{(A)}}} is relatively easy to bound as follows:
\begin{align*}
    {\textcolor{cyan}{\textbf{(A)}}}&\leq \left(2\left(\frac{K_2}{1-\gamma} +\frac{\gamma}{1-\gamma} K_1 \frac{\overline{R}+\lambda \log |\mathcal{A}|}{1-\gamma^{0.5}}\right) \right)^2 S_f  \EE\left[\norm{{\pi^{t_{k+1}}(a;s)}-\pi^{t_k}(a;s)}_\infty^2 \right]\\
    &\leq \cO\left(2\EE\left[\norm{{\pi^{t_{k+1}}(a;s)}-\pi^*_{x,\xi}(a;s)}_\infty^2 \right]+2\EE\left[\norm{{\pi^{t_{k}}(a;s)}-\pi^*_{x,\xi}(a;s)}_\infty^2 \right]\right)\\
&\leq \cO\left(\frac{1}{2^k}\right).
\end{align*}
In the first inequality we use the smoothness of $f$ and \Cref{eq:boundonadvantageestimate}. In the second inequality we use \Cref{eq:triangleinequalitysquared} and in the final inequality the convergence of soft Q-learning.

Now, we turn our attention towards bounding {\textcolor{magenta}{\textbf{(B)}}}.
First, notice that both advantage derivatives are evaluated for the same state-action pair. This is because in \Cref{alg:hpgdrtq} we sample $a$ once according to $\pi^{t_{k+1}}$ and reuse the same $a$ for $A^{\pi^{t_{k}}}(s,a)$ by doing importance sampling. Without this trick, it would be hopeless to bound {\textcolor{magenta}{\textbf{(B)}}}.

When bounding {\textcolor{magenta}{\textbf{(B)}}}, we need to consider two sources of randomness. First, there is the randomness in the soft Q-learning iterations. Second, we have to account for the randomness over the trajectory rollouts of \texttt{GradientEstimator} to estimate $\widetilde{\partial_x A^{\pi}}$. 
\texttt{GradienEstimator} first separately estimates $\widehat{\partial_x V^{\pi^{t_k}}}$ and $\widehat{\partial_x V^{\pi^{t_k}}}$ from a trajectory and then returns $\widehat{\partial_x A^{\pi^{t_k}}}=\widehat{\partial_x Q^{\pi^{t_k}}}-\widehat{\partial_x V^{\pi^{t_k}}}$. We therefore denote by $${\widetilde{\partial_x V^{\pi^{t_k}}}}:=\frac{1}{2^{k}}\sum_{l=1}^{2^{k}}  \widehat{ \partial_x V^{\pi^{t_k}}}(\tau_l),$$ the average over the $2^k$ value function derivatives estimated as part of the $2^k$ \texttt{GradientEstimator} procedures performed in RT-Q. Note because of the unbiasedness of \texttt{GradienEstimator} (cf. \Cref{prop:unbiasedgradient}), it holds that ${\partial_x V^{\pi^{t_k}}}=\EE_\tau\left[{\widetilde{\partial_x V^{\pi^{t_k}}}}\right]$

Using the above notation, we get the following bound by repeatedly applying \Cref{eq:triangleinequalitysquared}:
    \begin{align*}
  {\textcolor{magenta}{\textbf{(B)}}} \leq& 2 \EE\norm{\widetilde{\partial_x Q^{\pi^{t_k}}}(s,a)-\widetilde{\partial_x Q^{\pi^{t_{k+1}}}}(s,a)}_\infty^2 +2\EE \norm{\widetilde{\partial_x V^{\pi^{t_k}}}(s)-\widetilde{\partial_x V^{\pi^{t_{k+1}}}}(s)}_\infty^2 \\
  \leq &4 \EE\norm{\widetilde{\partial_x Q^{\pi^{t_k}}}(s,a)-\partial_x Q^{\pi^*}(s,a)}_\infty^2 +4 \EE\norm{\widetilde{\partial_x Q^{\pi^{t_{k+1}}}}(s,a)-\partial_x Q^{\pi^{*}}(s,a)}_\infty^2\\
  &+4\EE \norm{\widetilde{\partial_x V^{\pi^{t_k}}}(s)-\partial_x V^{\pi^{*}}(s)}_\infty^2+4\EE \norm{\widetilde{\partial_x V^{\pi^{t_{k+1}}}}(s)-\partial_x V^{\pi^{*}}(s)}_\infty^2\\
    \leq& 8 \EE\norm{\widetilde{\partial_x Q^{\pi^{t_k}}}(s,a)-\partial_x Q^{\pi^{t_{k}}}(s,a)}_\infty^2 +8 \EE\norm{{\partial_x Q^{\pi^{t_{k}}}}(s,a)-\partial_x Q^{\pi^*}(s,a)}_\infty^2\\
    &+8 \EE\norm{\widetilde{\partial_x Q^{\pi^{t_{k+1}}}}(s,a)-\partial_x Q^{\pi^{t_{k+1}}}(s,a)}_\infty^2 +8 \EE\norm{{\partial_x Q^{\pi^{t_{k+1}}}}(s,a)-\partial_x Q^{\pi^{*}}(s,a)}_\infty^2\\
  &+8\EE \norm{\widetilde{\partial_x V^{\pi^{t_k}}}(s)-{\partial_x V^{\pi^{t_{k}}}}(s)}_\infty^2+8\EE \norm{{\partial_x V^{\pi^{t_{k}}}}(s)-\partial_x V^{\pi^{*}}(s)}_\infty^2\\
 &+8\EE \norm{\widetilde{\partial_x V^{\pi^{t_{k+1}}}}(s)-{\partial_x V^{\pi^{t_{k+1}}}}(s)}_\infty^2+8\EE \norm{{\partial_x V^{\pi^{t_{k+1}}}}(s)-\partial_x V^{\pi^{*}}(s)}_\infty^2.\\
    \end{align*}
Here $Q^*_{\lambda}$ denotes the optimal Q-function, i.e. for the policy $\pi^*_{x,\xi}$. In the equations above we have two flavours of terms. The first are the differences between the derivative estimator and its expectation for a given policy and the second are the differences between the expected derivative under the learned and optimal policy. We start by bounding the first kind of terms.
Recall that \begin{equation*}
    \widetilde{ \partial_x Q^{\pi^{t_k}}}=\frac{1}{2^{k}}\sum_{l=1}^{2^{k}}  \widehat{ \partial_x Q^{\pi^{t_k}}}(\tau_l), \quad \widetilde{ \partial_x V^{\pi^{t_k}}}=\frac{1}{2^{k}}\sum_{l=1}^{2^{k}}  \widehat{ \partial_x V^{\pi^{t_k}}}(\tau_l)
\end{equation*}
As the second moment of $ \widehat{ \partial_x Q^{\pi^{t_k}}}(\tau_l)$ and $\widehat{ \partial_x V^{\pi^{t_k}}}(\tau_l)$  is bounded, we have that:
\begin{equation}
\label{eq:boundedvarianceestimator}
    \EE_\tau\left[\norm{ \widetilde{ \partial_x Q^{\pi^{t_k}}} -{\partial_x Q^{\pi^{t_k}}} }_\infty^2\right]=\cO(2^{-k}) ,\quad \EE_\tau\left[\norm{ \widetilde{ \partial_x V^{\pi^{t_k}}} -{\partial_x V^{\pi^{t_k}}} }_\infty^2\right]=\cO(2^{-k})
\end{equation}
since we use $2^k$ sampled trajectories. The same analysis of course also holds for $t_{k+1}$ and thus plugging this in, we get
     \begin{align*}
  {\textcolor{magenta}{\textbf{(B)}}} \leq &\cO(\frac{1}{2^k})+
    8 \EE\norm{{\partial_x Q^{\pi^{t_{k}}}}(s,a)-\partial_x Q^{\pi^*}(s,a)}_\infty^2+8 \EE\norm{{\partial_x Q^{\pi^{t_{k+1}}}}(s,a)-\partial_x Q^{\pi^{*}}(s,a)}_\infty^2\\
  &+8\EE \norm{{\partial_x V^{\pi^{t_{k}}}}(s)-\partial_x V^{\pi^{*}}(s)}_\infty^2+8\EE \norm{{\partial_x V^{\pi^{t_{k+1}}}}(s)-\partial_x V^{\pi^{*}}(s)}_\infty^2\\
  \leq &\cO(\frac{1}{2^k})+
    24 \EE\norm{{\partial_x Q^{\pi^{t_{k}}}}(s,a)-\partial_x Q^{\pi^*}(s,a)}_\infty^2+24 \EE\norm{{\partial_x Q^{\pi^{t_{k+1}}}}(s,a)-\partial_x Q^{\pi^{*}}(s,a)}_\infty^2\\
  &+16|\cA|^2 \left(\frac{K_1(\overline{R}+\lambda \log |\mathcal{A}|)}{(1-\gamma)^2}\right)^2 \left(\EE \norm{\pi^{t_{k}}-\pi^{*}}_\infty^2+\EE \norm{\pi^{t_{k+1}}-\pi^{*}}_\infty^2\right)\\
  \leq &\cO(\frac{1}{2^k})+ 24 \EE\norm{{\partial_x Q^{\pi^{t_{k}}}}(s,a)-\partial_x Q^{\pi^*}(s,a)}_\infty^2+24 \EE\norm{{\partial_x Q^{\pi^{t_{k+1}}}}(s,a)-\partial_x Q^{\pi^{*}}(s,a)}_\infty^2
  .\\
    \end{align*}
In the first inequality, we simply plug in \Cref{eq:boundedvarianceestimator}. In the third inequality we use the convergence of soft Q-learning and in the second inequality, we use the following identity (cf. \Cref{eq:diffvsmallerdiffq}):
\begin{equation*}
    \begin{aligned}
        &\norm{\partial_x V^{\pi^o_{x_t,\xi}}_{\lambda,x,\xi}(s)-\partial_x{V^{\pi^*_{x_t,\xi}}_{\lambda,x,\xi}(s)}}_\infty^2\\
        =&\norm{\sum_a \pi^o_{x_t,\xi}(a;s) \partial_x  Q^{\pi^o_{x_t,\xi}}_{\lambda,x,\xi}(s,a)-\sum_a \pi^*_{x_t,\xi}(a;s) \partial_x{Q^{\pi^*_{x_t,\xi}}_{\lambda,x,\xi}(s,a)}}_\infty^2\\
        =& \norm{\sum_a \left(\pi^o_{x_t,\xi}(a;s)-\pi^*_{x_t,\xi}(a;s)\right) \partial_x  Q^{\pi^o_{x_t,\xi}}_{\lambda,x,\xi}(s,a)-\sum_a \pi^*_{x_t,\xi}(a;s) \left(\partial_x{Q^{\pi^*_{x_t,\xi}}_{\lambda,x,\xi}(s,a)}-\partial_x{Q^{\pi^o_{x_t,\xi}}_{\lambda,x,\xi}(s,a)}\right)}_\infty^2\\
        \leq& 2|\mathcal{A}|^2\left(\frac{K_1(\overline{R}+\lambda \log |\mathcal{A}|)}{(1-\gamma)^2} \right)^2 \norm{\pi^o_{x_t,\xi}-\pi^*_{x_t,\xi}}_\infty+2\norm{\partial_x{Q^{\pi^*_{x_t,\xi}}_{\lambda,x,\xi}(s,a)}-\partial_x{Q^{\pi^o_{x_t,\xi}}_{\lambda,x,\xi}(s,a)}}_\infty^2
        ,
    \end{aligned}
\end{equation*}
where we use \Cref{eq:triangleinequalitysquared} for the last inequality.

To bound {\textcolor{magenta}{\textbf{(B)}}}, it only remains to upper bound $\EE\norm{{\partial_x Q^{\pi^{t_{k}}}}(s,a)-\partial_x Q^{\pi^*}(s,a)}_\infty^2$ and $\EE\norm{{\partial_x Q^{\pi^{t_{k+1}}}}(s,a)-\partial_x Q^{\pi^*}(s,a)}_\infty^2$. Below we will derive an upper bound for the former term. The same analysis yields an equivalent bound for the latter term. We denote by $V^*_{\lambda}$ the optimal regularized value function. Note, the analysis is similar to the one performed in the proof of \Cref{thm:convergence},
    \begin{align*}
   &\norm{\partial_x Q^{\pi^{t_k}}(s,a)-\partial_x Q^*_{\lambda}(s,a)}_\infty \\
   = &\Bigg|\Bigg|\sum_{t=0}^\infty \sum_{s',a'} \gamma^t p_{x,\xi}(s,a \rightarrow s',a';t,\pi^{t_k}) \left(\frac{d r_{x,\xi}(s',a') }{dx}+\gamma \sum_{s''} \frac{d P_{x,\xi}(s'';s',a') }{dx}V^{\pi^{t_k}}(s'') \right) \\
    &- \sum_{t=0}^\infty \sum_{s',a'} \gamma^t p_{x,\xi}(s,a \rightarrow s',a';t,\pi^*_{x_t,\xi}) \left(\frac{d r_{x,\xi}(s',a') }{dx}+\gamma \sum_{s''} \frac{d P_{x,\xi}(s'';s',a') }{dx}V^*_{\lambda}(s'') \right) \Bigg|\Bigg|_{\infty}\\
   =&\Bigg|\Bigg|\frac{d r_{x,\xi}(s,a) }{dx} + \gamma \sum_{s'} \frac{d P_{x,\xi}(s';s,a) }{dx}V^{\pi^{t_k}}(s')\\
        &+ \sum_{t=1}^\infty \sum_{s',a'} \gamma^t p_{x,\xi}(s,a \rightarrow s',a';t,\pi^{t_k}) \left(\frac{d r_{x,\xi}(s',a') }{dx}+\gamma \sum_{s''} \frac{d P_{x,\xi}(s'';s',a') }{dx}V^{\pi^{t_k}}(s'') \right) \\
        &- \frac{d r_{x,\xi}(s,a) }{dx} - \gamma \sum_{s'} \frac{d P_{x,\xi}(s';s,a) }{dx}V^*_{\lambda}(s')\\
    &- \sum_{t=1}^\infty \sum_{s',a'} \gamma^t p_{x,\xi}(s,a \rightarrow s',a';t,\pi^*_{x_t,\xi}) \left(\frac{d r_{x,\xi}(s',a') }{dx}+\gamma \sum_{s''} \frac{d P_{x,\xi}(s'';s',a') }{dx}V^*_{\lambda}(s'') \right) \Bigg|\Bigg|_{\infty}  \\
\leq& \gamma |\cS| \norm{\frac{d P_{x,\xi}(s';s,a) }{dx}}_\infty \norm{V^{\pi^{t_k}}(s')-V^*_{\lambda}(s') }_{\infty}\\
    &+ \gamma \Bigg|\Bigg| \sum_{s',a'}  P(s';s,a)  \pi^*_{x_t,\xi}(a',s') \sum_{t=0}^\infty \sum_{s'',a''} \gamma^t p_{x,\xi}(s',a' \rightarrow s'',a'';t,\pi^*_{x_t,\xi}) \dots \\
    &\quad \left(\frac{d r_{x,\xi}(s'',a'') }{dx}+\gamma \sum_{s'''} \frac{d P_{x,\xi}(s''';s'',a'') }{dx}V^*_{\lambda}(s'') \right)\\
    &-   \sum_{s',a'}  P(s';s,a)  \pi^{t_k}(a',s') \sum_{t=0}^\infty \sum_{s'',a''} \gamma^t p_{x,\xi}(s',a' \rightarrow s'',a'';t,\pi^{t_k}) \dots \\
    &\quad \left(\frac{d r_{x,\xi}(s'',a'') }{dx}+\gamma \sum_{s'''} \frac{d P_{x,\xi}(s''';s'',a'') }{dx}V^{\pi^{t_k}}(s'') \right) \Bigg|\Bigg|_{\infty}\\
    \leq& \gamma |\cS| \norm{\frac{d P_{x,\xi}(s';s,a) }{dx}}_\infty \norm{V^{\pi^{t_k}}(s')-V^*_{\lambda}(s') }_{\infty}\\
    &+\gamma \Bigg|\Bigg| \sum_{s',a'}  P(s';s,a)  \pi^*_{x_t,\xi}(a',s')  \partial_x{Q^*_{\lambda}(s',a')}\\
    &-   \sum_{s',a'}  P(s';s,a)  \pi^{t_k}(a',s')  \partial_x{Q^{\pi^{t_k}}(s',a')} \Bigg|\Bigg|_{\infty}\\
    \leq& \gamma |\cS| \norm{\frac{d P_{x,\xi}(s';s,a) }{dx}}_\infty\norm{V^{\pi^{t_k}}(s')-V^*_{\lambda}(s') }_{\infty}\\
        &+\gamma  \norm{\partial_xQ^*_{\lambda}(s',a')}_{\infty}\norm{\sum_{s'}  P(s';s,a) }_\infty \norm{ \pi^*_{x_t,\xi}-  \pi^{t_k}}_\infty\\
    &+ \gamma \sum_{s',a'}  P(s';s,a)  \pi^{t_k}(a',s') \norm{\partial_x{Q^*_{\lambda}(s',a')}-\partial_x{Q^{\pi^{t_k}}(s',a')}  }_\infty\\
        \leq& \gamma |\cS| \norm{\frac{d P_{x,\xi}(s';s,a) }{dx}}_\infty \norm{V^{\pi^{t_k}}(s')-V^*_{\lambda}(s') }_{\infty}\\
        &+\gamma \norm{\partial_xQ^*_{\lambda}(s',a')}_{\infty} \norm{ \pi^*_{x_t,\xi}-  \pi^{t_k}}_\infty\\
    &+ \gamma\norm{\partial_x{Q^*_{\lambda}(s',a')}-\partial_x{Q^{\pi^{t_k}}(s',a')}  }_\infty\\
    \leq&  \gamma |\cS| K_1 \norm{V^{\pi^{t_k}}(s')-V^*_{\lambda}(s') }_{\infty}\\
        &+\gamma \left(\frac{K_2}{1-\gamma}+\frac{K_1(\overline{R}+\lambda \log \cA)}{(1-\gamma)^2}\right) \norm{ \pi^*_{x_t,\xi}-  \pi^{t_k}}_\infty\\
    &+ \gamma \norm{\partial_x{Q^*_{\lambda}(s',a')}-\partial_x{Q^{\pi^{t_k}}(s',a')}  }_\infty\\
        \leq&  \frac{\gamma}{1-\gamma} |\cS| K_1 \norm{V^{\pi^{t_k}}(s')-V^*_{\lambda}(s') }_{\infty}\\
        &+\frac{\gamma}{1-\gamma} \left(\frac{K_2}{1-\gamma}+\frac{K_1(\overline{R}+\lambda \log \cA)}{(1-\gamma)^2}\right) \norm{ \pi^*_{x_t,\xi}-  \pi^{t_k}}_\infty\\
\end{align*}
The first equality follows from plugging in the result from \Cref{prop:policygradient}. The second equality follows by taking out all terms with $t=0$. The first inequality used the triangle inequality. The second inequality plugs back in the definition from \Cref{prop:policygradient}. The third inequality uses Cauchy-Schwarz. The fourth inequality follows from simplifying. The fifth inequality uses \Cref{eq:boundqderivative}. The sixth inequality uses the geometric sum.

To simplify the bound above we want to expres $\norm{V^{\pi^{t_k}}(s')-V^*_{\lambda}(s') }_{\infty}$ using $\norm{Q^{\pi^{t_k}}-Q^*_{\lambda} }_\infty$ and $\norm{{\pi^{t_k}}-{\pi^{*}}}_{\infty}$ (as we know the latter two converge for soft Q-learning). We have previously seen in \Cref{eq:bound_v_difference} that
\begin{equation*}
    \norm{V^{\pi^{t_k}}(s')-V^*_{\lambda}(s') }_{\infty}\leq \norm{Q^{\pi^{t_k}}-Q^*_{\lambda} }_{\infty}+\norm{{\pi^{t_k}}-\pi^* }_{\infty}\left(\frac{\overline{R}}{1-\gamma}+\lambda |\mathcal{A}||\log l_2|  + \frac{2}{l_2}\right).
\end{equation*}
Here we use $l_2$ to denote the minimum possible value that any policy output by soft Q-learning can achieve. Note $l_2>0$, which follows from \Cref{eq:minmax_softmaxpolicy}.

Plugging this result back in, we get
\begin{align*}
   &\norm{\partial_x Q^{\pi^{t_k}}(s,a)-\partial_x Q^*_{\lambda}(s,a)}_\infty \\
\leq &\frac{\gamma}{1-\gamma} |\cS| K_1 \left(\norm{Q^{\pi^{t_k}}-Q^*_{\lambda} }_{\infty}+\norm{{\pi^{t_k}}-\pi^* }_{\infty}\left(\frac{\overline{R}}{1-\gamma}+\lambda |\mathcal{A}||\log l_2|  + \frac{2}{l_2}\right)\right)\\
        &+\frac{\gamma}{1-\gamma} \left(\frac{K_2}{1-\gamma}+\frac{K_1(\overline{R}+\lambda \log \cA)}{(1-\gamma)^2}\right) \norm{ \pi^*_{x_t,\xi}-  \pi^{t_k}}_\infty\\
\end{align*}
From \cite{Mei2020Global}[Lemma 24] we know
    \[
    \norm{\pi^{t_k}-\pi^*}_\infty \leq  \frac{1}{\lambda} \norm{Q^{t_k} - Q^*_{\lambda}}_\infty.
    \]
With that result we can simplify to 
\begin{align*}
   &\norm{\partial_x Q^{\pi^{t_k}}(s,a)-\partial_x Q^*_{\lambda}(s,a)}_\infty \\
\leq& \frac{\gamma}{1-\gamma} |\cS| K_1 \left(\norm{Q^{\pi^{t_k}}-Q^*_{\lambda} }_{\infty}+ \frac{1}{\lambda} \norm{Q^{t_k} - Q^*_{\lambda}}_\infty\left(\frac{\overline{R}}{1-\gamma}+\lambda |\mathcal{A}||\log l_2|  + \frac{2}{l_2}\right)\right)\\
        &+\frac{\gamma}{1-\gamma} \left(\frac{K_2}{1-\gamma}+\frac{K_1(\overline{R}+\lambda \log \cA)}{(1-\gamma)^2}\right) \frac{1}{\lambda} \norm{Q^{t_k} - Q^*_{\lambda}}_\infty\\
        \leq& \norm{Q^{t_k} - Q^*_{\lambda}}_\infty \left(\frac{\gamma}{1-\gamma}|\cS| K_1\left(1+\frac{1}{\lambda} \frac{\overline{R}}{1-\gamma}+ |\mathcal{A}||\log l_2|  + \frac{2}{\lambda l_2}\right)+ \frac{\gamma}{1-\gamma} \left(\frac{K_2}{1-\gamma}+\frac{K_1(\overline{R}+\lambda \log \cA)}{(1-\gamma)^2}\right) \frac{1}{\lambda}\right).
\end{align*}
Therefore it follows that
\begin{align*}
   &\EE\left[\norm{\partial_x Q^{\pi^{t_k}}(s,a)-\partial_x Q^*_{\lambda}(s,a)}_\infty^2 \right]=\cO\left(\frac{1}{2^k}\right),
\end{align*}
where we use the convergence of soft Q-learning.

Using the above analysis, we can now bound:
     \begin{align*}
  {\textcolor{magenta}{\textbf{(B)}}}&\leq \cO(\frac{1}{2^k})+ 24 \EE\norm{{\partial_x Q^{\pi^{t_{k}}}}(s,a)-\partial_x Q^{\pi^*}(s,a)}_\infty^2+24 \EE\norm{{\partial_x Q^{\pi^{t_{k+1}}}}(s,a)-\partial_x Q^{\pi^{*}}(s,a)}_\infty^2\\
  &=\cO\left(\frac{1}{2^k}\right)
  .\\
  \end{align*}

This allows us to finish bounding {\textcolor{violet}{\textbf{(ii)}}} as follows:
\begin{align*}
    {\textcolor{violet}{\textbf{(ii)}}}&\leq 4\frac{M_2}{\lambda m}{\textcolor{cyan}{\textbf{(A)}}} +L_f^2 {\textcolor{magenta}{\textbf{(B)}}}\\
    &=\cO(\frac{1}{2^k}).
\end{align*}

Plugging this result into {\textcolor{green}{\textbf{(2)}}}, we thus have:
\begin{align*}
    {\textcolor{green}{\textbf{(2)}}} &\leq \sum_{k=1}^K \frac{2}{p_k}\left({\textcolor{yellow}{\textbf{(i)}}}+ {\textcolor{violet}{\textbf{(ii)}}}\right)\\
    &=\sum_{k=1}^K 2^k \left(\cO(\frac{1}{2^k})+\cO(\frac{1}{2^k})\right) \\
    &=\cO(K).
\end{align*}

Coming back to the start, we get for the variance our desired result as follows:
\begin{align*}
    \EE&\left[\norm{\frac{d}{dx} F^{RT}_{t_K}-\EE\left[\frac{d}{dx} F^{RT}_{t_K}\right]}_\infty^2\right]\\
   &\leq 4 {\textcolor{red}{\textbf{(1)}}}+4{\textcolor{green}{\textbf{(2)}}}+2{\textcolor{blue}{\textbf{(3)}}}\\
   &= \cO(K)+\cO(K)+\cO(1)\\
   &= \cO(K).
\end{align*}

\end{proof}

\subsection{Proof of \texorpdfstring{\Cref{prop:polgradlinear}}{}}
\begin{proof}
In this proof we will derive an expression for \[
\frac{df(x,\pi^*_{\lambda,x},\xi)}{dx}.
\]
Applying the Dominated Convergence Theorem then directly gives the expression for the derivative of $F(x)$. As we focus on $f$ we can drop any dependence on $\xi$ below to make the proof more readable and concise.

Let 
\begin{itemize}
    \item $p(\mu \rightarrow s,t,\pi,x)$ denote the probability given the leader's choise $x$ of reaching state $s$ after $t$ steps starting at $\mu$ and following policy $\pi$
        \item $p(\mu \rightarrow s,a,t,\pi,x)$ denote the probability under choice $x$ of reaching state $s$ after $t$ steps and then taking action $a$ starting at $\mu$ and following policy $\pi$ 
            \item $p(\mu \rightarrow s,a,s',t,\pi,x)$ denote the probability under choice $x$ of reaching state $s'$ after $t$ steps having previously been in state $s$ and having taken action $a$, starting at $\mu$ and following policy $\pi$
\end{itemize}
Assuming $\overline{V}(s)$ is differentiable for all $s$, we show the following statement by induction.
\begin{equation}
 \begin{aligned}
    \frac{df(x,\pi^*_{\lambda,x},\xi)}{dx}=&\sum_s \mu_x(s) \frac{d \log\mu_x}{dx} \overline{V}(s) +\sum_{t=1}^{n+1}\sum_s\sum_a\sum_s' \gamma^t p(\mu_x \rightarrow s,a,s',t,\pi^*_{\lambda,x},x)\frac{d\log P_x(s';s,a)}{dx}\overline{V}(s')\\
&+\sum_{t=0}^{n}\gamma^t \sum_s\sum_a p(\mu_x\rightarrow s,a,t,\pi^*_{\lambda,x},x)\left(\frac{1}{\lambda}\partial_x A_{\lambda,x}^{\pi^*_{\lambda,x}}\overline{Q}(s,a)+\frac{d\overline{r}_x(s,a)}{dx}\right)\\
&+\gamma^{n+1}\sum_s p(\mu_x\rightarrow s,t,\pi^*_{\lambda,x},x) {\partial_x\overline{V}(s')}.
\end{aligned}
\label{eq:inductiongradf}
\end{equation}

Note that taking $n\rightarrow\infty$ then directly proves our claim.

\paragraph{Base case \boldsymbol{$n=0$}}
We prove the statement for $n=0$
\begin{align*}
     \frac{df(x,\pi^*_{x},\xi)}{dx}=&\frac{d}{dx}\sum_s \mu(s) \sum_a \pi^*_{x}(a;s) \overline{Q}(s,a)\\
     = &\sum_s \frac{d\mu_x(s)}{dx} \overline{V}(s) + \sum_s \mu_x(s)\sum_a\pi^*_{\lambda,x}(a;s)\left(\frac{1}{\lambda}\partial_x A_{\lambda,x}^{\pi^*_{\lambda,x}} +  \partial_x \overline{Q}(s,a)\right)\\
     = & \sum_s \frac{d\mu_x(s)}{dx} \overline{V}(s) + \sum_s \mu_x(s)\sum_a\pi^*_{\lambda,x}(a;s) \Biggl(\frac{1}{\lambda}\partial_x A_{\lambda,x}^{\pi^*_{\lambda,x}} +  \frac{d\overline{r}_x(s,a)}{dx} \\&+ \gamma \sum_{s'}\left(P_x(s';s,a)\frac{d \log P_x(s';s,a)}{dx} \overline{V}(s')+P_x(s';s,a){\partial_x\overline{V}(s')}\right)\Biggr)\\
    = &\sum_s \mu_x(s) \frac{d \log\mu_x}{dx} \overline{V}(s) +\sum_{t=1}^{1}\sum_s\sum_a\sum_{s'} \gamma^t p(\mu_x \rightarrow s,a,s',t,\pi^*_{\lambda,x},x)\frac{d\log P_x(s';s,a)}{dx}\overline{V}(s')\\
&+\sum_{t=0}^{0}\gamma^t \sum_s\sum_a p(\mu_x\rightarrow s,a,t,\pi^*_{\lambda,x},x)\left(\frac{1}{\lambda}\partial_x A_{\lambda,x}^{\pi^*_{\lambda,x}}\overline{Q}(s,a)+\frac{d\overline{r}_x(s,a)}{dx}\right)\\&+\gamma^{1}\sum_s p(\mu_x\rightarrow s,t,\pi^*_{\lambda,x},x) {\partial_x\overline{V}(s')},\\
\end{align*}
where we use the definition of $f$ in the first equality. The second equality follows from\Cref{prop:policy_gradient_configuration} and the product rule. Rearranging terms gives the third equality.

\paragraph{Induction step \boldsymbol{ $n\implies n+1$}}
Assuming \Cref{eq:inductiongradf} holds for $n$, we prove it for $n+1$.
\begin{align*}
    &\sum_s \mu_x(s) \frac{d \log\mu_x}{dx} \overline{V}(s) +\sum_{t=1}^{n+1}\sum_s\sum_a\sum_s' \gamma^t p(\mu_x \rightarrow s,a,s',t,\pi^*_{\lambda,x},x)\frac{d\log P_x(s';s,a)}{dx}\overline{V}(s')\\
&+\sum_{t=0}^{n}\gamma^t \sum_s\sum_a p(\mu_x\rightarrow s,a,t,\pi^*_{\lambda,x},x)\left(\frac{1}{\lambda}\partial_x A_{\lambda,x}^{\pi^*_{\lambda,x}}\overline{Q}(s,a)+\frac{d\overline{r}_x(s,a)}{dx}\right)\\&+\gamma^{n+1}\sum_s p(\mu_x\rightarrow s,t,\pi^*_{\lambda,x},x) {\partial_x\overline{V}(s')}\\
= &   \sum_s \mu_x(s) \frac{d \log\mu_x}{dx} \overline{V}(s) +\sum_{t=1}^{n+1}\sum_s\sum_a\sum_s' \gamma^t p(\mu_x \rightarrow s,a,s',t,\pi^*_{\lambda,x},x)\frac{d\log P_x(s';s,a)}{dx}\overline{V}(s')\\
&+\sum_{t=0}^{n}\gamma^t \sum_s\sum_a p(\mu_x\rightarrow s,a,t,\pi^*_{\lambda,x},x)\left(\frac{1}{\lambda}\partial_x A_{\lambda,x}^{\pi^*_{\lambda,x}}\overline{Q}(s,a)+\frac{d\overline{r}_x(s,a)}{dx}\right)\\&+\gamma^{n+1}\sum_s p(\mu_x\rightarrow s,t,\pi^*_{\lambda,x},x) \sum_a\pi^*_{\lambda,x}(a;s) \Biggl(\frac{1}{\lambda}\partial_x A_{\lambda,x}^{\pi^*_{\lambda,x}} +  \frac{d\overline{r}_x(s,a)}{dx} \\&+ \gamma \sum_{s'}\left(P_x(s';s,a)\frac{d \log P_x(s';s,a)}{dx} \overline{V}(s')+P_x(s';s,a)\frac{d\overline{V}(s')}{dx}\right)\Biggr)\\
=&\sum_s \mu_x(s) \frac{d \log\mu_x}{dx} \overline{V}(s) +\sum_{t=1}^{n+2}\sum_s\sum_a\sum_s' \gamma^t p(\mu_x \rightarrow s,a,s',t,\pi^*_{\lambda,x},x)\frac{d\log P_x(s';s,a)}{dx}\overline{V}(s')\\
&+\sum_{t=0}^{n+1}\gamma^t \sum_s\sum_a p(\mu_x\rightarrow s,a,t,\pi^*_{\lambda,x},x)\left(\frac{1}{\lambda}\partial_x A_{\lambda,x}^{\pi^*_{\lambda,x}}\overline{Q}(s,a)+\frac{d\overline{r}_x(s,a)}{dx}\right)\\&+\gamma^{n+2}\sum_s p(\mu_x\rightarrow s,t,\pi^*_{\lambda,x},x) {\partial_x\overline{V}(s')},
\end{align*}
which proves our claim. In the first equality we use the definition of $\overline{V}$ and the product rule. The second inequality follows from collecting terms.
\end{proof}

\subsection{Auxiliary Results}
\label{app:auxresults}
\begin{proposition}
    [Gradient of Best response Policy]
    \label{prop:policy_gradient_configuration}
    It holds that
    \begin{equation*}
        \frac{d \pi^*_{x,\xi}(s,a)}{d x}= \frac{1}{\lambda} \pi^*_{x,\xi}(a;s) \partial_x {A^{\pi^*_{x,\xi}}_{\lambda,x,\xi}(s,a)}.
    \end{equation*}
\end{proposition}
\begin{proof}
For a given $x,\xi$, this result was previoysly shown by \cite{chen2022adaptive}. We give a short proof below. 
    \[
    \begin{aligned}
        \frac{d \pi^*_{x,\xi}(s,a) }{d x} =&  \frac{d }{d x} \frac{\exp(Q^*_{\lambda,x,\xi}(s,a)/\lambda)}{\sum_a' \exp(Q^*_{\lambda,x,\xi}(s,a')/\lambda)}\\
        = &\ \frac{
                \exp(Q^*_{\lambda,x,\xi}(s,a)/\lambda) \cdot \frac{\partial_x Q^*_{\lambda,x,\xi}(s,a)}{\lambda} \sum_{a'} \exp(Q^*_{\lambda,x,\xi}(s,a')/\lambda)
                }{ \left( \sum_{a''} \exp(Q^*_{\lambda,x,\xi}(s,a)/\lambda) \right)^{2} }\\
                &- \frac{\exp(Q^*_{\lambda,x,\xi}(s,a)/\lambda) \sum_{a'} \exp(Q^*_{\lambda,x,\xi}(s,a')/\lambda) \frac{\partial_x Q^*_{\lambda,x,\xi}(s,a')}{\lambda}}{ \left( \sum_{a''} \exp(Q^*_{\lambda,x,\xi}(s,a)/\lambda) \right)^{2} }\\
        = &\ \pi^*(a;s) \frac{\partial_x Q^*_{\lambda,x,\xi}(s,a')}{\lambda}
            - \frac{\pi^*(a;s)\sum_{a'} \exp(Q^*_{\lambda,x,\xi}(s,a')/\lambda) \frac{\partial_x Q^*_{\lambda,x,\xi}(s,a')}{\lambda}}{ \sum_{a''} \exp(Q^*_{\lambda,x,\xi}(s,a)/\lambda) } \\
        = &\ \frac{1}{\lambda} \pi^*(a;s) \partial_x Q^*_{\lambda,x,\xi}(s,a')
        - \frac{1}{\lambda} \pi^*(a;s) \sum_{a'} \pi^*(a';s) \partial_x Q^*_{\lambda,x,\xi}(s,a') \\
        = &\ \frac{1}{\lambda} \pi^*(a;s) \left[
        \partial_x Q^*_{\lambda,x,\xi}(s,a') - \partial_x \EE_{a' \sim \pi^*(\cdot ; s)} [Q^*_{\lambda,x,\xi}(s,a')]
        \right] \\
        = &\ \frac{1}{\lambda} \pi^*(a;s) \left(\partial_x Q^*_{\lambda,x,\xi}(s,a)-\partial_x V^*_{\lambda,x,\xi}(s)\right)\\
        = &\ \frac{1}{\lambda} \pi^*(a;s) \partial_x{A^*_{\lambda,x,\xi}(s,a)}.
    \end{aligned}.
    \]
    The second equality follows from the quotient rule. The third and fourth equality follows from the definition of   $\pi^*_{x,\xi}(s,a)$. The remaining equalities leverage the definition of the advantage function.

    Note, we can replace the total with the partial derivative of $Q^*(x) = \max_{\pi} Q(\pi, x)$, due to Rockafellar's thoerem because $Q$ is continuously differentiable in  x for all $ \pi $ and $\Pi$ is compact and convex \citep{Blondel2024Elements,Rockafellar1998Variational}. 
\end{proof}

\begin{proposition}[Unbiased advantage derivative estimator]
\label{prop:unbiasedgradient}
    The output $\widehat{\partial_x A^{\pi^o_{x,\xi}}_{\lambda,x,\xi}}(s,a)$ of Algorithm \ref{alg:sample_gradient_estimation} is an unbiased estimate of $\partial_x A^{\pi^o_{x,\xi}}_{\lambda,x,\xi}(s,a)$, i.e \[\EE\left[\widehat{\partial_xA^{\pi^o_{x,\xi}}_{\lambda,x,\xi}}(s,a)\right]= \partial_x A^{\pi^o_{x,\xi}}_{\lambda,x,\xi}(s,a).\]
\end{proposition}
\begin{proof}
We drop any dependence on $x,\xi,\pi^o_{x,\xi}$ for notational clarity. We further emphasize that the trick of truncating a rollout after a geomtrically sampled time to obtain unbiased gradients is commonly used in the RL literature for obtaining unbiased estimates of the standard policy gradient \cite{Zhang2020Global}.

We show that the estimator $\widehat{\partial_x Q_{\lambda}}(s, a)$ given by \[
\sum_{t=0}^{T_Q} \frac{d}{dx} r(s_t,a_t)+
\frac{\gamma}{1-\gamma} \frac{d}{dx} \log P(s_{T_Q+1};s_{T_Q},a_{T_Q}) \sum_{t=T_Q+1}^{T_Q+T'_Q+1} \gamma^{(t-T_Q-1)/2} \left(r(s_t,a_t)+\lambda H(\pi(\cdot ;s_t))\right)
\]
is unbiased. The same argument then holds for $\widehat{\frac{d}{dx} V_{\lambda}}(s)$ and implies that $\widehat{\frac{d}{dx} A_{\lambda}}(s,a)$ is unbiased.
First of all, we have:
\begin{align}
    &\EE\left[\sum_{t=0}^{T'_Q} \gamma^{(t)/2} \left(r(s_t,a_t)+\lambda H(\pi(\cdot ;s_t))\right)\right]\nonumber\\
    =&\EE_{T_Q'}\EE_{s}^{\pi}\left[\sum_{t=0}^{T'_Q} \gamma^{(t)/2} \left(r(s_t,a_t)+\lambda H(\pi(\cdot ;s_t))\right)\right]\nonumber\\
    =&\EE_{s}^{\pi}\left[\EE_{T_Q'}\sum_{t=0}^{T'_Q} \gamma^{(t)/2} \left(r(s_t,a_t)+\lambda H(\pi(\cdot ;s_t))\right)\right]\nonumber\\
    =&\EE_{s}^{\pi}\left[\sum_{t=0}^{\infty}\EE_{T_Q'}\left[\mathds{1}_{t\leq T_Q'}\right] \gamma^{(t)/2} \left(r(s_t,a_t)+\lambda H(\pi(\cdot ;s_t))\right)\right]\label{eq:proof_prop_3}\\
    =&\EE_{s}^{\pi}\left[\sum_{t=0}^{\infty}\gamma^t \left(r(s_t,a_t)+\lambda H(\pi(\cdot ;s_t))\right)\right]\nonumber\\
    =& V_{\lambda}(s) \label{eq:proof_prop_3_2},
\end{align}
where we use Fubini's theorem for \cref{eq:proof_prop_3} and the Dominated Convergence Theorem and the fact that $T_Q' \sim \text{Geo}(1-\gamma^{0.5})$ in \cref{eq:proof_prop_3_2}.
Because $T_Q$ and $T_Q'$ are sampled independently, it immediately follows that

\begin{align*}
   & \EE_{T_Q,T_Q'}\EE_{s,a}^{\pi}\Bigg[\sum_{t=0}^{T_Q} \frac{d}{dx} r(s_t,a_t)\\
   &+
\frac{\gamma}{1-\gamma} \frac{d}{dx} \log P(s_{T_Q+1};s_{T_Q},a_{T_Q}) \sum_{t=T_Q+1}^{T_Q+T'_Q+1} \gamma^{(t-T_Q-1)/2} \left(r(s_t,a_t)+\lambda H(\pi(\cdot ;s_t))\right)\Bigg]\\
=&\EE_{T_Q}\EE_{s,a}^{\pi}\Biggl[
\underbrace{\sum_{t=0}^{T_Q} \frac{d}{dx} r(s_t,a_t)}_{\color{blue}\text{(1)}}
+
\underbrace{\frac{\gamma}{1-\gamma} \frac{d\log P(s_{T_Q+1};s_{T_Q},a_{T_Q})}{dx}  V_{\lambda}^{\pi}(s_{T_Q})}_{\color{red}\text{(2)}}
\Biggr].
\end{align*}
We seperately show that the two summands are unbaised estimates. Then by linearity of expectation the result follows.
For $\color{blue}\text{(1)}$ using Fubini's theorem and Dominated Convergence Theorem it holds that
\begin{align*}
    \EE_{T_Q}\EE_{s,a}^{\pi}\left[
\sum_{t=0}^{T_Q} \frac{d}{dx} r(s_t,a_t) \right]&=\EE_{s,a}^{\pi}\left[
\sum_{k=0}^{\infty} (1-\gamma) \gamma^k \sum_{t=0}^k \frac{d}{dx} r(s_t,a_t) \right]\\
&= (1-\gamma) \EE_{s,a}^{\pi}\left[
\sum_{t=0}^{\infty}  \sum_{k=t}^\infty \gamma^k\frac{d}{dx} r(s_t,a_t) \right]\\
&= (1-\gamma) \EE_{s,a}^{\pi}\left[
\sum_{t=0}^{\infty}  \frac{\gamma^t}{1-\gamma} \frac{d}{dx} r(s_t,a_t) \right]\\
&=\EE_{s,a}^{\pi}\left[
\sum_{t=0}^{\infty}  \gamma^t \frac{d}{dx} r(s_t,a_t) \right].
\end{align*}

Similarly for $\color{red}\text{(2)}$, we have using Fubini and Dominated Convergence Theorem that
\begin{align*}
    &\EE_{T_Q}\EE_{s,a}^{\pi}\left[
\frac{\gamma}{1-\gamma} \frac{d\log P(s_{T_Q+1};s_{T_Q},a_{T_Q})}{dx}  V_{\lambda}^{\pi}(s_{T_Q})
\right]\\
=&\frac{\gamma}{1-\gamma} \EE_{s,a}^{\pi}\EE_{T_Q}\left[ \sum_{t=0}^{\infty} \mathds{1}_{t=T_Q}
 \frac{d\log P(s_{t+1};s_{t},a_{t})}{dx}  V_{\lambda}^{\pi}(s_{t})
\right]\\
=& \EE_{s,a}^{\pi}\left[ \sum_{t=0}^{\infty} \gamma^{t+1}
 \frac{d\log P(s_{t+1};s_{t},a_{t})}{dx}  V_{\lambda}^{\pi}(s_{t})
\right].
\end{align*}
Plugging these results back in, we have
\begin{align*}
    &\EE_{T_Q}\EE_{s,a}^{\pi}\left[\sum_{t=0}^{T_Q} \frac{d}{dx} r(s_t,a_t)
+\frac{\gamma}{1-\gamma} \frac{d\log P(s_{T_Q+1};s_{T_Q},a_{T_Q})}{dx}  V_{\lambda}^{\pi}(s_{T_Q})
\right]\\
=&\EE_{s,a}^{\pi}\left[\sum_{t=0}^{\infty}  \gamma^t \frac{d}{dx} r(s_t,a_t)+ \gamma^{t+1}
 \frac{d\log P(s_{t+1};s_{t},a_{t})}{dx}  V_{\lambda}^{\pi}(s_{t})
\right]\\
=& {\partial_x{Q^{\pi}_{\lambda}(s,a)}},
\end{align*}
which proves the proposition.
\end{proof}

\begin{proposition}[Unbiased gradient estimator for $F$]
\label{prop:unbiaseddecomposedgradient}
The gradient estimator described in \Cref{alg:decomp_gradient_estimation} is unbiased for the given policy $\pi^o_{x,\xi}$.
\end{proposition}
\begin{proof}
   We need to show that:
\begin{align*}
  & \EE\Biggl[\underbrace{\left(\sum_{t=0}^T \frac{d}{dx} \overline{r}(s_t,a_t) \right)}_{\color{blue}\text{(1)}}  + \underbrace{\frac{1}{\lambda(1-\gamma)} \partial_x \widehat{A^{\pi^o_{x,\xi}}_{\lambda,x,\xi}} (s_T,a_T) \sum_{t'=T}^{T+T'} \gamma^{(t-T)/2} \overline{r}(s_{t'},a_{t'})}_{\color{red}\text{(2)}}\\
  &+\underbrace{\frac{1}{1-\gamma} \partial_x \log P(s_T,a_{T-1},s_{T-1}) \sum_{t'=T}^{T+T'} \gamma^{(t'-T)/2}\overline{r}(s_{t'},a_{t'})}_{\color{green}\text{(3)}}\Biggr]\\
  =&\EE_{\xi}\Bigg[ \EE^{\pi^o_{x,\xi}}_{s_0\sim \mu} \Bigg[ \sum_{t=0}^\infty \gamma^t \Biggl( \frac{1}{\lambda} \partial_x A_{\lambda,x,\xi}^{\pi^o_{x,\xi}}(s_t,a_t) \overline{Q}(s_t,a_t)+\frac{d}{dx} \overline{r}(s_t,a_t) + \partial_x \log P_{x,\xi}(s_t,a_{t-1},s_{t-1}) \overline{V}(s_t)\Biggr) \Bigg] \Bigg].
\end{align*}

We can show the claim seperately for {\color{blue}\text{(1)}}, {\color{red}\text{(2)}} and {\color{green}\text{(3)}}. Note for {\color{blue}\text{(1)}} and {\color{green}\text{(3)}} the claim directly follows from the proof of \Cref{prop:unbiasedgradient}. And the proof for {\color{red}\text{(2)}} works almost identical to the one for {\color{green}\text{(3)}}, relying on the fact that a truncation via a geometric distribution is identical to an infinite trajectory with a discount factor.
\end{proof}

\subsection{Convergence Results for Popular RL Algorithms}
\label{app:convergenceRL}
For the next Proposition, consider the following soft Bellmann optimality operator, which has been shown to be a contraction \cite{Dai2018SBEED,Nachum2017Bridging}. 
\begin{equation}
\label{eq:valuebellmannop}
\left(\mathcal{T}^*_\lambda V_\lambda\right)(s):=\lambda \log \left(\sum_{a \in \mathcal{A}} \exp \left(\frac{r(s, a)+\gamma \mathbb{E}_{s^{\prime} \mid s, a}\left[V_\lambda\left(s^{\prime}\right)\right]}{\lambda}\right)\right).
\end{equation}
Using \Cref{eq:valuebellmannop}, one can define a standard soft value iteration algorithm (see \Cref{alg:valueiteration} in \Cref{app:algorithms}). We show soft value iteration satisfies \Cref{assumption:algorithm}.
\begin{proposition}
    \label{prop:valueiterconv}
    Algorithm \ref{alg:valueiteration} converges, such that $\norm{\pi^*_{x,\xi}-\pi^o_{x,\xi}}^2_{\infty} \leq \delta^2$ after $T$ iterations, where $T=\mathcal{O}(\log 1/\delta)$.
\end{proposition}
\begin{proof}
    From \cite{Mei2020Global}[Lemma 24] we have 
    \[
    \norm{\pi^o-\pi^*}_\infty \leq  \norm{\pi^o-\pi^*}_1 \leq \frac{1}{\lambda} \norm{Q^T_\lambda - Q^*_{\lambda}}_\infty.
    \]
    Moreover 
    \[
    \frac{1}{\lambda} \norm{Q^T_\lambda - Q^*_{\lambda}}_\infty \leq \frac{1}{\lambda} \norm{V^T_\lambda - V^*_{\lambda}}_\infty \leq \frac{\gamma^T}{\lambda} \norm{V^*_{\lambda}} \leq \frac{\gamma^T}{\lambda(1-\gamma)} \left(\overline{R}+ \lambda \log |\mathcal{A}| \right),
    \]
    where we use the contraction property shown in \cite{Dai2018SBEED,Nachum2017Bridging} and the fact that we instantiate $V_\lambda$ with 0, such that no value iterate can ever be larger than  higher than $\left(\overline{R}+ \lambda \log |\mathcal{A}| \right)$.
    The claim follows from $\delta \leq \cO(\gamma^T)$.
\end{proof}

As soft value iteration assumes knowledge of the transition function and scales badly when the state and action space are large, in practice stochastic methods such as soft Q-learning are used instead. For this method, consider the soft Bellman state-action optimality operator~\cite{Asadi2017alternative,Haarnoja2017Reinforcement}:
\begin{equation}
\label{eq:bellmanstateaction}
\left(\mathcal{T}_\lambda^* Q_\lambda\right)(s,a):= r(s,a) + \gamma \EE_{s'\sim P(\cdot|s,a)} \left[\lambda \log \left(\sum_{a' \in \mathcal{A}} \exp \left(\frac{Q_\lambda(s,a')}{\lambda}\right)\right)\right].
\end{equation}
We can use \Cref{eq:bellmanstateaction} to run soft Q-learning, as described in \Cref{alg:qlearning} in \Cref{app:algorithms}. Equivalently to soft value iteration, we can show soft Q-learning satisfies \Cref{assumption:algorithm}.

\begin{proposition}
\label{prop:qlearningconv}
    Let $\pi_B$ be sufficiently exploratory, such that the induced Markov chain is ergodic. Then soft Q-learning converges, such that $\EE_{o}\left[\norm{\pi^*_{x,\xi}-\pi^o_{x,\xi}}^2_{\infty}\right] \leq \delta^2$ after $T$ iterations, where $T=\mathcal{O}(\frac{\log (1/\delta)}{\delta^2})$.
\end{proposition}

We use the following Theorem from \cite{Qu2020Finite} to prove our claim:

{\itshape Theorem \cite{Qu2020Finite} Let $ x \in \mathbb{R}^{d}$, and $F$ : $\mathbb{R}^{d} \rightarrow \mathbb{R}^{d}$ be an operator. We use $F_i$ to denote the $i$ 'th entry of $F$. We consider the following stochastic approximation scheme that keeps updating $x(t) \in \mathbb{R}^{d}$ starting from $x(0)$ being the all zero vector,
$$
\begin{array}{ll}
x_i(t+1)=x_i(t)+\alpha_t\left(F_i(x(t))-x_i(t)+w(t)\right) & \text { for } i=i_t, \\
x_i(t+1)=x_i(t) & \text { for } i \neq i_t,
\end{array}
$$
where $i_t \in \{1,\dots,d\}$ is a stochastic process adapted to a filtration $\mathcal{F}_t$, and $w(t)$ is some noise. Assume the following:

Assumption 1 (Contraction) (a) Operator $F$ is $\gamma$ contraction in $\norm{\cdot}_\infty$, i.e. for any $x, y \in \mathbb{R}^{d}$, $\norm{F(x)-F(y)}_\infty \leq \gamma\norm{x-y}_\infty$. (b) There exists some constant $C>0$ s.t. $\norm{F(x)}_\infty \leq \gamma\norm{x}_\infty + C, \forall x \in \mathbb{R}^{d}$.

Assumption 2 (Martingale Difference Sequence) $w(t)$ is $\mathcal{F}_{t+1}$ measurable and satisfies $\mathbb{E} w(t) \mid \mathcal{F}_t=$ 0 . Further, $|w(t)| \leq \bar{w}$ almost surely for some constant $\bar{w}$.

Assumption 3 (Sufficient Exploration) There exists a $\sigma \in(0,1)$ and positive integer, $\tau$, such that, for any $i \in \mathcal{N}$ and $t \geq \tau, \mathbb{P}\left(i_t=i \mid \mathcal{F}_{t-\tau}\right) \geq \sigma$.

Suppose Assumptions 1,2 and 3 hold. Further, assume there exists constant $\bar{x} \geq\left\|x^*\right\|_\infty$ s.t. $\forall t,\|x(t)\|_\infty \leq \bar{x}$ almost surely. Let the step size be $\alpha_t=\frac{h}{t+t_0}$ with $t_0 \geq \max (4 h, \tau)$, and $h \geq \frac{2}{\sigma(1-\gamma)}$. Then, with probability at least $1-\delta$,
$$
\left\|x(T)-x^*\right\|_\infty \leq \frac{12 \bar{\epsilon}}{1-\gamma} \sqrt{\frac{(\tau+1) h}{\sigma}} \sqrt{\frac{\log \left(\frac{2(\tau+1) T^2 n}{\delta}\right)}{T+t_0}}+\frac{4}{1-\gamma} \max \left(\frac{16 \bar{\epsilon} h \tau}{\sigma}, 2 \bar{x}\left(\tau+t_0\right)\right) \frac{1}{T+t_0},
$$
where $\bar{\epsilon}=2 \bar{x}+C+\bar{w}$.
}

\begin{proof}
Our algorithm can be seen as a stochastic approximation scheme where we update $Q$ asynchronously just like $x$ above in the following way
$$
\begin{array}{ll}
Q_{s_t,a_t}(t+1)=Q_{s_t,a_t}(t)+\alpha_t\left(F_{s_t,a_t}(Q(t))-Q_{s_t,a_t}(t)+w_t\right) & { }  \\
Q_{s,a}(t+1)=Q_{s,a}(t) & \text { for } s,a \neq s_t,a_t,
\end{array}
$$
where
\[F_{s_t,a_t}(Q)=r(s,a) + \gamma \EE_{s'\sim P(\cdot|s,a)} \left[\lambda \log \left(\sum_{a' \in \mathcal{A}} \exp \left(\frac{Q(s,a')}{\lambda}\right)\right)\right],\] 
and the errors:
\begin{align*}
    w_t =& r(s_t,a_t)+\gamma \lambda \log \left(\sum_{a'\in \mathcal{A}} \exp \left(\frac{Q_\lambda(s_{t+1},a')}{\lambda}\right)\right)\\
    &- r(s_t,a_t)+\gamma \EE_{s'\sim P(\cdot;s_t,a_t)} \left[\lambda \log \left(\sum_{a'\in \mathcal{A}} \exp \left(\frac{Q_\lambda(s_{t+1},a')}{\lambda}\right)\right)\right].
\end{align*}

We now  show that $F$ satisfies  the assumptions of the Theorem from \cite{Qu2020Finite} and use the result to prove our own claim. 

In the following we let $\mathcal{F}_t$ be the $\sigma$--algebra generated by the random variables $(s_0,a_0,\cdots, s_t,a_t)$. 

First we restate the following identity from \cite{Nachum2017Bridging}
\begin{align*}
    T_{\lambda}^*(Q)(s,a)&=r(s,a) + \gamma \EE_{s'\sim P(\cdot|s,a)} \left[\lambda \log \left(\sum_{a' \in \mathcal{A}} \exp \left(\frac{Q(s,a')}{\lambda}\right)\right)\right] \\
    &= r(s,a) + \gamma \EE_{s'\sim P(\cdot|s,a)} \left[\max_{\pi} \langle Q(\cdot,s'),\pi \rangle +\lambda H(\pi;s') \right].
\end{align*}
We use it to show that $T_{\lambda}^*$ is a contraction. Indeed we have:
\begin{align*}
    &\norm{T_{\lambda}^*(Q_1)-T_{\lambda}^*(Q_2)}_{\infty} \\
    =&\Bigg|\Bigg|r(s,a) + \gamma \max_{\pi}\sum_{s'}P(s';s,a)\left(\langle Q_1(\cdot,s'),\pi \rangle +\lambda H(\pi;s')\right) \\
    &-r(s,a) - \gamma \max_{\pi}\sum_{s'}P(s';s,a)\left(\langle Q_2(\cdot,s'),\pi \rangle +\lambda H(\pi;s')\right)
    \Bigg|\Bigg|_{\infty}\\
    \leq& \gamma \norm{ \max_{\pi}\sum_{s'}P(s';s,a)\left(\langle Q_1(\cdot,s'),\pi \rangle -\langle Q_2(\cdot,s'),\pi \rangle \right)
    }_{\infty}\\
    \leq& \gamma \norm{Q_1-Q_2}_{\infty}.
\end{align*}

Moreover, it holds that \[F(Q)\leq \overline{R}+\gamma \norm{Q}_\infty +\lambda \log{|\mathcal{A}|}.\]
So we can set $C=\overline{R}+\lambda \log{|\mathcal{A}|}$

Next we note that $w_t$ is $F_{t+1}-$measurable (it depends on $s_{t+1}$) and that \[
\EE[w_t|\mathcal{F}_t]=0.
\]
Moreover $w(t)$ is bounded by $\overline{w}=\frac{2\gamma(\overline{R}+\lambda \log |\mathcal{A}|)}{1-\gamma}$.

Further we have assumed that the behavioural policy $\pi_B$ is sufficiently exploratory. Let $\tilde{\mu}$ be the corresponding stationary distribution, $\mu_{\min}=\inf_{s,a}\tilde{\mu}(s,a)$ and $t_{mix}$ the mixing time. Then \cite{Qu2020Finite} show that for $\sigma=\frac{1}{2}\mu_{\min}$ and $\tau=\lceil \log_2 (\frac{2}{\mu_{\min}})\rceil t_{mix}$ it holds that
\begin{equation}
    \forall s\in \cS, a\in \mathcal{A},\forall t\geq \tau : \mathbb{P}(s_t,a_t=s,a|\mathcal{F}_{t-\tau})\geq \sigma.
\end{equation}

Moreover, we note that $Q(t)$ and $Q^*_{\lambda}$ are bound by $\overline{x}= \frac{\overline{R}+\lambda \log |\mathcal{A}|}{1-\gamma}$.

Using the Theorem from \cite{Qu2020Finite} we thus have  the following result:

Let $\alpha_t=\frac{h}{t+t_0}$ with $t_0 \geq$ $\max \left(4 h,\left\lceil\log _2 \frac{2}{\mu_{\min }}\right\rceil t_{\mathrm{mix}}\right)$ and $h \geq \frac{4}{\mu_{\min }(1-\gamma)}$. Then, with probability at least $1-p$,
\begin{align*}
&\left\|Q(T)- Q^*_{\lambda}\right\|_{\infty} \leq \\ &\leq \frac{60 (\bar{R}+\lambda\log|\mathcal{A}|)}{(1-\gamma)^2} \sqrt{\frac{2\left(\left\lceil\log _2 \frac{2}{\mu_{\min }}\right\rceil t_{\mathrm{mix}}+1\right) h}{\mu_{\min }}} \sqrt{\frac{\log \left(\frac{2\left(\left\lceil\log _2 \frac{2}{\mu_{\min }}\right\rceil t_{\mathrm{mix}}+1\right) T^2|\mathcal{S} \| \mathcal{A}|}{p}\right)}{T+t_0}} \\
& +\frac{4 (\bar{R}+\lambda\log|\mathcal{A}|)}{(1-\gamma)^2} \max \left(\frac{160 h\left\lceil\log _2 \frac{2}{\mu_{\min }}\right\rceil t_{\mathrm{mix}}}{\mu_{\min }}, 2\left(\left\lceil\log _2 \frac{2}{\mu_{\min }}\right\rceil t_{\mathrm{mix}}+t_0\right)\right) \frac{1}{T+t_0}.\\
\end{align*}
We denote the bound above by {\color{red}\textbf{(A)}}. Let us choose $p=\delta^2$.\footnote{Note that $\delta < 1$, as it represents the distance between two softmax  policies under the supremum norm.} With probability $p$, $\left\|Q(T)-Q^*_{\lambda}\right\|_{\infty}$ is not bounded by {\color{red}\textbf{(A)}}. However, it is always upper bounded by $\frac{2(\Bar{R}+\lambda\log(|\mathcal{A}|))}{1-\gamma}$.

Setting $T=\cO (\frac{\log(1/\delta)}{\delta^2})$ , we get
\begin{align*}
     {\color{red}\textbf{(A)}}=& \cO\left(\sqrt{\frac{\log\left(\frac{T^2}{\delta^2}\right)}{T}}+\frac{1}{T}\right)\\
    =& \cO\left(\sqrt{\frac{\log\left(\frac{T}{\delta}\right)}{T}}\right)\\
\underset{T=\frac{\log\left(1/\delta\right)}{\delta^2}}{=}&\cO\left(\sqrt{\frac{\log\left(\frac{\frac{\log\left(1/\delta\right)}{\delta^2}}{\delta}\right)}{\frac{\log\left(1/\delta\right)}{\delta^2}}}\right)   \\
{=}&\cO\left(\delta\sqrt{\frac{\log\left(\frac{\log\left(1/\delta\right)}{\delta^3}\right)}{\log\left(1/\delta\right)}}\right)   \\
{=}&\cO\left(\delta\sqrt{\frac{-3\log \left(\delta\right) +\log\left(\log\left(1/\delta\right)\right)}{\log\left(1/\delta\right)}}\right)   \\
{=}&\cO\left(\delta\sqrt{\frac{-3\log \left(\delta\right)}{-\log\left(\delta\right)} +\frac{\log\left(\log\left(1/\delta\right)\right)}{\log\left(1/\delta\right)}}\right) .  \\
\end{align*}
Since $\delta<1$, $\log(1/\delta)>0$ and for $x>0$ it holds that $\frac{\log(x)}{x}<1/e$, such that we get:
\begin{align*}
     {\color{red}\textbf{(A)}}{\leq}&\cO\left(\delta\sqrt{3+\frac{1}{e}}\right)   \\
    =& \cO\left(\delta \right).
\end{align*}
 Using \cite{Mei2020Global}[Lemma 24], we arrive at:
\begin{align*}
        &\EE_o \left[\norm{\pi^o-\pi^*}^2_\infty  \right] \\
        \leq& 4 \EE_o \left[ \norm{Q^T_\lambda - Q^*_{\lambda}}^2_\infty\right]\\
        \leq& 4\left((1-p) {\color{red}\textbf{(A)}}^2 + p \left(\frac{2(\Bar{R}+\lambda\log(|\mathcal{A}|))}{1-\gamma}\right)^2\right)\\
        =& \cO (\delta^2).
\end{align*}
\end{proof}

A popular class of RL algorithms are policy gradient methods such as REINFORCE \cite{Williams1992Simple}. For the entropy-regularised problem, it generally makes sense to choose a softmax parametrization for the policy, as we know $\pi^*_{x,\xi}(s;a)\propto \exp(Q^*_{\lambda,x,\xi}(s,a)/\lambda)$~\cite{Mei2020Global}. We defer the details to \Cref{alg:pg}\ in \Cref{app:algorithms} and present the following convergence result, which shows using vanilla policy gradient for the lower level also fulfills \Cref{assumption:algorithm}---at least asymptotically.
\begin{proposition}
\label{prop:pggconv}
    Vanilla policy gradient with softmax parameterization converges, such that $\forall \delta, \exists T, \forall t\geq T: \norm{\pi^*_{x,\xi}-\pi^o_{t,x,\xi}} \leq \delta^2$, where $\pi^o_{t,x,\xi}$ is the computed policy after $t$ iterations.
\end{proposition}

\begin{proof}
As in most proofs we drop the subscripts for $x,\xi$. The proof is an adaptation of the one presented in \cite{Mei2020Global}[Lemma 16].
We denote by $\pi_t$ the iterates of the policies of the algorithm and by $V^{\pi_t}_{\lambda}(\mu)$ the corresponding value function with starting distribution $\mu$.
 It can be shown that $V^\pi_{\lambda}$ is $s$-smooth for some $s$ \cite{Mei2020Global}. Choosing a stepsize of $1/s$, we have that the value functions increase monotonically, i.e.
 \[ 
\forall t:  V^{\pi_{t+1}}_{\lambda}(\mu)\geq  V^{\pi{_t}}_{\lambda}(\mu).
 \]
 At the same time, it holds that: 
 \[
  V^{\pi{_t}}_{\lambda}(\mu)\leq \frac{\overline{R}+\lambda \log \cA}{1-\gamma}.
 \]
 By monotone convergence it thus follows that $V^{\pi{_t}}_{\lambda}(\mu) \rightarrow V^*_{\lambda}(\mu)$, where $V^*_{\lambda}(\mu)$ is the maximum possible value.

 Since $\pi_t \in \Delta(\cA)^{|\cS|}$ and $\Delta(\cA)^{|\cS|}$ is compact it follows that $\{\pi_t\}_t$ has a convergent subsequence  $\{\pi_{t_k}\}_k$. Denote by $\pi^*$ the limit of this subsequence. It has to hold that $V^{\pi^*}_\lambda(\mu)=V^*_{\lambda}(\mu)$ and thus $\pi^*$ is the optimal policy.

 Now assume that $\{\pi_t\}_t$ does not converge to $\pi^*$. In that case \[
 \exists \epsilon,\forall t, \exists t' \geq t: \norm{\pi^*-\pi_{t'}}_\infty > \epsilon.
 \]
 Note that due to entropy regularization $V^*_{\lambda}(\mu)$ is the unique maximum. This means that 
 \[
 \exists \kappa: \max \{V^\pi_{\lambda}|\norm{\pi-\pi^*}_\infty\geq\epsilon\} +\kappa < V^*_\lambda.
 \]
 It follows then that 
 \[
\forall t, \exists t' \geq t: \norm{V^{\pi^*}_\lambda-V^{\pi_{t'}}_\lambda}_\infty > \kappa,
 \]
 which implies $ V^{\pi{_t}}_{\lambda}(\mu)$ does not converge to $V^*$---a contradiction to our conclusion above. It therefore has to hold that $\pi_{t} \rightarrow \pi^{*}_{}$.
\end{proof}

The asymptotic guarantee of Vanilla Policy Gradient can be improved to non-asymptotic by using Natural Policy Gradient, as introduced by \cite{Kakade2001natural}. We restate the following result from \cite{Cen2022Fast}.

\begin{proposition}[Linear convergence of exact entropy-regularized NPG, \cite{Cen2022Fast}]
\label{prop:npgconv}
For any learning rate $0<\eta \leq(1-\gamma) / \tau$, the entropy-regularized NPG updates (18) satisfy
$$
\begin{aligned}
\left\|Q_\lambda^{\star}-Q_\lambda^{(t+1)}\right\|_{\infty} & \leq C_1 \gamma(1-\eta \lambda)^t \\
\left\|\log \pi_\lambda^{\star}-\log \pi^{(t+1)}\right\|_{\infty} & \leq 2 C_1 \lambda^{-1}(1-\eta \lambda)^t,
\end{aligned}
$$
for all $t \geq 0$, where
$$
C_1:=\left\|Q_\lambda^{\star}-Q_\lambda^{(0)}\right\|_{\infty}+2 \lambda\left(1-\frac{\eta \lambda}{1-\gamma}\right)\left\|\log \pi_\lambda^{\star}-\log \pi^{(0)}\right\|_{\infty} .
$$
\end{proposition}

\section{Implementation Details}
\label{appendix:implementation_details}

\subsection{Baseline Algorithms}

\subsubsection{Adaptive Model Design \texorpdfstring{\citep{chen2022adaptive}}{}}
\label{appendix:amd_algorithm}
As noted in \Cref{sec:numerical_experiments}, the Adaptive Model Design (AMD) algorithm \citep{chen2022adaptive} was proposed for the Regularized Markov Design (RMD) problem which is a special case of \prob. In particular, when $\mathbb{P}_\xi$ is a Dirichlet distribution \sprob\ reduces to the RMD problem. To account for this difference, we modify the AMD algorithm (Algorithm 2 in \citep{chen2022adaptive}) as described in \Cref{alg:modified_amd}. We denote the upper-level reward and value functions with the superscript $u$ in the algorithm.

\begin{algorithm}
\caption{(Modified) Adaptive Model Design}
\label{alg:modified_amd}
\begin{algorithmic}[H]
\STATE \textbf{Input:} Iterations $T$, Inner iterations: $K$, Learning rate $\alpha$, Regularization $\lambda$, gradient of the pre-learned model $\nabla_x \log P$, gradient of the reward function $\nabla_x r$
\STATE \textbf{Initialize} $x_0$, $Q_0$, $\nabla_x Q_0$, and  $\Tilde{Q}_0$
\FOR{$t = 0$ to $T-1$}
    \STATE $\xi\sim \mathbb{P}_\xi$
    \FOR{$k=0$ to $K-1$}
        \STATE $\pi_{x_t, \xi} \gets \exp (\lambda Q_k(s, \cdot))$
        \STATE Calculate $V_k, \nabla_{x_t} V_k, V^U_k, \nabla_{x_t} A_k, A_k^u, \Tilde{V}_k$
        \STATE $Q_{k+1} \gets \calT_{r, \gamma}(V_k)$
        \STATE $\nabla_{x_t} Q_{k+1} = \calT_{\nabla_{x_t}r, \gamma} (\nabla_{x_t} V_k + V_k \nabla_{x_t} \log P)$
        \STATE $Q_{k+1}^u = \calT_{r_u, \gamma_u}(V_k^u)$
        \STATE $\Tilde{Q}_{k+1} \gets \calT_{\nabla_{x_t}r^u + \lambda A^u_k \nabla_{x_t}A_k}(\Tilde{V}_k + V_k^u \nabla_{x_t} \log P)$
    \ENDFOR
    \STATE Set $\frac{\widehat{d F}}{dx} = \Tilde{V}_K$
    \STATE $x_{t+1}\gets x_t+\alpha \widehat{\frac{dF}{dx}}$
    \STATE Reinitialize $Q_0 \gets Q_K$, $\nabla_x Q_0 \gets \nabla_x Q_K$, and  $\Tilde{Q}_0 \gets \Tilde{Q}_K$
\ENDFOR
\STATE \textbf{Output:} Optimised parameter $x_T$
\end{algorithmic}
\end{algorithm}

\subsubsection{Zero-Order Algorithm}
\label{appendix:zero_order_algorithm}
\Cref{alg:zero_order} defines the zero-order gradient estimation algorithm described in \Cref{sec:numerical_experiments}. We parametrize the perturbation constant to decrease with the number of iterations such as $u_t = \frac{C}{t}$ where $C$ is a positive constant.
\begin{algorithm}
\caption{Zero-Order Algorithm}
\label{alg:zero_order}
\begin{algorithmic}[H]
\STATE \textbf{Input:} Iterations $T$, Learning rate $\alpha$, Regularization $\lambda$
\STATE \textbf{Initialize} $x_0$
\FOR{$t = 0$ to $T-1$}
    \STATE $\xi\sim \mathbb{P}_\xi$
    \STATE Sample $z \sim N(0, I_{d_x})$
    \STATE $\pi^o_{x_t,\xi}\gets \texttt{OraclePolicy}(x_t,\xi)$
    \STATE $\pi^o_{x_t+u_t*z,\xi}\gets \texttt{OraclePolicy}(x_t+z u_t,\xi)$
    \STATE Set $\frac{\widehat{d F}}{dx} = \frac{f(x+u_t*z, \pi^*_{x+u_t*z,\xi}, \xi) - f(x, \pi^*_{x,\xi}, \xi)}{u_t}z$

    \STATE $x_{t+1}\gets x_t+\alpha \widehat{\frac{dF}{dx}}$
\ENDFOR
\STATE \textbf{Output:} $\hat{x}_T \sim U(\{x_0,\dots,x_{T-1}\})$
\end{algorithmic}
\end{algorithm}

\subsection{Four Rooms}
\label{appendix:implementation_four_rooms}

\subsubsection{Implementation Details}
\label{appendix:four_rooms_implementation_details}
We parametrize the penalty function $\Tilde{r}$ as the softmax transformation of $x \in \RR^{d_s+1}$ where the $i$-th entry of $x$ corresponds to the $i$-th cell in the state space $\cS$ and the additional dimension $d_s+1$ is used to allocate the penalties not effective and also excluded from the penalty term received by the leader at the end of each episode. In particular,
\begin{equation*}
    \Tilde{r}(s, a) = -0.2 * \softmax(s; x),
\end{equation*}
where $\softmax(s; x)$ denotes the value of the softmax transformation of $x$ at the entry corresponding to the state $s$. Note that this parametrization explicitly restricts the maximum available budget for penalties to $-0.2$.

\subsubsection{Hyperparameters}
For the upper-level optimization problem, we use gradient norm clipping of $1.0$. The learning rate for each algorithm has been chosen as the best performing one from $[1.0, 0.5, 0.1, 0.05, 0.01]$ individually. Additionally, we tune the parameter $C$ for the Zero-order algorithm on the values $[0.1, 0.5, 1.0, 2.0, 5.0]$. For \algo, we sample $10,000$ environment steps for each gradient calculation.

\subsubsection{Additional Figures}
\label{appendix:four_rooms_additional_figures}
\begin{figure}[H]
  \centering
    \includegraphics[width=0.8\textwidth]{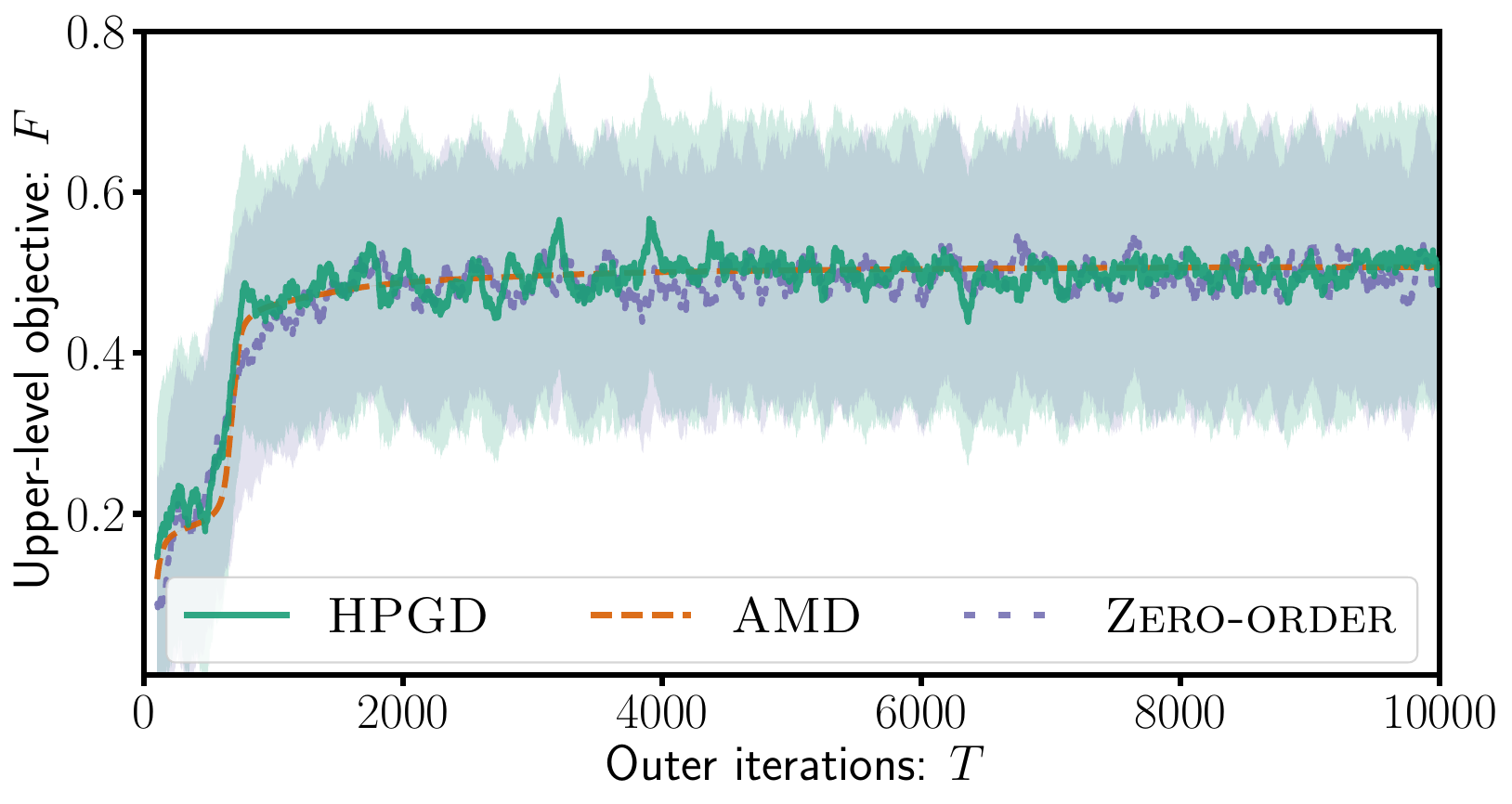}
    
  \caption{Upper-level objective values, $F$, over the number of outer iterations for hyperparameters $\lambda = 0.001$ and $\beta = 3.0$}
\end{figure}

\begin{figure}[H]
  \centering
    \includegraphics[width=0.8\textwidth]{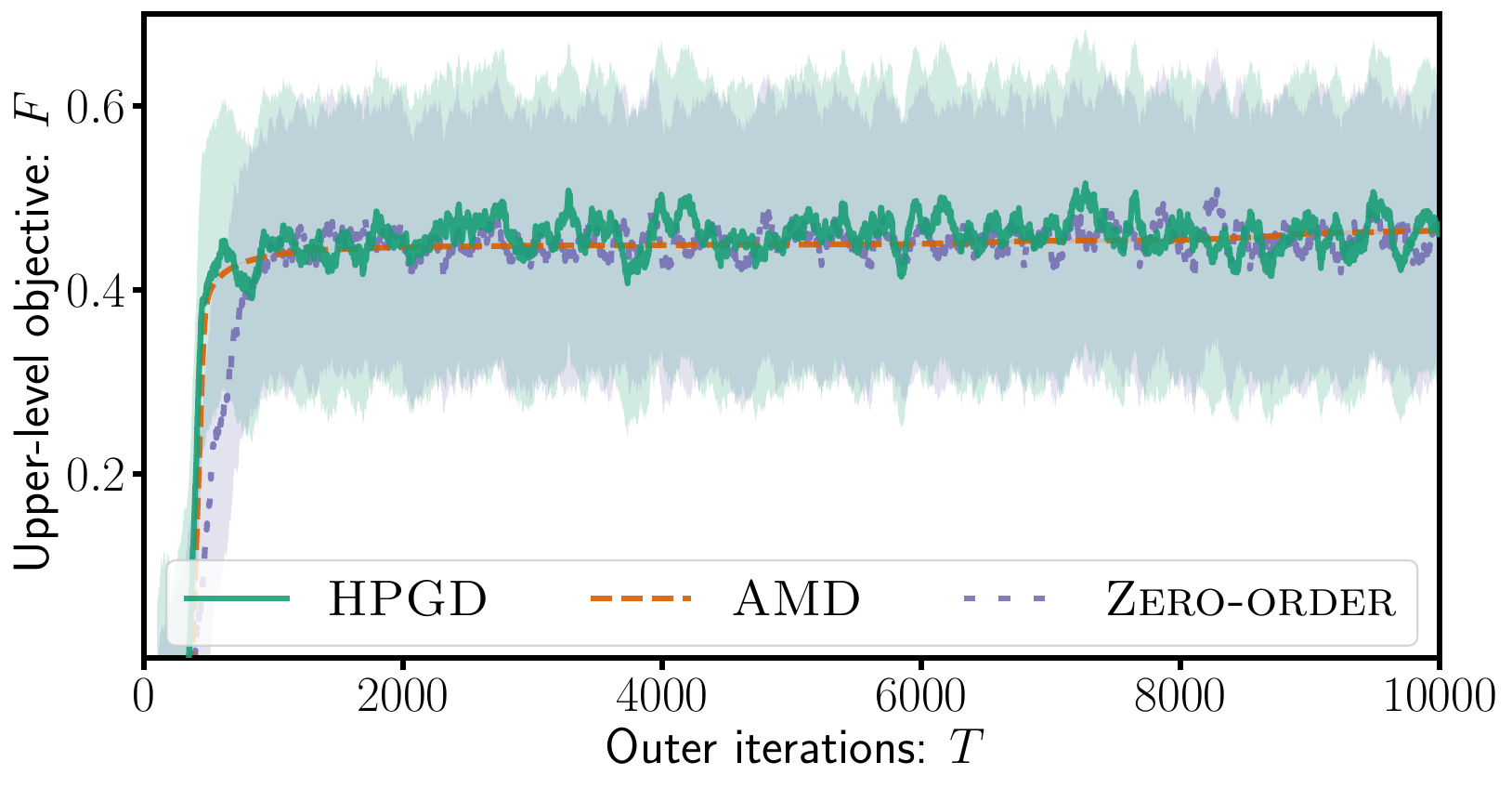}
    
  \caption{Upper-level objective values, $F$, over the number of outer iterations for hyperparameters $\lambda = 0.001$ and $\beta = 5.0$}
\end{figure}

\begin{figure}[H]
  \centering
    \includegraphics[width=0.8\textwidth]{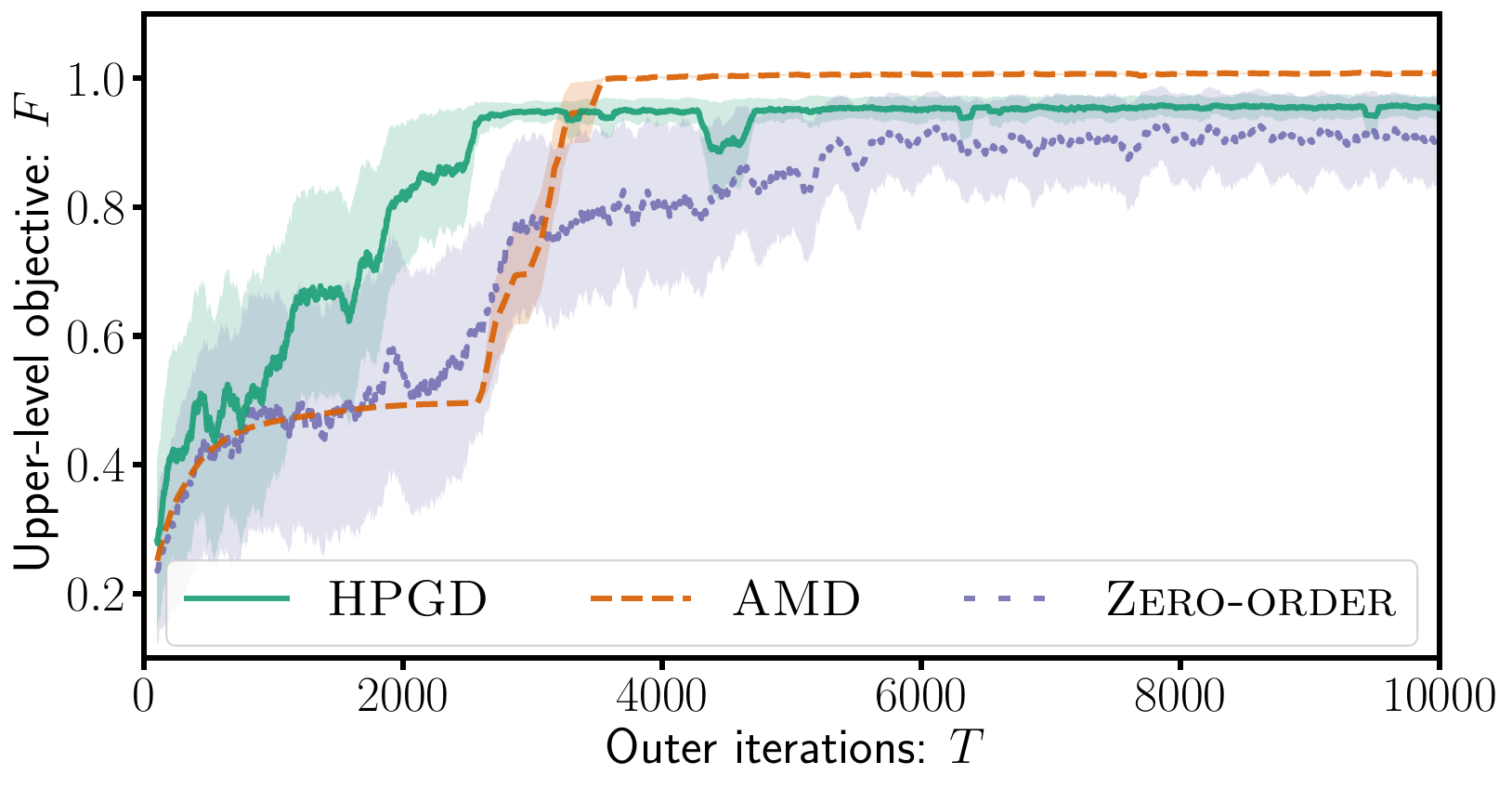}
    
  \caption{Upper-level objective values, $F$, over the number of outer iterations for hyperparameters $\lambda = 0.003$ and $\beta = 1.0$}
\end{figure}

\begin{figure}[H]
  \centering
    \includegraphics[width=0.8\textwidth]{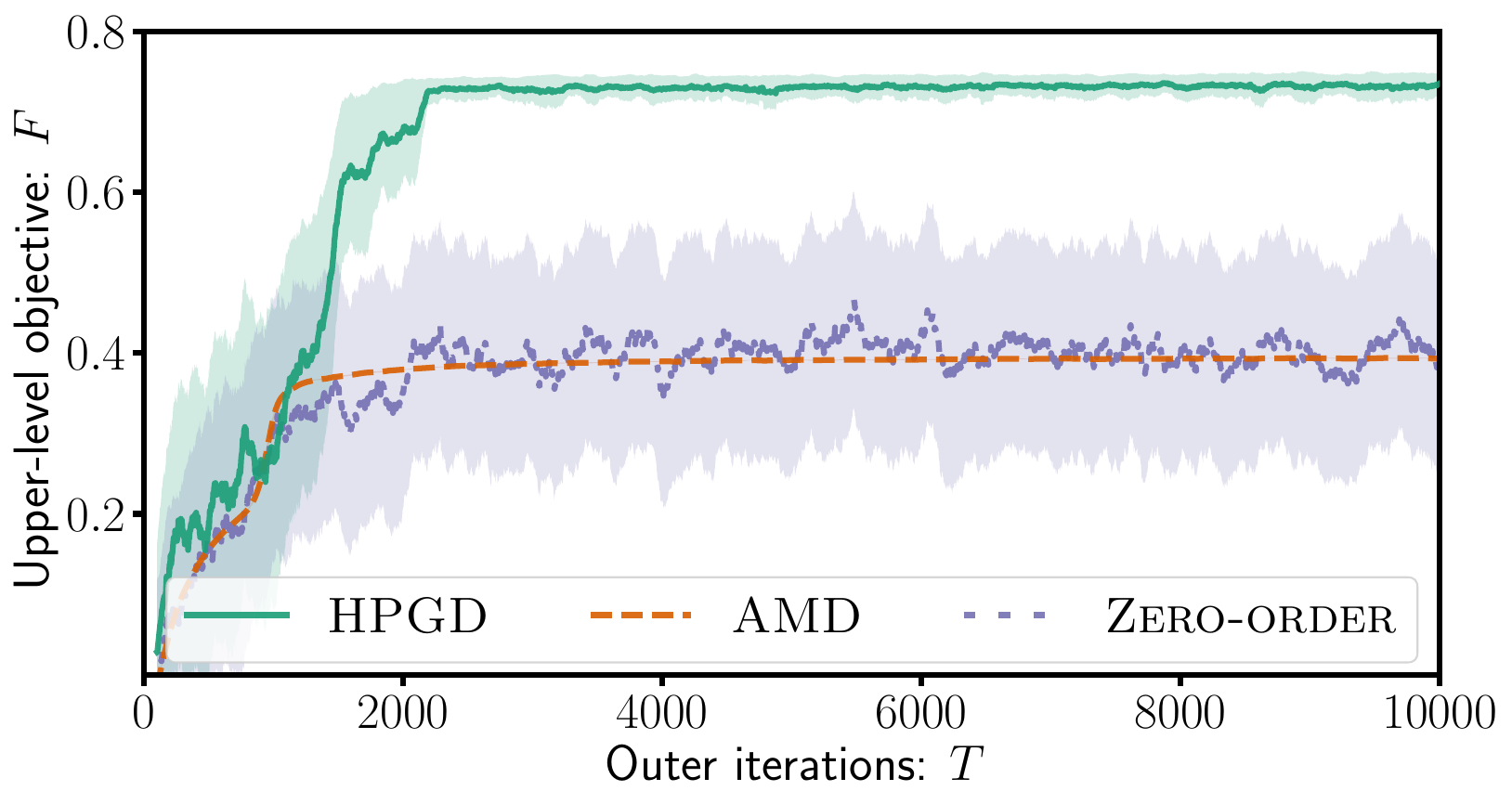}
    
  \caption{Upper-level objective values, $F$, over the number of outer iterations for hyperparameters $\lambda = 0.003$ and $\beta = 3.0$}
\end{figure}

\begin{figure}[H]
  \centering
    \includegraphics[width=0.8\textwidth]{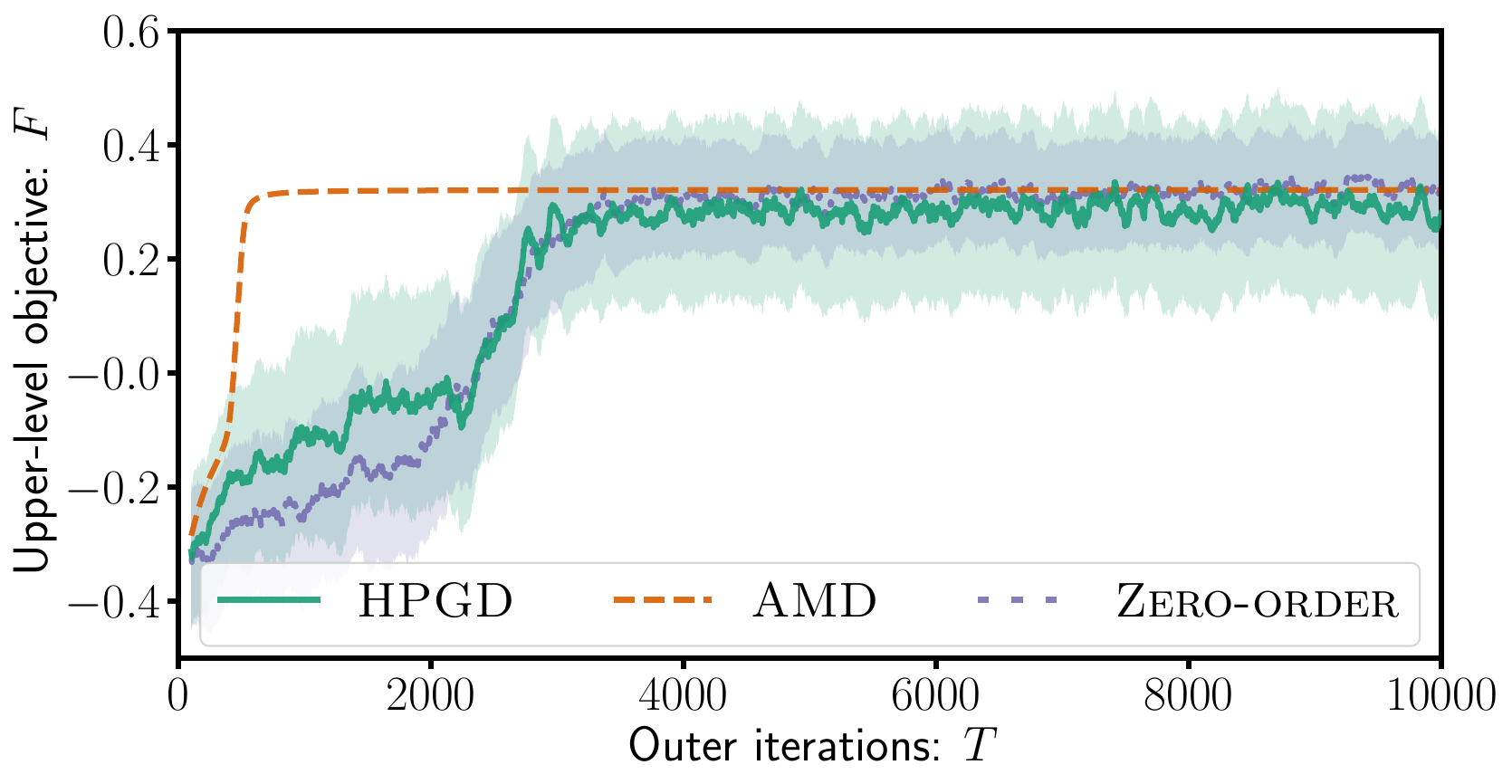}
    
  \caption{Upper-level objective values, $F$, over the number of outer iterations for hyperparameters $\lambda = 0.003$ and $\beta = 5.0$}
\end{figure}

\begin{figure}[H]
  \centering
    \includegraphics[width=0.8\textwidth]{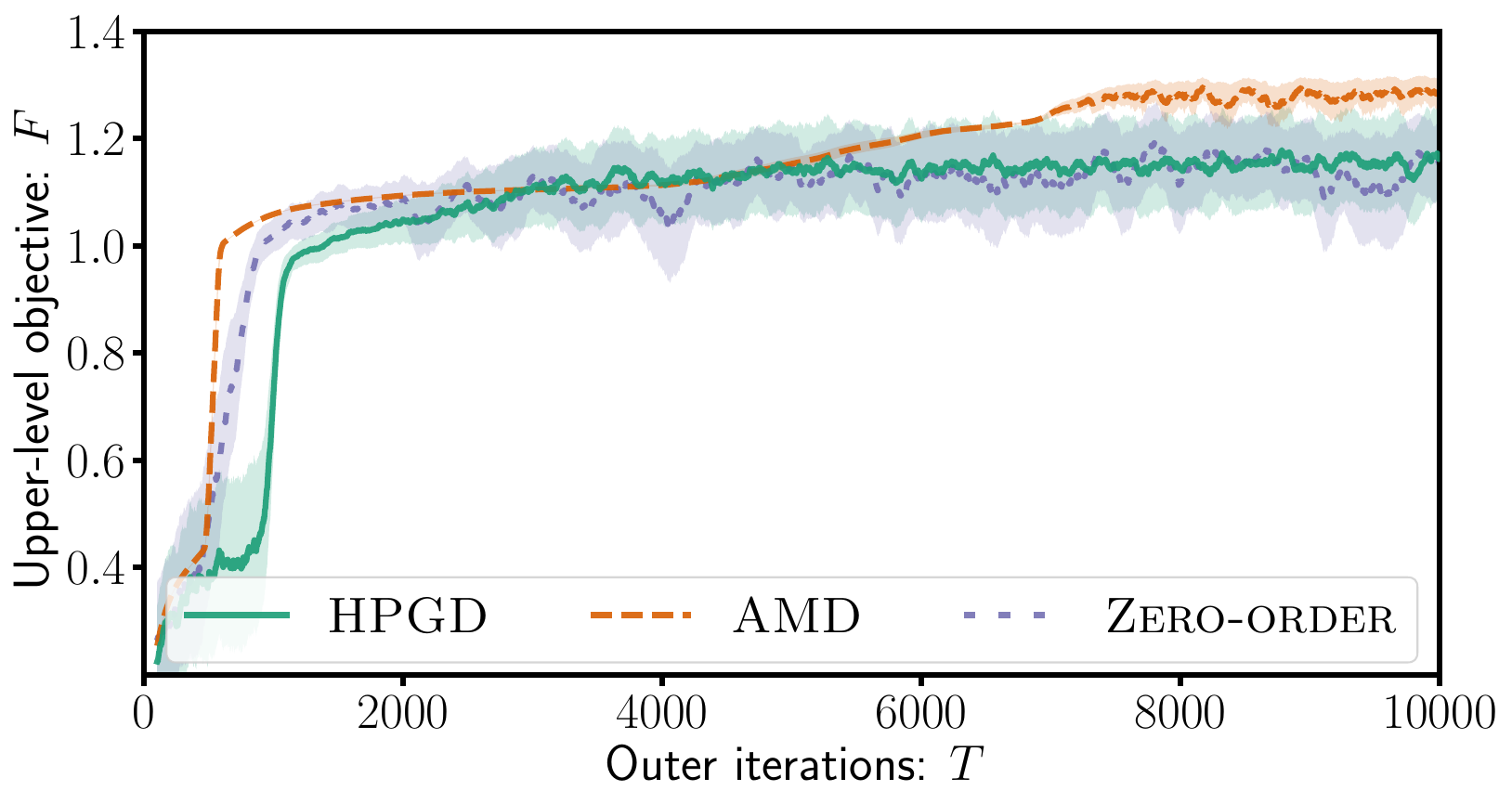}
    
  \caption{Upper-level objective values, $F$, over the number of outer iterations for hyperparameters $\lambda = 0.005$ and $\beta = 1.0$}
\end{figure}

\begin{figure}[H]
  \centering
    \includegraphics[width=0.8\textwidth]{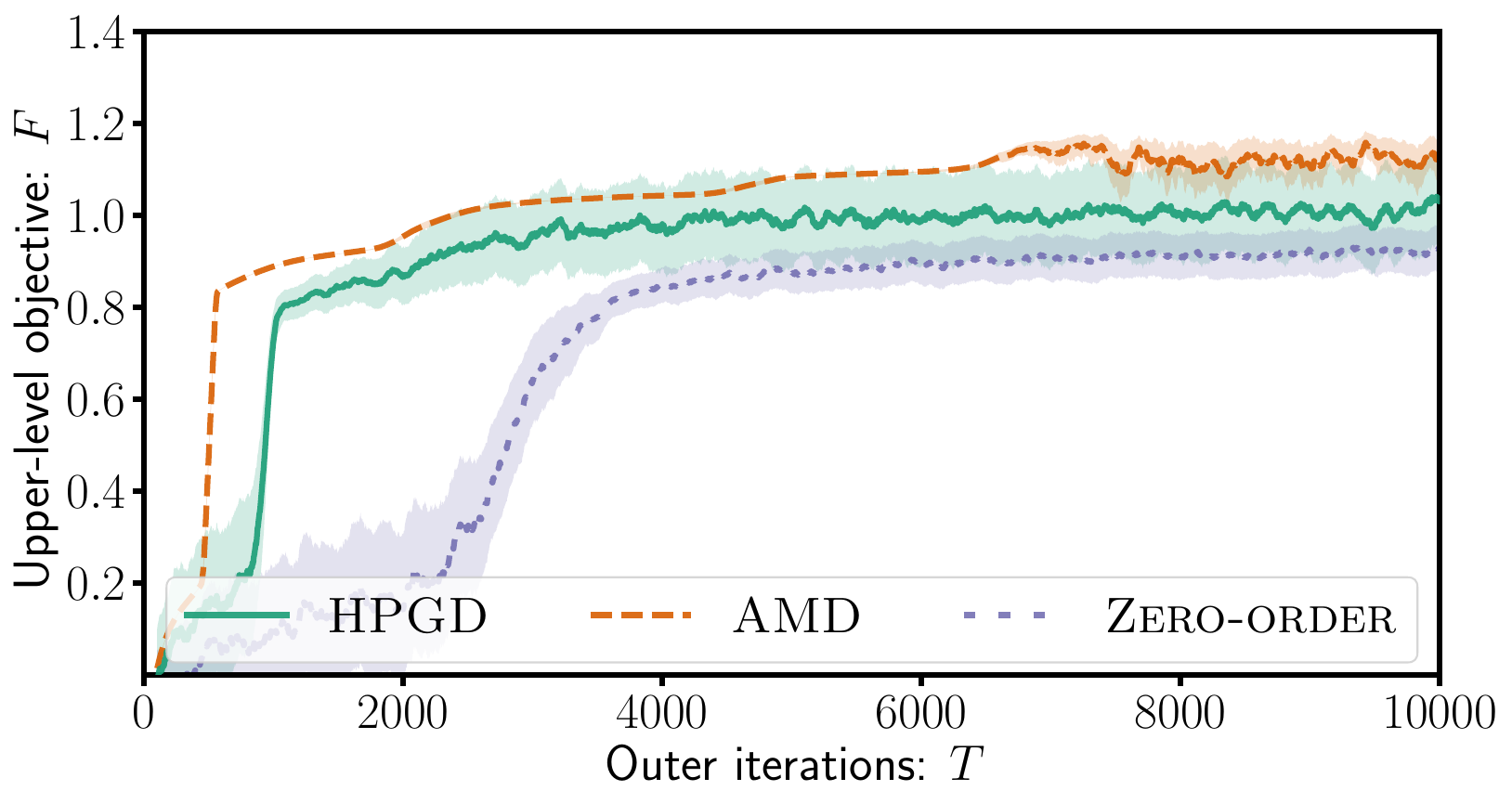}
    
  \caption{Upper-level objective values, $F$, over the number of outer iterations for hyperparameters $\lambda = 0.005$ and $\beta = 3.0$}
\end{figure}

\begin{figure}[H]
  \centering
    \includegraphics[width=0.8\textwidth]{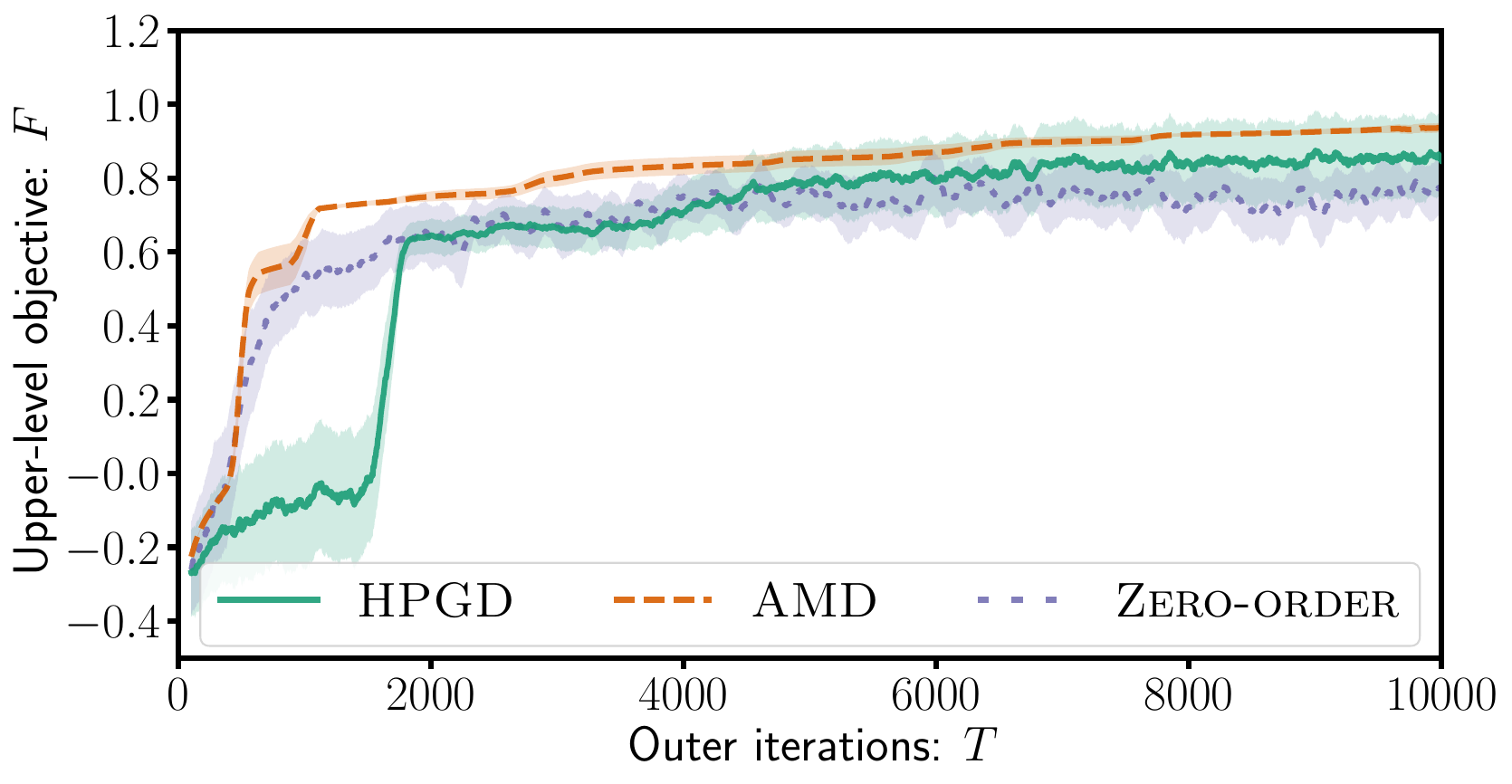}
    
  \caption{Upper-level objective values, $F$, over the number of outer iterations for hyperparameters $\lambda = 0.005$ and $\beta = 5.0$}
\end{figure}

\begin{figure}[H]
  \centering
    \includegraphics[width=0.9\textwidth]{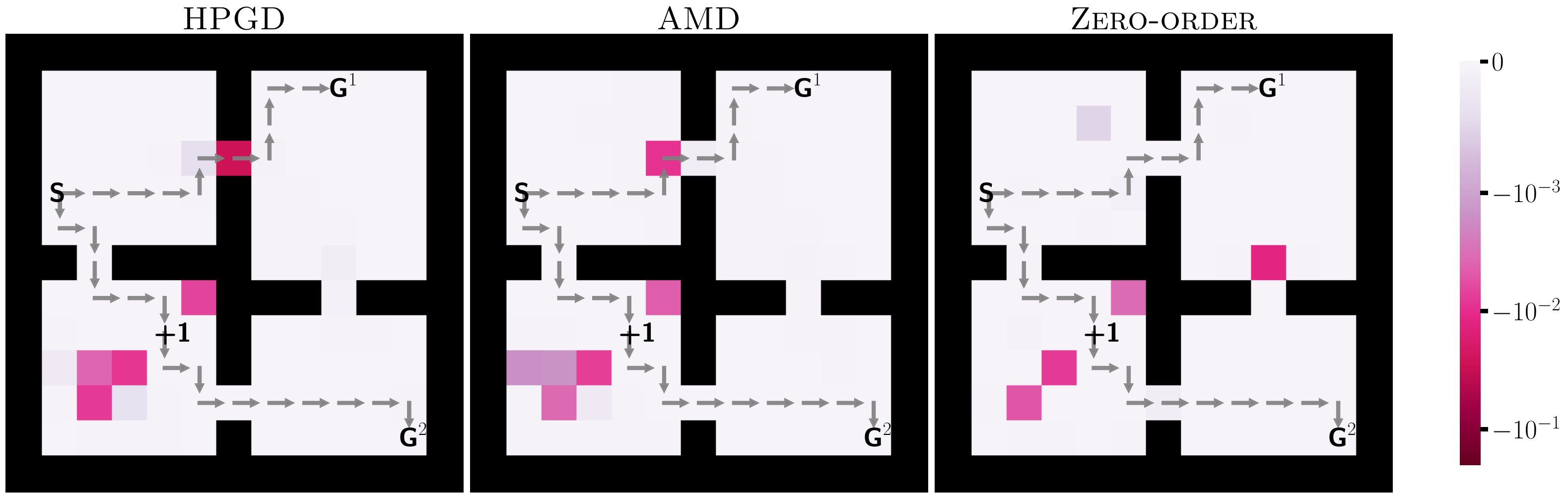}
    
  \caption{Reward penalties given to the lower-level agent in each state of the Four-Rooms problem optimized by the \salgo, AMD, and Zero-Order, respectively, for hyperparameters $\lambda = 0.001$ and $\beta = 3.0$}
\end{figure}

\begin{figure}[H]
  \centering
    \includegraphics[width=0.9\textwidth]{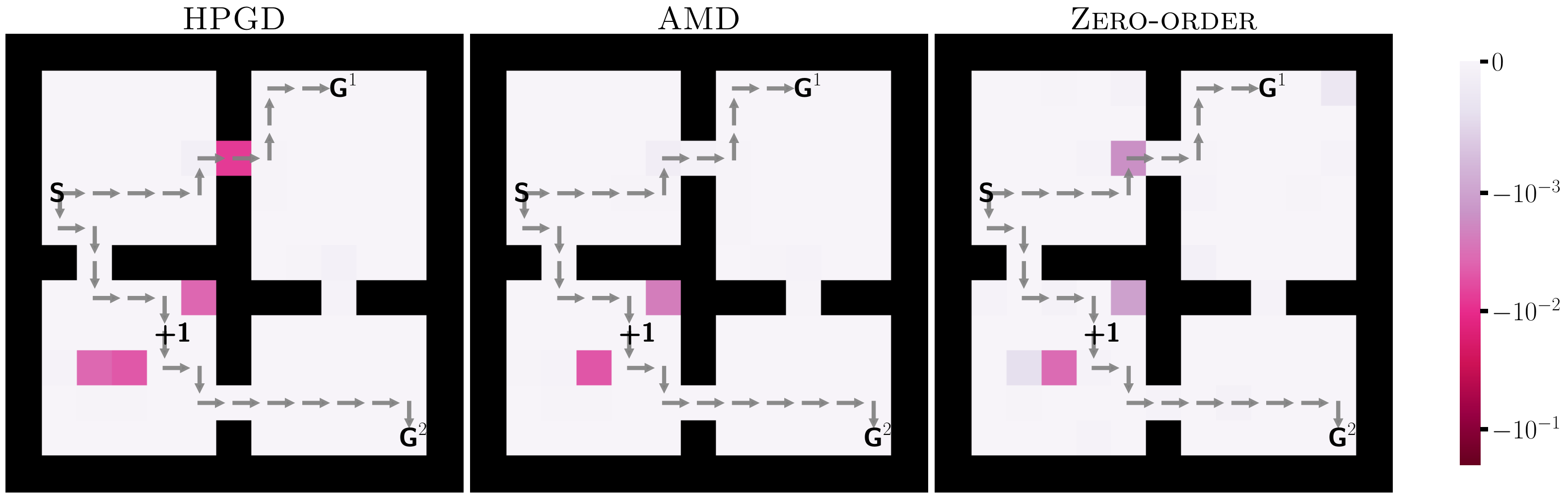}
    
  \caption{Reward penalties given to the lower-level agent in each state of the Four-Rooms problem optimized by the \salgo, AMD, and Zero-Order, respectively, for hyperparameters $\lambda = 0.001$ and $\beta = 5.0$}
\end{figure}

\begin{figure}[H]
  \centering
    \includegraphics[width=0.9\textwidth]{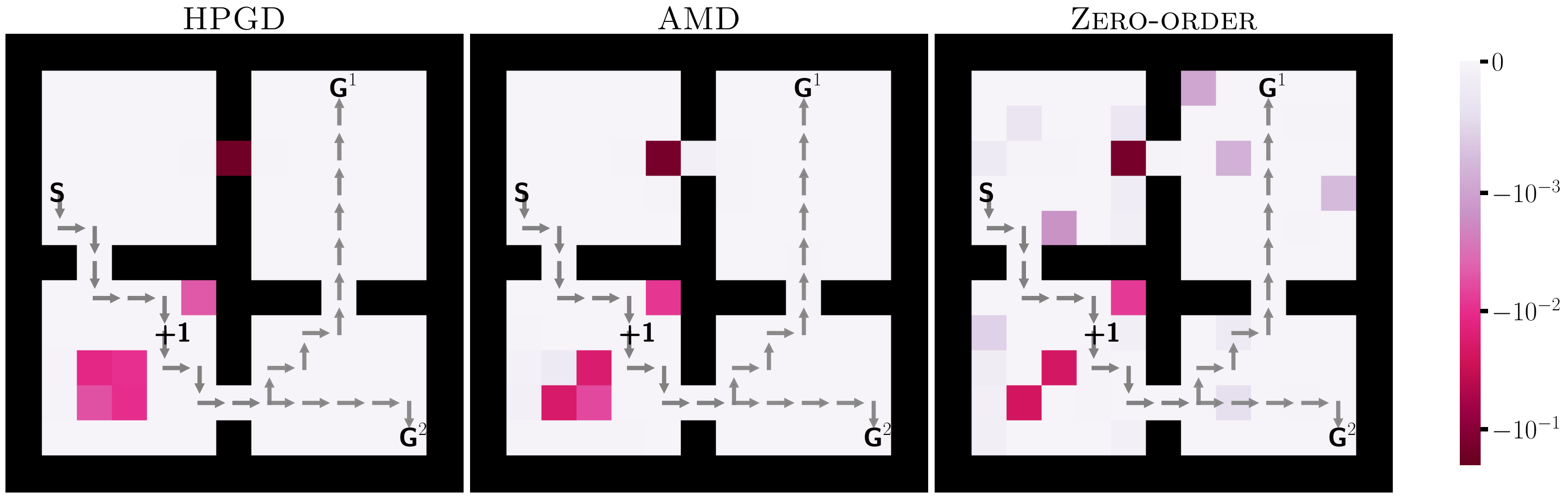}
    
  \caption{Reward penalties given to the lower-level agent in each state of the Four-Rooms problem optimized by the \salgo, AMD, and Zero-Order, respectively, for hyperparameters $\lambda = 0.003$ and $\beta = 1.0$}
\end{figure}

\begin{figure}[H]
  \centering
    \includegraphics[width=0.9\textwidth]{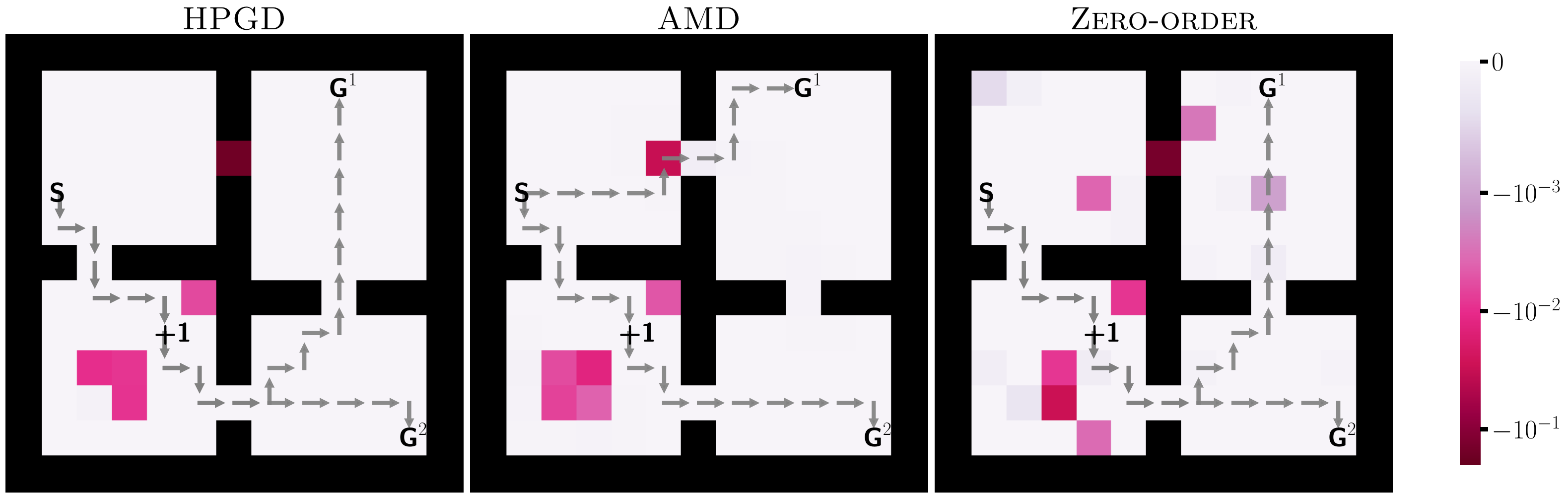}
    
  \caption{Reward penalties given to the lower-level agent in each state of the Four-Rooms problem optimized by the \salgo, AMD, and Zero-Order, respectively, for hyperparameters $\lambda = 0.003$ and $\beta = 3.0$}
\end{figure}

\begin{figure}[H]
  \centering
    \includegraphics[width=0.9\textwidth]{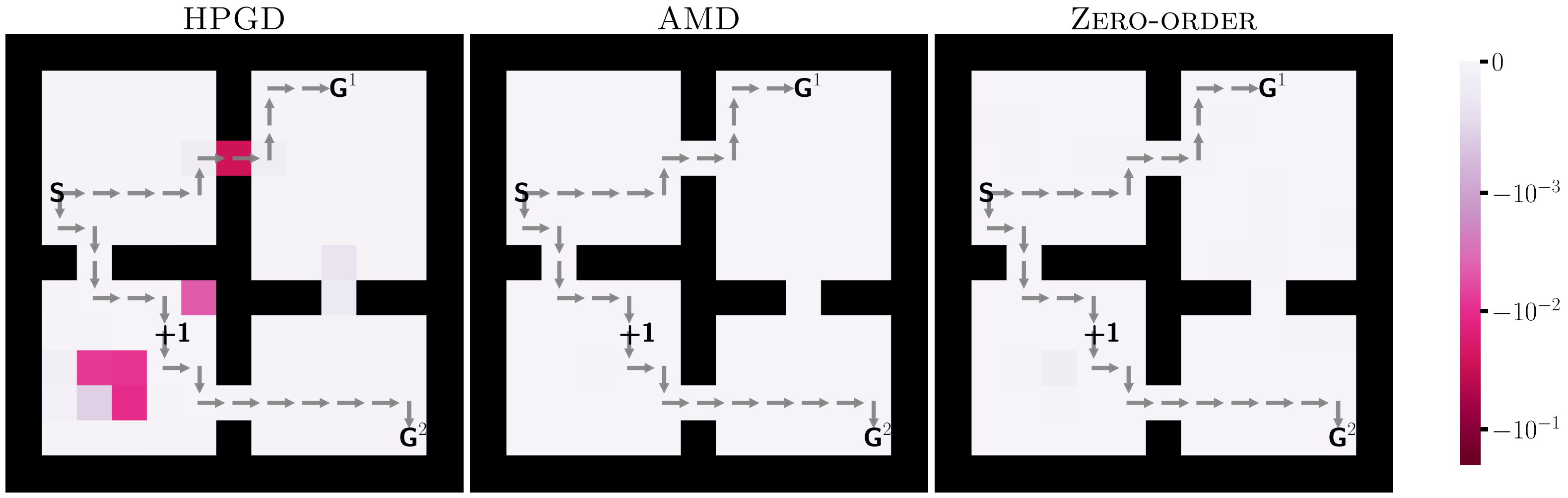}
    
  \caption{Reward penalties given to the lower-level agent in each state of the Four-Rooms problem optimized by the \salgo, AMD, and Zero-Order, respectively, for hyperparameters $\lambda = 0.003$ and $\beta = 5.0$}
\end{figure}

\begin{figure}[H]
  \centering
    \includegraphics[width=0.9\textwidth]{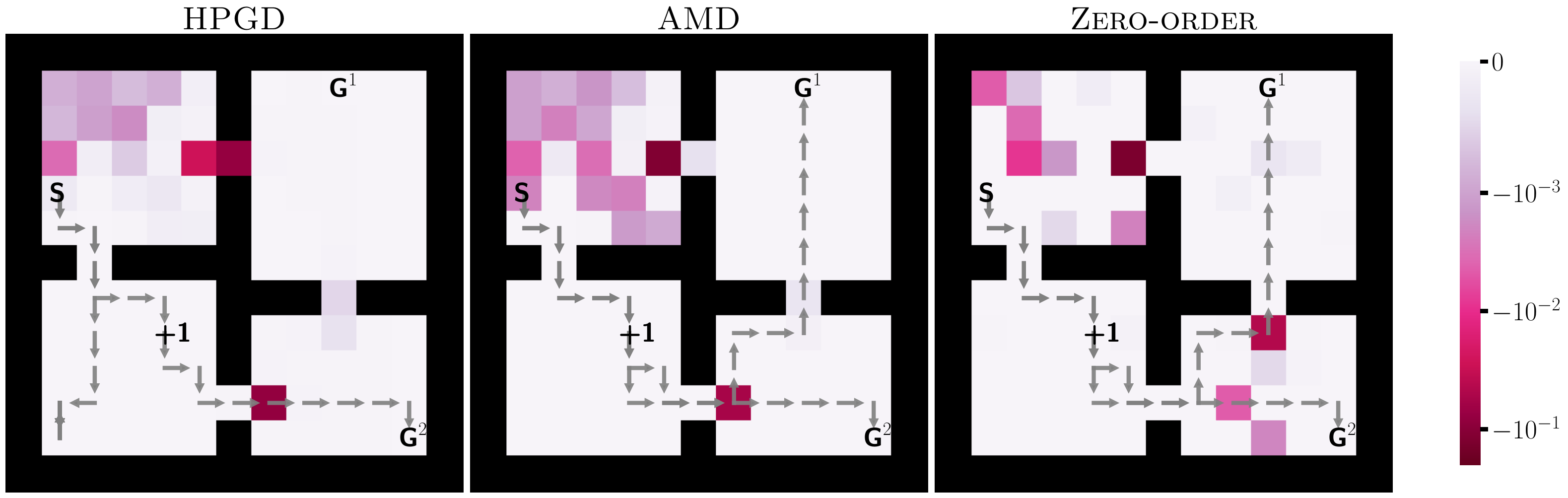}
    
  \caption{Reward penalties given to the lower-level agent in each state of the Four-Rooms problem optimized by the \salgo, AMD, and Zero-Order, respectively, for hyperparameters $\lambda = 0.005$ and $\beta = 1.0$}
\end{figure}

\begin{figure}[H]
  \centering
    \includegraphics[width=0.9\textwidth]{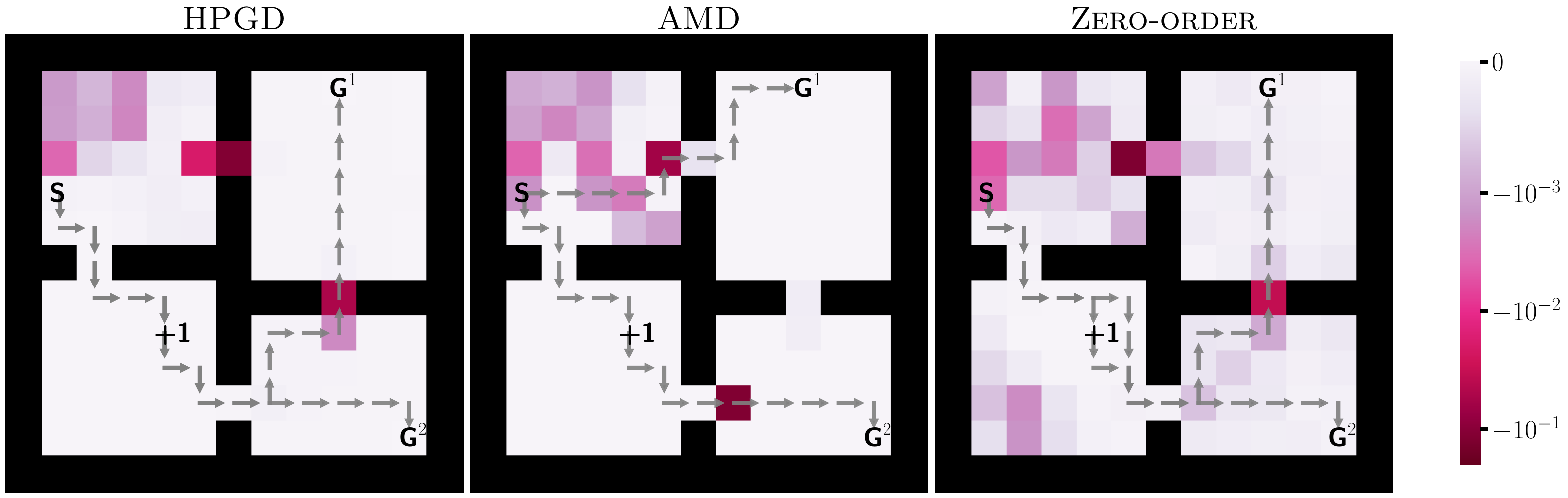}
    
  \caption{Reward penalties given to the lower-level agent in each state of the Four-Rooms problem optimized by the \salgo, AMD, and Zero-Order, respectively, for hyperparameters $\lambda = 0.005$ and $\beta = 3.0$}
\end{figure}

\begin{figure}[H]
  \centering
    \includegraphics[width=0.9\textwidth]{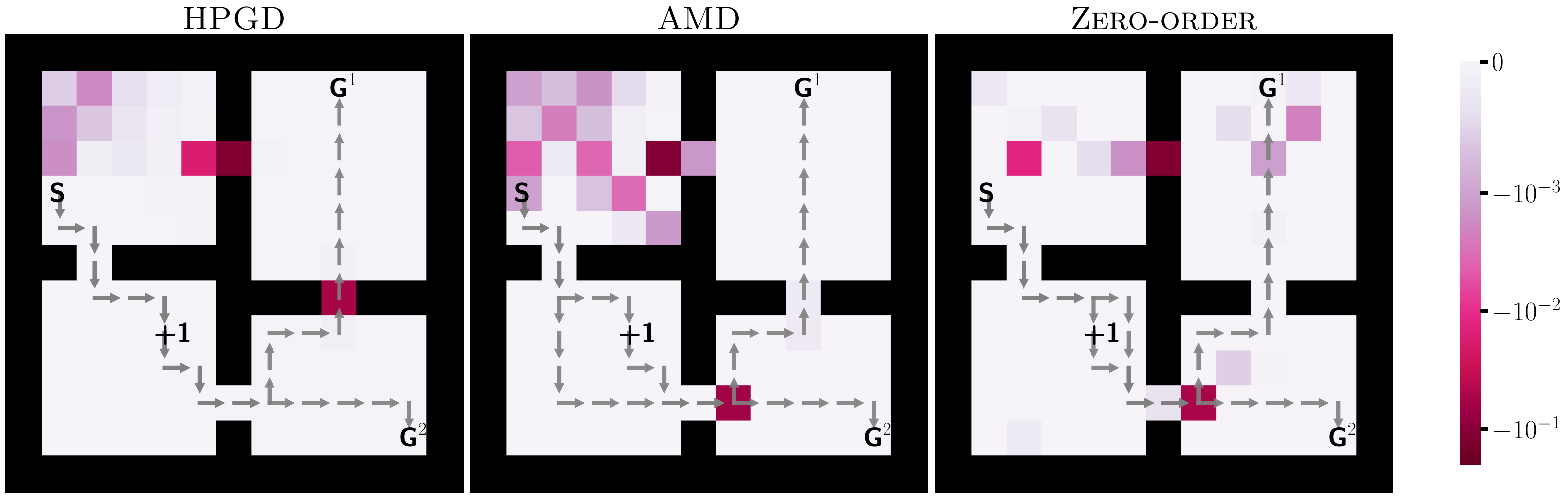}
    
  \caption{Reward penalties given to the lower-level agent in each state of the Four-Rooms problem optimized by the \salgo, AMD, and Zero-Order, respectively, for hyperparameters $\lambda = 0.005$ and $\beta = 5.0$}
\end{figure}

\subsection{Tax Design for Macroeconomic Models}
\subsubsection{Implementation Details}
\label{appendix:tax_design_implementation_details}
In our experiments, we use $\sigma(s) = \max(0, \log(s/20 +1))$ and $\omega(s) = \min(s, 0)$ and select the hyperparameters as $\theta = 0.1, \phi = 5.0, \varsigma = 5.0, w = 1.0$ and $\lambda = 0.3$.
We use $3$ products in the simulated economy with unit prices and assume two equal-sized socio-economic groups with preferences $\alpha = (0.6, 0.3, 0.1)$ and $(0.1, 0.7, 0.2)$.
Assets are initialized as $s_0 \sim N(0, \sqrt{2})$ and each episode is truncated after $200$ steps.
Both chosen consumption levels and hours worked are discretized with ranges $[0, 5]$ and $[0, 8]$ and cardinality $5$ and $10$, respectively.
All tax rates are initialized as $0.3$ and optimized over the continuous domain $[0.0, 2.0]$.

\subsubsection{Effects of lower-level regularization}
\Cref{fig:tax_03} and \Cref{fig:tax_01} shows the results for $\lambda=0.3$ and $\lambda=0.1$, respectively. Changing the regularization has small effects on the final results and only marginally change the speed of convergence, especially for the Zero-Order algorithm.

\begin{figure}
    \begin{subfigure}[T]{0.4\textwidth}
        \centering
        \includegraphics[width=\textwidth]{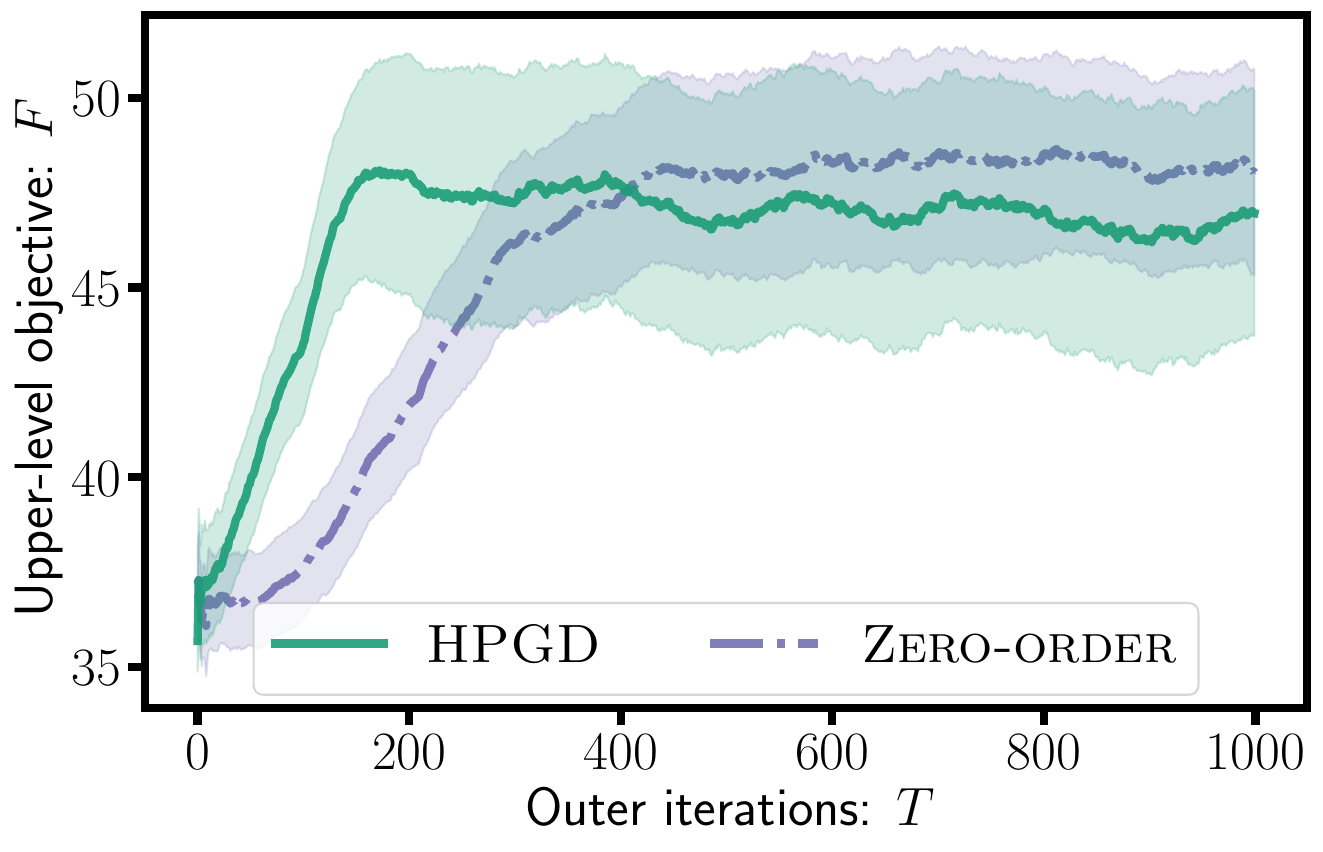}
        \caption{Performance on the Tax Design problem}
        \label{fig:tax_design_performance_03}
    \end{subfigure}
    \hspace{2pt}
    \begin{subfigure}[T]{0.6\textwidth}
        \centering
        \includegraphics[width=\textwidth]{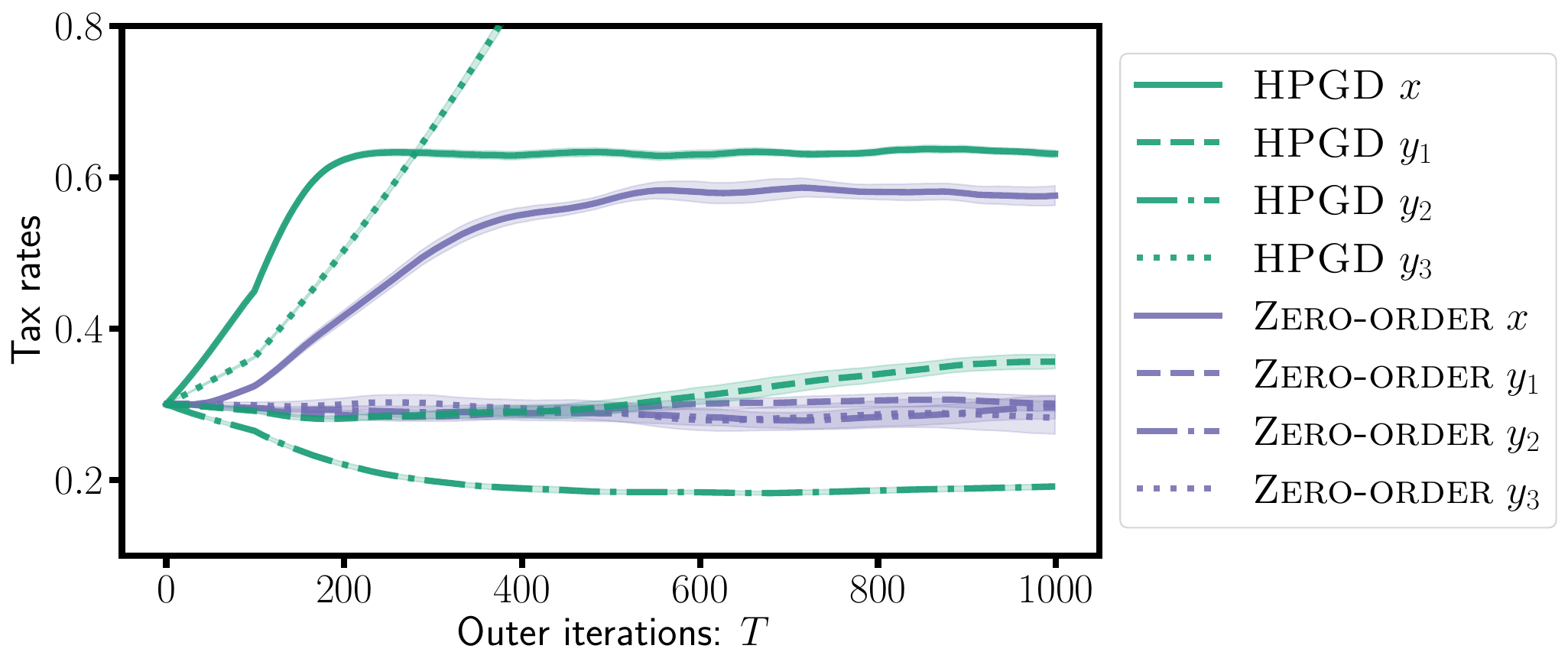}
        \caption{Tax rates over the outer iterations.}
        \label{fig:tax_design_tax_outcomes_03}
    \end{subfigure}
    \caption{Results with $\lambda = 0.3$}
    \label{fig:tax_03}
\end{figure}

\begin{figure}
    \begin{subfigure}[T]{0.4\textwidth}
        \centering
        \includegraphics[width=\textwidth]{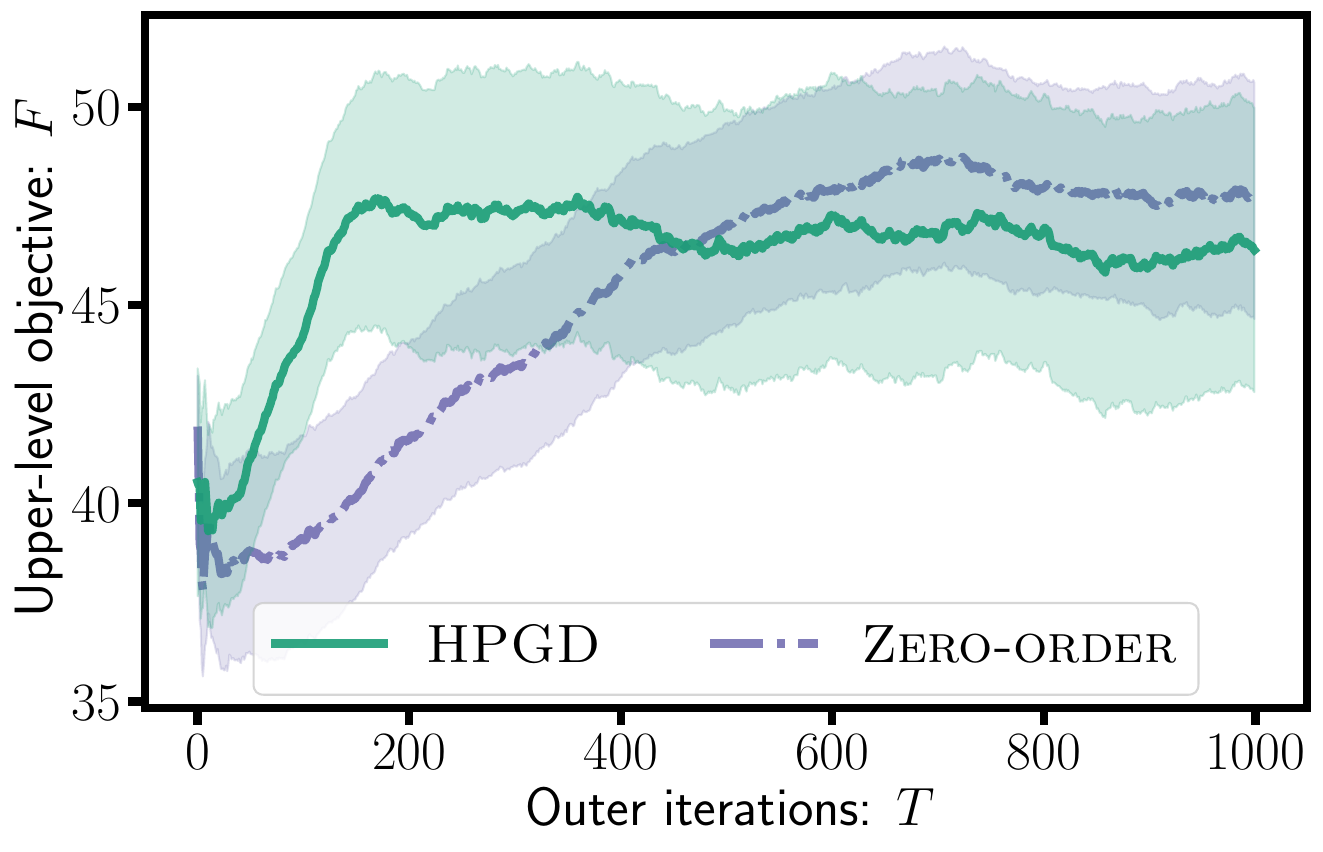}
        \caption{Performance on the Tax Design problem}
        \label{fig:tax_design_performance_01}
    \end{subfigure}
    \hspace{2pt}
    \begin{subfigure}[T]{0.6\textwidth}
        \centering
        \includegraphics[width=\textwidth]{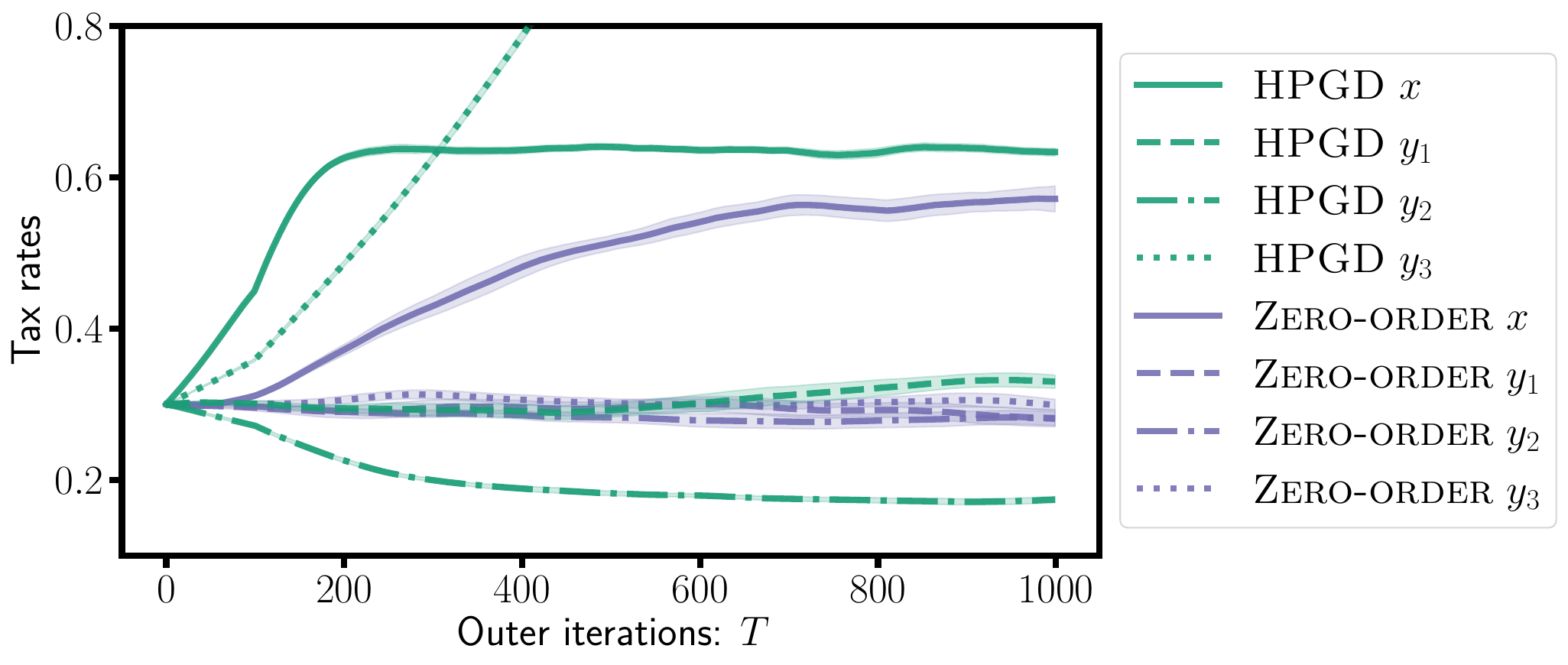}
        \caption{Tax rates over the outer iterations.}
        \label{fig:tax_design_tax_outcomes_01}
    \end{subfigure}
    \caption{Results with $\lambda = 0.1$}
    \label{fig:tax_01}
\end{figure}

\subsection{Computational Costs}
\label{appendix:computational_costs}
We ran our experiments on a shared cluster equipped with various NVIDIA GPUs and AMD  EPYC CPUs. Our default configuration for all experiments was a single GPU with 24 GB of memory, 16 CPU cores, and 4 GB of RAM per CPU core.
For all parameter configurations reported in \Cref{table:4rooms_parameter_comparison}, the total runtime of the experiments for \salgo, AMD, and Zero-Order were 17, 40, and 2 hours, respectively, totaling 59 hours.
Our total computational costs including the intermediate experiments are estimated to be 2-3 times more.

\end{document}